\documentclass[11pt]{amsart}
\usepackage{amssymb,latexsym,amsmath,amscd,amsthm,amsfonts, enumerate}
\usepackage{multirow}
\usepackage{color}
\usepackage[all]{xy}
\usepackage{etex}
\usepackage{caption}
\usepackage{soul}
\usepackage{float}
\usepackage{titletoc}
\usepackage[normalem]{ulem}
\usepackage{graphicx}

\usepackage{tikz,tikz-cd}
\usepackage{mathrsfs}
\usepackage[all]{xy}
\usepackage{tikz}
\usepackage{extarrows}
\usepackage{tikz-cd}
\usetikzlibrary{calc}
\usetikzlibrary{matrix,arrows,decorations.pathmorphing}
\usetikzlibrary{snakes}
\usetikzlibrary{shapes.geometric,positioning}
\usetikzlibrary{arrows,decorations.pathmorphing,decorations.pathreplacing}
\usetikzlibrary{positioning,shapes,shadows,arrows,snakes}

\usepackage[colorlinks=true,pagebackref,hyperindex]{hyperref}
\newcommand\myshade{85}
\colorlet{mylinkcolor}{violet}
\colorlet{mycitecolor}{red}
\colorlet{myurlcolor}{cyan}

\hypersetup{
  linkcolor  = mylinkcolor!\myshade!black,
  citecolor  = mycitecolor!\myshade!black,
  urlcolor   = myurlcolor!\myshade!black,
  colorlinks = true,
}

\usepackage{tabularx}
\newcolumntype{E}{>{\hsize=0.5cm \centering\arraybackslash}X}%
\newcolumntype{C}[1]{>{\hsize=#1\hsize \centering\arraybackslash}X}%

\numberwithin{equation}{section}
\newtheorem{theorem}{Theorem}[section]
\newtheorem{proposition}[theorem]{Proposition}
\newtheorem{proposition-definition}[theorem]{Proposition-Definition}

\newtheorem{corollary}[theorem]{Corollary}
\newtheorem{lemma}[theorem]{Lemma}

\theoremstyle{definition}
\newtheorem{remark}[theorem]{Remark}
\newtheorem{example}[theorem]{Example}
\newtheorem{definition}[theorem]{Definition}

\newcommand{\Hom}{\mathrm{Hom}}

\newcommand{\Ext}{\mathrm{Ext}}
\newcommand{\add}{\mathsf{add}}

\newcommand{\za}{\alpha}
\newcommand{\zb}{\beta}

\newcommand{\zD}{\Delta}

\newcommand{\zg}{\gamma}
\newcommand{\zG}{\Gamma}

\newcommand{\aaa}{{\bf{a}}}
\newcommand{\bbb}{{\bf{b}}}
\newcommand{\ccc}{{\bf{c}}}
\newcommand{\ddd}{{\bf{d}}}

\newcommand{\proj}{\operatorname{{\rm proj }}}
\newcommand{\pd}{\operatorname{{\rm pd }}}

\newcommand{\ma}{\operatorname{{\rm mod-A }}}
\newcommand{\maG}{\operatorname{{\rm mod-A_{\Gamma} }}}

\newcommand{\cals}{\mathcal{S}}

\newcommand{\calc}{\mathcal{C}}

\newcommand{\call}{\mathcal{L}}
\newcommand{\cald}{\mathcal{D}}
\newcommand{\calm}{\mathcal{M}}

\newcommand{\calp}{\mathcal{P}}

\newcommand{\calk}{\mathcal{K}}

\newcommand{\bbp}{\mathbb{P}}
\newcommand{\dba}{\cald^b(A)}
\newcommand{\ka}{\calk^b(A)}

\newcommand{\surf}{(\cals,\calm,\zD^*)}

\renewcommand{\P}{P^\bullet}

\definecolor{dark-green}{RGB}{14,150,2}
\definecolor{red}{RGB}{250,0,0}

\newcommand{\bpoint}{\circ}
\newcommand{\rpoint}{\color{red}{\bullet}}

\begin{document}

\title[Tilting-completion for gentle algebras]{Tilting-completion for gentle algebras}

\author{Wen Chang}
\address{School of Mathematics and Statistics, Shaanxi Normal University, Xi'an 710062, China}
\email{changwen161@163.com}


\thanks{}

%


\begin{abstract}
It is demonstrated that any almost-tilting module over a gentle algebra of rank $n$ is indeed partial-tilting, meaning it can be completed as a tilting module. Furthermore, such a module has at most $2n$ possible complements, thereby confirming a (modified) conjecture of Happel for the case of gentle algebras. Additionally, for any $n\geq 3$ and $1\leq m \leq n-2$, there always exists a (connected) gentle algebra with rank $n$ and a pre-tilting module of rank $m$ which is not partial-tilting. The tool we use is the surface model associated with the module category of a gentle algebra. 
\end{abstract}

\maketitle
\setcounter{tocdepth}{2} 

\tableofcontents

\section*{Introduction}\label{Introductions}
The tilting theory was originally introduced and developed within the representation theory of finite-dimensional algebras in the early 1980s. It is now regarded as an essential tool in various branches of mathematics, including finite and algebraic group theory, commutative and non-commutative algebraic geometry, and algebraic topology.
In \cite{BB79}, Brenner and Butler initiated tilting theory by generalizing reflection functors developed by Bernstein-Gelfand-Ponomarev \cite{BGP73} and Auslander-Platzeck-Reiten \cite{APR79} into so-called tilting functors. 
Happel and Ringel \cite{HR82} further refined the theory by relaxing and simplifying the axioms defining a tilting module, and completed the picture by considering additional functors.
This was further generalized by Miyashita \cite{M86} and Happel \cite{H87}, who introduced the concept of a general tilting module with finite projective dimension. Happel demonstrated that a tilting module induces a derived equivalence between the original algebra and the endomorphism algebra of the tilting module. Subsequently, Rickard \cite{R89} showed that tilting complexes, as a generalization of tilting modules, determine when two algebras have equivalent derived categories. 

In this paper, we consider tilting modules in the sense of Miyashita and Happel.
Let $A$ be a finite-dimensional algebra and $\ma$ be the category of finite-dimensional right modules over $A$. A module $T\in \ma$ is called a \emph{tilting module} if the following conditions hold:
\begin{enumerate}
	\item[(T1)] the projective dimension of $T$ is finite, that is $\pd T < +\infty$;
	\item[(T2)] $T$ is self-orthogonal, that is $\Ext^i_A(T,T)=0$ for any $i>0$;
	\item[(T3)] there exists $d\ge0$ and an exact sequence $0\to A_A\to T_0\to\cdots\to T_d\to0$ with $T_i\in\add T$ for all $0\leq i \leq d$, where $\add T$ is the category of direct sums of direct summands of $T$.
\end{enumerate}

If only the first two conditions are satisfied, we refer to $T$ as a \emph{pre-tilting module}. A pre-tilting module is called a \emph{partial-tilting module} if it can be completed as a tilting module. We assume throughout that any pre-tilting, tilting, or partial-tilting module is multiplicity-free; in other words, any two different direct summands are not isomorphic. A pre-tilting module is called \emph{almost-tilting}, if it has $n-1$ indecomposable direct summands, where $n=|A|$ denotes the {\em rank} of $A$, that is the number of non-isomorphic indecomposable projective modules over $A$.

We note that the terminology introduced above has varying interpretations across different references. For example, a pre-tilting module is referred to as a partial generalized tilting module in \cite{RS89} and as a partial-tilting module in \cite{CHU94}. In this paper, we adopt the more recent conventions, such as the definition of a partial-tilting module that aligns with that in \cite{U16}.
It is important to clarify that in our terminology, an almost-tilting module is not necessarily partial-tilting. Indeed an almost-tilting module that qualifies as partial-tilting is called an almost complete tilting module in \cite{CHU94}. To minimize ambiguity, we refer to a partial-tilting module with $n-1$ indecomposable direct summands a {\em maximal partial-tilting module}.

For any $1\leq m \leq n$, it is natural to ask the following question, see \cite{B81,HR82}:

{\bf(Cm): Does the conditions (T1), (T2), and $|M|=m$ imply that $M$ can be completed as a tilting module?}
 
In particular, $\bf{(Cn)}$ is equivalent to say that $(T3)$ can be replaced by the condition $|M|=n$.
We briefly recall the history of these tilting-completion questions. For the classical tilting module, that is, when $\pd T\leq 1$ and $d=1$ in $(T3)$, any pre-tilting module can be completed as a tilting module, which is called Bongartz's key lemma \cite{B81}. Thus the answer of $\bf{(Cm)}$ is positive for any $1\leq m \leq n$. 

For a general tilting module as defined above, Rickard and Schofield posted the question {\bf (Cm)} in \cite{RS89}, which has been proven true for any representation-finite algebras.
However, a counterexample exists for the representation-infinite case, which is a special case of the algebra given in Figure \ref{fig:not partial2} when $n=3$. 
In the proof of the result mentioned above in \cite{RS89}, the subcategory $$\calc=\{X\in \ma, \Ext^i_A(X,M)=0, \forall i>0, \pd M <\infty\}$$
serves a critical role, indeed \cite{CHU94} establishes that for a pre-tilting module $M$, if $\calc$ is contravariantly finite, then $M$ is a partial-tilting module. 
Notably, for the counterexample provided in \cite{RS89}, $n=3$ and $m=1$. This gives a negative answer to the question ${\bf (Cn-2)}$.
To the best of the author's knowledge, the question ${\bf (Cn-1)}$, which asks whether an almost-tilting module is also partial-tilting, remains open for general algebras.
It is also worth mentioning that if one considers more general infinitely generated modules, then any (finite-dimensional) pre-tilting module is partial-tilting, that is, it can be completed as a (possibly infinitely dimensional) tilting module, see \cite{CT95,AC02}.

Let $M$ be a partial-tilting module, and let $N$ be an $A$-module such that $M\oplus N$ is a tilting module and the intersection $\add M\cap \add N=0$. Then we call $N$ a {\em complement} to $M$.
For a maximal partial-tilting module $M$, it is natural to consider the number of its (non-isomorphic) complements.
If the algebra $A$ is hereditary, then there are at most two complements to a maximal partial-tilting module \cite{RS90,U90}, with exactly two existing if and only if the maximal partial-tilting module is sincere, see \cite{HU89}.
For the classical tilting module over a general finite dimension algebra, a maximal partial-tilting module has at most two complements, and there are exactly two if and only if the maximal partial-tilting module is faithful, see \cite{RS91,H92}. 
For the general case — namely, for general tilting modules over arbitrary finite-dimensional algebras — \cite{CHU94} proves that if a maximal partial-tilting module $M$ is not faithful and $\calc$ is contravariantly finite, then $M$ has a unique indecomposable complement. Conversely, if $M$ is faithful, it may have multiple complements, even if $\calc$ is contravariantly finite.

It is interesting that the problem of whether a maximal partial-tilting module admits only a finite number of complements turned out to be strongly connected with some well-known homological
conjectures for finite dimensional algebras. For instance, \cite{H95} demonstrates that if the Finitistic Dimension Conjecture is satisfied for an algebra $A$, then any maximal partial-tilting module in $\ma$ admits only a finite number of complements. At the same time, an algebra
$A$ satisfies the Generalized Nakayama Conjecture if and only if any projective
maximal partial-tilting module admits only a finite number of complements, see \cite{BS96,HU96}.
We refer to \cite[Section 13.6, 13.7]{S23} for more information on homological conjectures related to self-orthogonal modules.
In particular, the following complement conjectures are made:

{\bf (A)} (\cite{BS96,HU96}) Any maximal partial-tilting module admits a finite number of complements.

{\bf (B)} (Happel) If a maximal partial-tilting module M admits a finite number of
complements, then the number of complements is bounded by $2n$.

It is clear that conjecture ${\bf (A)}$ holds for any representation-finite algebra.
Mantese proved in \cite{M05} that the conjecture ${\bf (A)}$ holds for any projective maximal partial-tilting module over a monomial algebra.
The original conjecture {\bf(B)} made by Happel predicts that the number of complements is bounded by $2n-1$, and it is true for some special algebras, for example, left serial algebras \cite[Corollary 3.6]{M05}.
However, in \cite{M05}, Mantese constructed an algebra with rank two which has four complements, which suggests that
the supposed bound in conjecture ${\bf (B)}$ should be changed from $2n-1$ to $2n$, see details in \cite[Section 2]{M05}.

The main result of this paper says that for any gentle algebra, the answer of ${\bf (Cn-1)}$ is affirmative, and both the conjectures {\bf (A)} and {\bf (B)} are valid:
 
\begin{theorem}[Theorem \ref{thm:almost-tilting is partial-tilting}]
 Any almost-tilting module over a gentle algebra is partial-tilting, that is, it can be completed as a tilting module. Additionally, there are at most $2n$ complements.
 \end{theorem}

However, Theorem \ref{thm:pre-tilting not partial-tilting} demonstrates that for any $n\geq 3$ and $1\leq m \leq n-2$, there always exists a (connected) gentle algebra of rank $n$ and a pre-tilting module of rank $m$ that does not qualify partial-tilting.
We note that the complement problem for silting objects in the derived category of (potentially graded) gentle algebras has been explored in \cite{CJS22, JSW23, LZ23}. Building on the existence results for silting objects established in \cite{CJS22}, \cite{JSW23} presents a series of examples where there exists a pre-silting object that is not partial silting. Additionally, \cite{LZ23} provides a counter-example by offering an explicit classification of the pre-silting objects. On the other hand, it is proved in \cite{JSW23} that for a homologically smooth and proper graded gentle algebra, any almost complete pre-silting object is partial silting.

As mentioned above, this paper focuses on a special class of finite dimensional algebras, known as gentle algebras, which are classical objects in the representation theory of associative algebras. 
Gentle algebras were introduced in the 1980s  as a generalization of iterated tilted algebras of type $A_n$ \cite{AH81}, and affine type $\widetilde A_n$ \cite{AS87}. Since their inception, gentle algebras have been a constant object of study, and much is known about their representation. For example, the indecomposable objects and the Auslander-Reiten sequences in their module category are described in \cite{WW85,AS87,BR87}.

In recent years, topological and geometric methods have been widely applied in the study of the representation theory of finite dimensional associative algebras.
For a gentle algebra, geometric models for its module category and derived category are established using partial triangulations (or dissections) of the surfaces, as described in \cite{BC21} and  \cite{OPS18} respectively. 
These models represent arcs as indecomposable objects, intersections as morphisms, and rotations of arcs as Auslander-Reiten translations, allowing the combinatorics and topology of the surface to address problems in the representation theory of gentle algebras. This paper contributes an additional example to this growing area of applications. 

There is extensive literature on surface models for various categories associated with gentle algebras, including \cite{APS23,BC21,HKK17,LP20,OPS18,PPP19,PPP21}.
For our purpose, we will closely follow the constructions and the notations used in \cite{C23}, where the author refines the geometric model for the module category of a gentle algebra
given in \cite{BC21}, and then integrates it with the geometric model of the derived category
given in \cite{OPS18}, in the sense that each so-called zigzag curve on the surface represents
an indecomposable module as well as the minimal projective resolution of this module.
Notably, extensions of modules can thus be easily visualized on the marked surface, a feature we will frequently use in studying pre-tilting modules over gentle algebras.

The idea we use is to reduce the problem to several base cases: disks, once-punctured disks, and annulus. We construct a new marked surface and a new simple coordinate by cutting the surface along a simple zigzag $\bpoint$-arc. 
We mention that The technique of cutting surfaces has emerged as a powerful methodology for analyzing derived categories of gentle algebras. For example, it is proved that cutting surfaces give rise to silting reductions in \cite{CS23}, and recollements and semi-orthogonal decompositions in \cite{CJS22} and \cite{KS22} respectively. It is also used to reduce the action of the braid group on the full exceptional sequences in \cite{CS23b}, with \cite{CHS25} providing an concrete example of non-transitive action. 

In this paper, we explicitly construct the simple coordinate obtained by cutting the surface along any simple zigzag $\bpoint$-arc. 
This construction yields a new gentle algebra from the original one, which is used in \cite{C25} to explain maximal rigid modules over a gentle algebra as so-called admissible $5$-partial triangulations on the associated surface.

The paper is structured as follows: Section 1 lays out the foundational concepts, covering the module category over gentle algebras and its corresponding surface model. Section 2 centers on tilting modules over gentle algebras, beginning with a geometric interpretation and introducing a surface-cutting reduction method. The main result, proven in subsection 2.3, establishes that any almost-tilting module over a gentle algebra is partial-tilting, with at most $2n$ complements. Subsection 2.4 provides examples of pre-tilting modules that are not partial-tilting. The Appendix includes details on various cutting surfaces and the construction of algebras derived through surface cuts.

\section{Acknowledgments}
This paper is supported by  the NSF of China (Grant No. 12271321) and Fundamental Research Funds for the Central Universities (No. GK202403003).
The author would like to express gratitude to Xiaojin Zhang and Zhiwei Li for organizing the conference in May 2024. The two talks delivered by Bin Zhu and Yu Zhou during this conference directly inspired the initiation of this work.
He would also like to thank Steffen Koenig for insightful discussions on tilting completion. Additionally, he is grateful to Sibylle Schroll for inviting him to visit the University of Cologne from September to November 2024. Her financial support and hospitality created an ideal research environment, and portions of this manuscript were completed during this period.

\section{Preliminaries}\label{Preliminaries}

In this paper, a quiver will be denoted by $Q=(Q_0,Q_1)$, where $Q_0$ is the set of vertices and $Q_1$ is the set of arrows. The numbers of the vertices and the arrows of $Q$ is $|Q_0|$ and $|Q_1|$ respectively. For an arrow $\aaa$, $s(\aaa)$ is the source and $t(\aaa)$ is the target of it.
Arrows in a quiver are composed from left to right as follows: for arrows $\aaa$ and $\bbb$ we write $\aaa\bbb$ for the path from the source of $\aaa$ to the target of $\bbb$.
We adopt the convention that maps are also composed from left to right, that is if $f: X \to Y$ and $g: Y \to Z$ then $fg: X \to Z$. We denote by $\mathbb{Z}$ the set of integer numbers, and by $\mathbb{N}$ the set of positive integer numbers. 

An algebra $A$ will be assumed to be basic with finite dimension over a base field $k$ which is algebraically closed. In general, we consider right modules, where $\ma$ is the category of finite-dimensional modules over $A$.
For a module $M$, we denote by $\add M$ the subcategory of
$\ma$ consisting of direct summands of finite direct sums of
copies of $M$. For example $\add A$ is the category
of finitely generated projective $A$-modules, and
$\add DA$ is the category
of finitely generated injective $A$-modules.

We denote by $\cald^b(A)$ the (bounded) derived category of $\ma$. Since $\cald^b(A)$ is triangle equivalent to the homotopy category $K^{-,b}(\proj A)$ of complexes of projective $A$-modules bounded on the right and bounded in homology, we will not distinguish these two categories, and view the perfect derived category $K^b(\proj A)$ as a full subcategory of $\cald^b(A)$.

\subsection{The modules over gentle algebras}\label{subsection: gentle algebras}

In this subsection, we recall some basic definitions and constructions of the modules over gentle algebras.

\begin{definition}\label{definition:gentle algebras}
We call an algebra $A=kQ/I$ a \emph{gentle algebra}, if $Q=(Q_0,Q_1)$ is a finite quiver and $I$ is an admissible ideal of $kQ$ satisfying the following conditions:
\begin{enumerate}[\rm(1)]
 \item Each vertex in $Q_0$ is the source of at most two arrows and the target of at most two arrows.

 \item For each arrow $\aaa$ in $Q_1$, there is at most one arrow $\bbb$ such that  ${\bf{0}} \neq \aaa\bbb\in I$; at most one arrow $\ccc$ such that  ${\bf{0}} \neq \ccc\aaa\in I$; at most one arrow $\bbb'$ such that $\aaa\bbb'\notin I$; at most one arrow $\ccc'$ such that $\ccc'\aaa\notin I$.

 \item $I$ is generated by paths of length two.
\end{enumerate}
\end{definition}

It is well-known that any indecomposable module in $\ma$ is either a \emph{string module} or a \emph{band module}, which is parameterized by string and band combinatorics respectively \cite{BR87}.
The maps between the indecomposable modules are characterized in \cite{C89}.
Since we are interested in the tilting theory of gentle algebras, and any band module has non-trivial self-extensions, the band module plays no role in our consideration. 
So in the following, we only recall the construction of string modules. 
We refer the reader to \cite{BR87} for more details. 

For an arrow $\aaa$, let $\aaa^{-1}$ be its \emph{formal inverse} with $s(\aaa^{-1})=t(\aaa)$ and $t(\aaa^{-1})=s(\aaa)$. A \emph{walk} is a (possibly infinite) sequence $\cdots\omega_1\omega_2\cdots\omega_m\cdots$ of arrows and inverse arrows in $Q$ such that $t(\omega_i)=s(\omega_{i+1})$ and $\omega_{i+1}\neq\omega^{-1}_i$ for each $i$.
A walk $\omega$ is called a \emph{finite walk}, if it consists of finite number of arrows and formal inverses.

A \emph{string} is a finite walk $\omega$ that avoids relations, that is, there is no subsequence of $\omega$ or of $\omega^{-1}$ belongs to $I$.
A \emph{direct string} is a string consisting of arrows, and an \emph{inverse string} is a string consisting of formal inverses.
For each vertex $v$, we associate a \emph{trivial string} $1_v$ to it.
Each string $\omega$ defines a \emph{string module} $M(\omega)$, which is given by the quiver representation of type $A$ obtained by replacing every vertex in $\omega$ by a copy of $k$ and every arrow by the identity map.
This gives a bijection between the inversion equivalent classes of strings and the isomorphism classes of string modules. In particular, $M(1_v)$ is the simple module arising from the vertex $v$ of $Q$.

\subsection{Marked surfaces}\label{subsection: geo-module categories}

We recall some concepts about marked surfaces.

\begin{definition}
	\label{definition:marked surface}
	A \emph{marked surface} is a pair $(\cals,\calm)$, where
	\begin{enumerate}[\rm(1)]
		\item $\cals$ is an oriented surface with non-empty boundary with
		connected components $\partial \cals=\sqcup_{i=1}^{b}\partial_i \cals$;
		\item $\calm = \calm_{\bpoint} \cup \calm_{\rpoint} \cup \calp_{\bpoint}$ is a finite set of \emph{marked points} on $\cals$. The elements of~$\calm_{\bpoint}$ and~$\calm_{\rpoint}$ are on the boundary of $\cals$, which will be respectively represented by symbols~$\bpoint$ and~$\rpoint$. Each connected component $\partial_i \cals$ is required to contain at least one marked point of each colour, where the points~$\bpoint$ and~$\rpoint$ are alternating on $\partial_i \cals$. The elements in $\calp_{\bpoint}$ are in the interior of $\cals$. We refer to these points as \emph{punctures}, and we will also represent them by the symbol $\bpoint$.
	\end{enumerate}
\end{definition}
Let $(\cals,\calm)$ be a marked surface.
	\begin{enumerate}[\rm(1)]
		\item An \emph{arc} is a non-contractible curve, with endpoints in~$\calm_{\bpoint}\cup \calp_{\bpoint}\cup\calm_{\rpoint}$. It is an $\rpoint$-\emph{arc} if the endpoints are from $\calm_{\rpoint}$, and it is an $\bpoint$-\emph{arc} if the endpoints are from $\calm_{\bpoint}\cup \calp_{\bpoint}$.
		\item A \emph{loop} is an arc whose endpoints coincide.

		\item A \emph{simple arc} is an arc without interior self-intersections.
	\end{enumerate}

In order for some definitions and notations to be well-defined in the case of a loop, we will treat the unique endpoint of a loop as two distinct endpoints.
On the surface, all curves are considered up to homotopy with respect to the boundary components and the punctures, and all intersections of curves are required to be transversal.

\subsection{Coordinates and dissections}\label{subsection:coord-diss}
To realize the gentle algebras and their module categories on marked surfaces, we introduce the following.
\begin{definition}\label{definition:addmissable dissections}
	\begin{enumerate}[\rm(1)]
		\item  A collection of simple arcs with endpoints in $\calm_{\rpoint}$ is called a \emph{simple coordinate}, if the arcs have no interior intersections and they cut the surface into polygons each of which contains exactly one marked point from $\calm_{\bpoint}\cup \calp_{\bpoint}$. 
		\item A collection of simple arcs with endpoints in $\calm_{\bpoint}$ is called an \emph{admissible dissection}, if the arcs have no interior intersections and they cut the surface into polygons each of which contains exactly one marked point from $\calm_{\rpoint}\cup \calp_{\bpoint}$.
	\end{enumerate}
\end{definition}

We will denote a simple coordinate by $\zD^*$ and an admissible dissection by $\zD$.
More generally, a collection $\zG$ of simple $\bpoint$-arcs with endpoints in $\calm_\bpoint$ is called an \emph{admissible collection}, if the arcs have no interior intersections and they cut the surface into sub-surfaces each of which contains \emph{at least} one marked point from $\calm_{\rpoint}\cup \calp_{\bpoint}$. We mention that the arc systems appearing above are named in different ways in different papers, for example, an admissible dissection is called a full formal arc system in \cite{HKK17}, a lamination in \cite{OPS18}, and a projective dissection in \cite{C23}. Since the paper \cite{C23} deals with both the module category and the derived category, and the `simple-projective duality' plays an important role, so the projective dissection is used there, see \cite[Remark 1.5]{C23} for more details. However, in this paper, we only focus on the module category and the tilting modules, and we will introduce tilting dissection in the following, therefore `projective dissection' is misleading. Thus we use `admissible dissection', as the one used in \cite{APS23}.

We assume that $\cals$ has $c$ connected components. 
Denote by $g$ the genus of $\cals$ and by $b$ the number of connected components of $\partial \cals$. Then the Euler character of $\cals$ is $\chi=2c-2g-b-|\calp_\bpoint|$, where a puncture can be viewed as a boundary component without marked points. It is shown in \cite{APS23}, see also \cite{PPP21}, that there are exactly $|\calm_{\bpoint}|-\chi|$ arcs in a simple coordinate as well as in an admissible dissection. We call $|\calm_{\bpoint}|-\chi|$ the \emph{rank} of $(\cals,\calm)$, and denote it by $n$.
Furthermore, an admissible collection is an admissible dissection if and only if it has $n$ arcs.

Since the boundary of a marked surface is non-empty and there is at least one $\bpoint$-point on each boundary component, therefore $|\calm_{\bpoint}|\geq b\geq 1$, and we have the following 
\begin{lemma}\label{lamma:rankonesur}
Let $(\cals,\calm)$ be a marked surface with rank one, then it must be one of the following cases: a disk with two marked $\bpoint$-points on the boundary, a once-punctured disk with one marked $\bpoint$-point on the boundary.
\end{lemma}

\subsection{Gentle algebras from simple coordinates}\label{subsection: gentle algebras}

We recall how to construct a gentle algebra from a simple coordinate $\zD^*=\{\ell_i^*, 1\leq i \leq n\}$ on $(\cals,\calm)$.

\begin{definition}\label{definition:oriented intersection}
	Let $q$ be a common endpoint of arcs $\ell_i^*, \ell_j^*$ in $\zD^*$, an \emph{oriented intersection} from $\ell_i^*$ to $\ell_j^*$ at $q$ is an anti-clockwise angle locally from $\ell_i^*$ to $\ell_j^*$ based at $q$ such that the angle is in the interior of the surface. An oriented intersection is \emph{minimal} if it is not a composition of two oriented intersections of arcs from $\zD^*$.
\end{definition}

\begin{definition}\label{definition:gentle algebra from dissection}
	We define the algebra $A(\zD^*)$ as the quotient of the path algebra $kQ(\zD^*)$ of the quiver $Q(\zD^*)$ by the ideal $I(\zD^*)$ defined as follows:
	\begin{enumerate}[\rm(1)]
		\item The vertices of $Q(\zD^*)$ are given by the arcs in $\zD^*$.
		\item Each minimal oriented intersection $\aaa$ from $\ell_i^*$ to $\ell_j^*$ gives rise to an arrow from $\ell_i^*$ to $\ell_j^*$, which is still denoted by $\aaa$.
		\item The ideal $I(\zD^*)$ is generated by paths $\aaa\bbb:\ell_i^*\rightarrow \ell_j^*\rightarrow \ell_k^*$, where the common endpoint of $\ell_i^*$ and $\ell_j^*$, and the common endpoint of $\ell_j^*$ and $\ell_k^*$ that respectively gives rise to $\aaa$ and $\bbb$ are different.
	\end{enumerate}
\end{definition}

Then it is not hard to check that $A(\zD^*)$ is a gentle algebra. 
Conversely, it is proved in \cite{BC21,OPS18,PPP19} that any gentle algebra arises from this way. So this establishes a bijection between the set of triples  $(\cals,\calm,\zD^*)$ of homeomorphism classes of marked surfaces together with simple coordinates and the set of isomorphism classes of gentle algebras $A(\zD^*)$.

\begin{example}\label{ex:gen-alg}
See Figure \ref{fig:alg-ex} for an example of a marked surface with a simple coordinate and the associated gentle algebra.
 \begin{figure}[ht]
		\begin{center}
 		\begin{tikzpicture}[>=stealth,scale=0.6]
				\draw[line width=1.8pt,fill=white] (0,0) circle (4cm);
				\draw[thick,fill=gray!50] (0,0) circle (0.8cm);
				\path (0:4) coordinate (b1)
				(70:4) coordinate (b2)
				(110:4) coordinate (b3)
				(135:4) coordinate (b4)
				(200:4) coordinate (b5)
				(90:4) coordinate (b6)
				(-90:4) coordinate (b7)
				(-70:4) coordinate (b8)
				(-50:4) coordinate (b9)
				(45:4) coordinate (b10);
				
				\draw[red,line width=1.3pt]plot [smooth,tension=0.6] coordinates {(0,-4) (-1.7,0) (-0.8,1.2) (0,0.8)};
				\draw[red,line width=1.3pt]plot [smooth,tension=0.6] coordinates {(0,-4) (1.7,0) (.8,1.2) (0,0.8)};	
				\draw [red, line width=1.5pt] (b6) to (0,0.8);
				\draw[cyan,line width=1.5pt,red] (b6) to[out=-120,in=-10](b4);
				\draw[cyan,line width=1.5pt,red] (b6) to[out=-60,in=-170](b10);
				\draw[cyan,line width=1.5pt,red] (b7) to[out=50,in=170](b9);
				\draw[red] (-1.5,2.5) node {$1$};
				\draw[red] (-.3,2.5) node {$2$};
				\draw[red] (1.5,2.5) node {$3$};
				\draw[red] (1.5,-1.8) node {$4$};
				\draw[red] (-1.5,-1.8) node {$5$};
				\draw[red] (1.7,-2.7) node {$6$};
	\draw (7,0) node {};
	\draw[thick,fill=black] (b1) circle (0.08cm)
				(b2) circle (0.08cm)
				(b3) circle (0.08cm)
				(b5) circle (0.08cm)
				(b8) circle (0.08cm);	
	\draw[thick,red, fill=red] (b4) circle (0.08cm)
				(b6) circle (0.08cm)
				(b7) circle (0.08cm)
				(b9) circle (0.08cm)
				(b10) circle (0.08cm);
				
				\draw[thick, red ,fill=red] (0,0.8) circle (0.08cm);
				\draw[thick, black ,fill=black] (0,-0.8) circle (0.08cm);
			\end{tikzpicture}
{\begin{tikzpicture}[ar/.style={->,very thick,>=stealth}]
				\draw(-5,3)node(6){$6$}
				(-4,1.5)node(4){$4$}
				(-2,1.5)node(5){$5$}
				(-3,0)node(2){$2$}
				(-4,-1.5)node(3){$3$}
				(-2,-1.5)node(1){$1$};

				\draw[ar](6) to  (4);
				\draw[ar](4) to (5);
				\draw[ar](4) to  (2);
				\draw[ar](2)to  (5);
				\draw[ar](2) to  (3);
				\draw[ar](1) to  (2);
				\draw[ar](2)to (3);
	
   \draw[bend right,dotted,thick]($(6)!.6!(4)-(0,0.1)$) to ($(4)!.25!(2)-(0,0.1)$);
				\draw[bend right,dotted,thick]($(4)!.55!(2)-(0,0.1)$) to ($(2)!.3!(3)-(0,0.05)$);
				\draw[bend left,dotted,thick]($(5)!.55!(2)-(0,0.1)$) to ($(2)!.3!(1)-(0,0.05)$);
				
		\end{tikzpicture}}
		\end{center}	
  \caption{The left picture is a marked surface with a simple coordinate $\zD^*$, where the arcs in $\zD^*$ are the $\rpoint$-arcs. The right picture shows the gentle algebra associated with it, where the dotted lines represent the (quadratic) relations in the algebra.} 
	\label{fig:alg-ex}
	\end{figure}
\end{example}

\subsection{Zigzag arcs and string modules}\label{subsection: sur-module}
Now we recall the geometric model of the module category of a gentle algebra, see \cite{BC21,C23} for details.
Let's start with constructing a string module from an $\bpoint$-arc. 
In the following, let $(\cals,\calm,\zD^*)$ be a marked surface with a simple coordinate, and let $A$ be the associated gentle algebra.

Note that $\zD^*$ cut the surface into polygons $\bbp$ each of which has exactly one marked point from $\calm_{\bpoint}$ or from $\calp_{\bpoint}$. These polygons will be called the \emph{polygons of $\zD^*$}.
We denote a polygon $\bbp$ of $\zD^*$ by $(\ell^*_{i_1},\cdots,\ell^*_{i_m})$, the ordered set of arcs in $\zD^*$ which form $\bbp$, where the arcs are ordered clockwise and where the index $1\leq j \leq m$ is considered modulo $m$ if $\bbp$ contains a puncture. For any $\ell^*_{i_j}\in \bbp$, we call $\ell^*_{i_{j-1}}$ (if exists) the \emph{predecessor} of $\ell^*_{i_{j}}$ in $\bbp$ and call $\ell^*_{i_{j+1}}$ (if exists) the \emph{successor} of $\ell^*_{i_{j}}$ in $\bbp$, see Figure \ref{figure:def-zigzag}.
In particular, $\ell^*_{i_{1}}$ has no predecessor and $\ell^*_{i_{m}}$ has no successor if $\bbp$ has a marked point from $\calm_{\bpoint}$, while $\ell^*_{i_{1}}$ has the predecessor $\ell^*_{i_{m}}$ and $\ell^*_{i_{m}}$ has the successor $\ell^*_{i_{1}}$ if $\bbp$ has a puncture.

{\bf Setting}: {\it Let $\za$ be an $\bpoint$-arc on the marked surface $(\cals,\calm,\zD^*)$. After choosing a direction of $\za$, denote by $\bbp_0, \bbp_1,\cdots, \bbp_{n+1}$ the ordered polygons of $\zD^*$ that successively intersect with $\za$. Denote by $\ell^*_0, \ell^*_1,\cdots, \ell^*_{n}$ the ordered arcs in $\zD^*$ that successively intersect with $\za$ such that $\ell^*_i$ belongs to $\bbp_{i}$ and $\bbp_{i+1}$ for each $0 \leq i \leq n$.}

\begin{definition}\label{definition:zigzag arcs}
Let $\za$ be an $\bpoint$-arc on $(\cals,\calm,\zD^*)$. Under the notations in {\bf Setting}, we call $\za$ a \emph{zigzag arc} (with respect to $\zD^*$) if in each polygon $\bbp_{i+1}, 0 \leq i \leq n+1$, $\ell^*_{i+1}$ is the predecessor or the successor of $\ell^*_{i}$. Furthermore, if $\bbp_{i+1}$ is a polygon which contains a puncture and has $m$ edges with $m\neq 2$, then we also need the puncture is not in the (unique) triangle formed by the segments of $\ell^*_{i}$, $\ell^*_{i+1}$ and $\za$, see the right picture of Figure \ref{figure:def-zigzag}.
	\begin{figure}
			\begin{tikzpicture}[>=stealth,scale=0.8]
				\draw[red,very thick] (0,0)--(-1,2)--(2,3)--(5,2)--(4,0);
\draw [ gray!60, line width=4pt] (5,-.07)--(-1,-.07);	
\draw[line width=2pt](5,0)--(-1,0);\draw[red,thick,fill=red] (0,0) circle (0.1);
				\draw[red,thick,fill=red] (-1,2) circle (0.1);
				\draw[red,thick,fill=red] (2,3) circle (0.1);
				\draw[red,thick,fill=red] (5,2) circle (0.1);
				\draw[red,thick,fill=red] (4,0) circle (0.1);
				
				\draw[thick,fill=white] (2,0) circle (0.1);
				\node at (.8,.5) {$\bbp_{i+1}$};
				\node at (.8,1.5) {\tiny$\za$};
				\node[red] at (3.9,2.7) {\tiny$\ell^*_{i_{j+1}}$};
				\node[red] at (0,2.7) {\tiny$\ell^*_{{i_j}}$};
				\node[red] at (-.9,.7) {\tiny$\ell^*_{{i_1}}$};	
    			\node[red] at (4.9,.7) {\tiny$\ell^*_{{i_m}}$};	
				\draw[bend right,thick](-1,2.5)to(5,2.5);
				\draw[bend right,thick,->](1.4,2.8)to(2.6,2.8);
				\node at (2,2.3) {\tiny$\aaa_{i+1}$};
				\node at (7,1.5) {};
			\end{tikzpicture}	
			\begin{tikzpicture}[>=stealth,scale=0.8]
				\draw[red,very thick] (0,0)--(-1,2)--(2,3)--(5,2)--(4,0);
				\draw[red,very thick](4,0)--(0,0);
				\draw[red,thick,fill=red] (0,0) circle (0.1);
				\draw[red,thick,fill=red] (-1,2) circle (0.1);
				\draw[red,thick,fill=red] (2,3) circle (0.1);
				\draw[red,thick,fill=red] (5,2) circle (0.1);
				\draw[red,thick,fill=red] (4,0) circle (0.1);
				
				\draw[thick,fill=white] (2,1) circle (0.1);
				\node at (.8,.5) {$\bbp_{i+1}$};
				\node at (.8,1.5) {\tiny$\za$};
			\node[red] at (3.9,2.7) {\tiny$\ell^*_{i_{j+1}}$};
				\node[red] at (0,2.7) {\tiny$\ell^*_{{i_j}}$};
				\node[red] at (-.9,.7) {\tiny$\ell^*_{{i_1}}$};	
    			\node[red] at (5,.7) {\tiny$\ell^*_{{i_{m-1}}}$};	
    			\node[red] at (2,.3) {\tiny$\ell^*_{{i_m}}$};	
				\draw[bend right,thick](-1,2.5)to(5,2.5);
				\draw[bend right,thick,->](1.4,2.8)to(2.6,2.8);
				\node at (2,2.3) {\tiny$\aaa_{i+1}$};
				
		\end{tikzpicture}
		\begin{center}
			\caption{Two types of polygons formed by arcs in a simple coordinate and boundary segments, where a zigzag arc $\za$ passes through the polygon along an oriented intersection $\aaa_{i+1}$.}\label{figure:def-zigzag}
		\end{center}
	\end{figure}
\end{definition}

{\bf Construction. String of a zigzag arc.}

Let $\za$ be a zigzag arc on $(\cals,\calm,\zD^*)$ with notation in {\bf Setting}, we associate a string $\omega(\za)$ of $A$ to it in the following way.
For each $0\leq i \leq n-1$, there is an arrow in $A$ from $\ell_i^*$ to $\ell_{i+1}^*$ arising from an oriented intersection of $\bbp_{i+1}$, which is a minimal oriented intersection and we denote it by $\aaa_{i+1}$, see the pictures in Figure \ref{figure:def-zigzag}.
We associate a walk $\omega(\za)=\omega_1\cdots\omega_n$ to $\za$, where $\omega_{i+1}=\aaa_{i+1}$ if $\za$ enter $\bbp_{i+1}$ through $\ell_i^*$ and leave through $\ell_{i+1}^*$, or $\omega_{i+1}=\aaa^{-1}_{i+1}$ if $\za$ enter $\bbp_{i+1}$ through $\ell_{i+1}^*$ and leave through $\ell_i^*$. Since $\za$ is zigzag, it is straightforward to see that $\omega(\za)$ is a string of $A$.

\begin{definition}\label{prop-def:string and band from curves}
	Let $(\cals,\calm,\zD^*)$ be a marked surface with a simple coordinate.
For a zigzag $\bpoint$-arc $\za$ on the surface, we call the string module $M_\za$ of $A$ arising from $\omega(\za)$ the \emph{string module of $\za$}.
\end{definition}


\begin{proposition}\label{theorem:main arcs and objects}
Let $(\cals,\calm,\zD^*)$ be a marked surface with a simple coordinate.
The map $M: \za\mapsto M_\za$ gives a bijection between zigzag arcs on $(\cals,\calm,\zD^*)$ and indecomposable string modules over $A$. 
\end{proposition}

\subsection{Intersections as morphisms and extensions}\label{subsection: sur-ext}

Let's interpret the intersections of zigzag arcs as the morphisms and extensions between the associated modules, see \cite{C23}. 
For convenience, we view any $\bpoint$-point as a zero zigzag arc, and then the module associated to it is just the zero module.

\begin{definition}\label{definition: weight}
	Let $\za$ be a zigzag arc with an endpoint $p$ in $\calm_{\bpoint}$, where $p$ belongs to a polygon $\bbp=\{\ell^*_1,\ell^*_2,\cdots,\ell^*_n\}$ of $\zD^*$ with the arcs labeled clockwise, see Figure \ref{figure:weight and co-weight}.
	Assume that starting from $p$, $\ell^*_t$ is the first arc that $\za$ intersects, for some $1\leq t \leq n$. We call $n-t$ the \emph{weight} of $\za$ at $p$, which is denoted by $w_p(\za)$.
	\begin{figure}
		\begin{center}
			\begin{tikzpicture}[>=stealth,scale=.8]
				\draw[red,very thick] (1,0)--(3,1)--(3,3)--(1,4);
				\draw[red,very thick,dashed] (-1,4)--(1,4);
				\draw[red,very thick] (-1,0)--(-3,1)--(-3,3)--(-1,4);
				
				\draw[line width=1pt] (-2.5,4)--(0,0)--(2.5,4);
				\draw [ gray!60, line width=4pt] (-2,-.07)--(2,-.07);		
					\draw[line width=2pt] (-2,0)--(2,0);			\draw[thick,fill=white] (0,0) circle (0.06);
				\draw (0,-.5) node {$p$};
				\node[red!50] at (-1.5,3.5) {\tiny$\ell^*_{t}$};
				\node[red!50] at (1.5,3.5) {\tiny$\ell^*_{t+\omega}$};
				\node at (-1.5,2) {$\za$};
				\node at (1.5,2) {$\zb$};
				\node[red!50] at (-2.3,1) {\tiny$\ell^*_{1}$};
				\node[red!50] at (2.3,1) {\tiny$\ell^*_{n}$};
				\node at (0,3) {$\bbp$};
				
				\draw[thick,bend left,->](-.2,.3)to(.2,.3);
				\node [] at (0,.6) {$\mathfrak{p}$};	
				\draw[red,thick,fill=red] (1,0) circle (0.1);
				\draw[red,thick,fill=red] (-1,0) circle (0.1);
				\draw[red,thick,fill=red] (3,1) circle (0.1);
				\draw[red,thick,fill=red] (3,3) circle (0.1);
				\draw[red,thick,fill=red] (1,4) circle (0.1);
				\draw[red,thick,fill=red] (1,0) circle (0.1);
				\draw[red,thick,fill=red] (-3,1) circle (0.1);
				\draw[red,thick,fill=red] (-3,3) circle (0.1);
				\draw[red,thick,fill=red] (-1,4) circle (0.1);
			\end{tikzpicture}
		\end{center}
		\begin{center}
			\caption{For a zigzag arc $\za$ with endpoint $p$ which intersects $\ell^*_t$, the weight  $w_p(\za)$ of $\za$ at $p$ equals $n-t$. The weight $\omega(\mathfrak{p})$ of an oriented intersection $\mathfrak{p}$ from $\za$ to $\zb$ is defined as $w_p(\za)-w_p(\zb)$, which equals $\omega$.}\label{figure:weight and co-weight}
		\end{center}
	\end{figure}
\end{definition}
The following proposition says that the weights of the associated arc give the projective dimension of any string module, see Corollary 2.28 in \cite{C23}.
\begin{proposition}\label{prop:proj-res-string}
	Let $\za$ be a zigzag arc, then
	\begin{enumerate}[\rm(1)]
		\item $\pd M_\za=max\{w_{p_1}(\za),w_{p_2}(\za)\}$, if both the endpoints $p_1$ and $p_2$ of $\za$ belong to $\calm_{\bpoint}$;
		\item $\pd M_\za=\infty$, if at least one endpoint of $\za$ belongs to $\calp_{\bpoint}$.
	\end{enumerate}
\end{proposition}

\begin{definition}\label{definition:weighted intersections1}
	Let $\za$ and $\zb$ be two zigzag arcs with endpoints in $\calm_\bpoint$.
	\begin{enumerate}[\rm(1)]
		\item if $\za$ and $\zb$ share a common endpoint $p$ in $\calm_{\bpoint}$, then a \emph{weighted-oriented-intersection} from $\za$ to $\zb$ is an \emph{oriented intersection $\mathfrak{p}$ clockwise} from $\za$ to $\zb$ which arises from $p$, with \emph{weight} $w(\mathfrak{p})=w_p(\za)-w_p(\zb)$, see Figure \ref{figure:weight and co-weight};
		\item if $\za$ and $\zb$ intersect at an interior point $p$, then there is a \emph{weighted-oriented-intersection} $\mathfrak{p}$ from $\za$ to $\zb$, as well as a \emph{weighted-oriented-intersection} $\mathfrak{p}'$ from $\zb$ to $\za$, with \emph{weight} $w(\mathfrak{p})=1, w(\mathfrak{p}')=0$, or $w(\mathfrak{p})=0, w(\mathfrak{p}')=1$, depending on the position of the intersection, see \cite[Proposition-Definition 2.25]{C23} for the precise definition.
	\end{enumerate}  
\end{definition}

\begin{remark}
The definition of an oriented intersection from an $\bpoint$-arc $\za$ to an $\bpoint$-arc $\zb$ based at a common endpoint $p$ is similar to the definition of an oriented intersection between $\rpoint$-arcs $\ell^*$ in $\zD^*$ given in Definition \ref{definition:oriented intersection}, that is an angle locally from $\za$ to $\zb$ based at $p$ such that the angle is in the interior of the surface. The difference is that here we use clockwise orientation, rather than anti-clockwise orientation. We also mention that in order for the definition to be well-defined in the case of a loop, we treat the unique endpoint of a loop as two distinct endpoints.
\end{remark}

\begin{proposition}\label{prop:main-extensions}
	Let $(\cals,\calm,\zD^*)$ be a marked surface with a simple coordinate, and let $\za$ and $\zb$ be two zigzag arcs on the surface with endpoints in $\calm_{\bpoint}$. Then for any \emph{oriented intersection} from $\za$ to $\zb$ with weight $\omega$, there is a morphism in $\Ext^\omega(M_\za,M_\zb)$ associated to it.	
	Furthermore, all of such morphisms form a basis of the space $\Ext^\omega(M_\za,M_\zb)$, unless
$\za$ and $\zb$ are the same $\bpoint$-arc. In this case, the identity map is the extra basis map of $\Hom(M_\za,M_\zb)$.
\end{proposition}

\begin{remark}\label{rem:punctureext}
Since the projective dimension of a tilting module is finite, by Corollary \ref{prop:proj-res-string}, it arises from zigzag arcs with endpoints from $\calm_\bpoint$. Thus Proposition \ref{prop:main-extensions} is enough for us.
Suppose two zigzag arcs $\za$ and $\zb$ share an endpoint $q$ which is a puncture from $\calp_\bpoint$ (rather than from $\calm_\bpoint$). In that case, there are infinitely many oriented intersections from $\za$ to $\zb$, and from $\zb$ to $\za$ associated to $q$, which gives rise to infinite many of linear independent morphisms in $\Ext^\omega(M_\za,M_\zb)$ and $\Ext^\omega(M_\zb,M_\za)$ respectively. 
We refer the reader for more details in \cite{C23}.
\end{remark}

\section{Tilting modules over gentle algebras}\label{section:tilting-module}
Now we consider the problem of tilting completion in the module category over a gentle algebra. We start with realizing tilting modules on the marked surface.

\subsection{A geometric realization for  tilting modules
}\label{subsection: geo-tilting}
Let $(\cals,\calm,\zD^*)$ be a marked surface with a simple coordinate, and let $A$ be the associated gentle algebra.
Recall the definition of admissible collections and admissible dissections given in Subsection \ref{subsection:coord-diss}.

\begin{definition}\label{def:tilting-dissection}
(1) an admissible collection $\top$ on $(\cals,\calm)$ is called a \emph{pre-tilting collection on $(\cals,\calm,\zD^*)$} if any arc in $\top$ is zigzag, and the weight of any oriented intersection of two arcs from $\top$ is zero; 

(2) a pre-tilting collection is called a \emph{tilting dissection}, if it is an admissible dissection;

(3) a subset of a tilting dissection is called a \emph{partial-tilting collection}.
\end{definition}

The following proposition justifies the above naming.

\begin{proposition}\label{prop:tilting dissection}
Let $T=\oplus_{i=1}^m M_{\zg_i}$ be an $A$-module arising from a collection of zigzag $\bpoint$-arcs  $\top=\{\zg_i, 1\leq i \leq m\}$, then 
 
 (1) $T$ is a pre-tilting module if and only if $\top$ is a pre-tilting collection;
 
 (2) $T$ is a tilting module if and only if $\top$ is a tilting dissection;

 (3) $T$ is a partial-tilting module if and only if $\top$ is a partial-tilting collection.
\end{proposition}
\begin{proof}
	Let $\top$ be a pre-tilting collection. Then by Corollary \ref{prop:proj-res-string}, $\pd T$ is finite, noticing that the endpoints of each arc from $\top$ are in $\calm_\bpoint$. 
	Furthermore, since there are no interior intersections between the arcs in $\top$ and the weight of any oriented boundary intersection of two arcs from $\top$ is zero, then by Proposition \ref{prop:main-extensions}, $T$ is orthogonal. Therefore $T$ is a pre-tilting module.
	
	Denote by $\P_T$ the projective resolution of $T$. 
	If $\top$ is a tilting-dissection, we will show that $T$ is a tilting module, or equivalently to showing that  $\P_T$ is a tilting complex, see for example \cite[Corollary 3.7]{Wei13}. On the other hand, because $\P_T$ is a projective resolution of a module, it has no negative extensions, and thus it is a tilting complex if and only if it is a silting complex. In fact, note that there are $n$ arcs in $\top$, and thus the rank of $\P_T$ equals the rank of $A$, therefore $\P_T$ is a silting complex by \cite[Proposition 5.7]{APS23}, which states that a pre-silting complex in $\ka$ is a silting complex if and only if its rank equals the rank of $A$.
    
	Now let $T$ be a pre-tilting/tilting module in $\ma$ arising from an $\bpoint$-arc system $\top$. In particular, $\P_T$ is a pre-silting/silting complex in $\dba$, and thus $\top$ is an admissible collection/dissection by \cite[Lemma 5.6]{APS23}/\cite[Theorem 5.2]{APS23}. Furthermore, since $T$ is orthogonal, the degrees of extensions between any two indecomposable direct summands $M_\za$ and $M_\zb$ of $T$ are zero. On the other hand, the weights of the intersections between $\za$ and $\zb$ coincide with the degrees of the corresponding extensions, therefore the weight of any intersection between arcs in $\top$ is zero, and thus by definitions, $\top$ is a pre-tilting collection/tilting dissection.
 
    The statement for the partial-tilting module follows from the second statement straightforwardly. 
\end{proof}

Enomoto conjectured in \cite{E23} that a self-orthogonal $A$-module has a finite projective dimension. In \cite{M23}, Marczinzik confirmed this conjecture for Iwanaga-Gorenstein monomial algebras.
In particular, since gentle algebras are always Iwanaga-Gorenstein by \cite{GR02}, a self-orthogonal module over a gentle algebra has finite projective dimension. This can also be observed using the surface model. More precisely, an indecomposable module over a gentle algebra has infinite projective dimension if and only if it is a string module, and there is at least one endpoint of the associated arc belonging to $\calp_\bpoint$. In this case, the module has infinite many self-intersections, see \cite[Theorem 2.30]{C23}. Therefore an indecomposable self-orthogonal module always arises from an $\bpoint$-arc with both endpoints in $\calm_\bpoint$, whose projective dimension is finite by Proposition \ref{prop:proj-res-string}, which confirms the claim.
Combine this observation with Proposition \ref{prop:tilting dissection}, we have the following
\begin{corollary}
	An orthogonal $A$-module $T$ is a tilting module if and only if it is of full rank, that is, $|T|=|A|$.
\end{corollary}

\begin{remark}
	(1) The corollary shows that the answer of $\bf{(Cn)}$ is positive for gentle algebras.
 
	(2) Since a $\tau$-tilting module is of full rank, the answer to Zhang's conjecture \cite{Z19} is positive for gentle algebras: any self-orthogonal $\tau$-tilting module is tilting.
\end{remark}

\subsection{A reduction construction by cutting surfaces}\label{section:reduction construction}

We will give a reduction construction for both the marked surface and the simple coordinate, which is the key tool we will use in the following.
Let $(\cals,\calm,\zD^*)$ be a marked surface with a simple coordinate and let $\zg$ be a zigzag simple $\bpoint$-arc on the surface with endpoints $p_1,p_2 \in \calm_\bpoint\cup \calp_\bpoint$.
We define a triple $(\cals_\zg,{\calm_\zg},{\zD^*_\zg})$ in the following way.

The surface $\cals_\zg$ is obtained by cutting $\cals$ along $\zg$, where $\zg$ becomes two boundary segments, which are denoted by $\zg'$ and $\zg''$, with endpoints $p_1',p_2'$ and $p_1'',p_2''$ respectively.

The set $\calm_\zg$ is defined as 
$$\calm_\zg=\calm\setminus \{p_1,p_2\}\cup\{p_1',p_2',p_1'',p_2''\}\cup\{q',q''\},$$
where $q'$ and $q''$ are newly added $\rpoint$-points on $\zg'$ and $\zg''$ (in between $\bpoint$-points $p_1'$ and $p_2'$, and $p_1''$ and $p_2''$ respectively). Depending on the positions of $p_1$ and $p_2$, that is on the boundary or in the interior of the surface, as well as whether $p_1$ coincides with $p_2$ or not, there are several possibilities, which we list in the Appendix, see Figure \ref{table:list}.

The simple coordinate $\zD^*_\zg$ is obtained in the following steps, with the notations used corresponding to those in Figure \ref{fig:coord.}:

{\bf Step one:} Let $\call=\{\ell^*\}$ be the set of arcs in $\zD^*$ that intersect $\zg$. Assume that $\zg$ cuts each $\ell^*$ into consecutive segments $\varsigma_1, \varsigma_2,\cdots,\varsigma_t$, where we denote the endpoints of $\varsigma_i$ by $\wp_{i-1}$ and  $\wp_i$. In particular, $\wp_0$ and $\wp_t$ are endpoints of $\ell^*$ which belong to $\calm_{\rpoint}$. 

{\bf Step two:} When we cut the surface along $\zg$, each point $\wp_i$ splits into two points $\wp'_i$ and $\wp''_i$, located on $\zg'$ and $\zg''$, respectively, excepting for $\wp_0$ and $\wp_t$.
Note that for each $2\leq i \leq t-1$, $\varsigma_i$ (as a line segment on $\cals$) induces a corresponding line segment $\overline{\varsigma}_i$ on $\cals_\zg$, with endpoints $\wp'_{i-1}/\wp''_{i-1}$ and $\wp'_{i}/\wp''_{i}$ located on $\zg'/\zg''$. 
For line segments $\overline{\varsigma}_1$ and $\overline{\varsigma}_t$, the endpoints are $\wp_0$, $\wp'_1/\wp''_1$ and $\wp'_{t-1}/\wp''_{t-1}$, $\wp_t$ respectively. 

{\bf Step three:} For each $2\leq i \leq t-1$, denote by $\ell^*_i$ the $\rpoint$-arc obtained by smoothing the segment $\overline{\varsigma}_i$ along with two segments ${\wp'_{i-1}}q'/{\wp''_{i-1}}q''$ and ${\wp'_{i}}q'/{\wp''_{i}}q''$ at common endpoints ${\wp'_{i-1}}/{\wp''_{i-1}}$ and ${\wp'_{i}}/{\wp''_{i}}$ respectively. 
Denote by $\ell^*_1$ the $\rpoint$-arc obtained by smoothing $\overline{\varsigma}_1$ and segment $\wp'_{1}q'/\wp''_{1}q''$ at the common endpoint ${\wp'_{1}}/{\wp''_{1}}$.
Denote by $\ell^*_t$ the $\rpoint$-arc obtained by smoothing $\overline{\varsigma}_t$ and segment $q'\wp'_{t-1}/q''\wp''_{t-1}$ at the common endpoint ${\wp'_{t-1}}/{\wp''_{t-1}}$.

{\bf Step four:} Finally, $\zD^*_\zg$ is defined as the set 
$$\zD^*_\zg=\zD^*\setminus \call\cup\{\ell^*_1,\cdots,\ell^*_t ~| ~\ell^*\in \call\}$$ of $\rpoint$-arcs on $\cals_\zg$, where we identify the arcs which are homotopic with each other. 

\begin{figure} 
	\centering 
\begin{tikzpicture}[>=stealth,scale=.5]
					
					\draw[thick ,fill=black] (-8,0) circle (0.13);
					\draw[thick ,fill=black] (8,0) circle (0.13);
				    \draw[thick, red ,fill=red] (-7.3,-3) circle (0.12);
				    \draw[thick, red ,fill=red] (7.3,-3) circle (0.12);
			
				    \draw[cyan, line width=1.5pt,red] (-7.3,-3) to[out=75, in=-180] 
				    (-4.5,3.4) to[out=0,in=100]
				    (-3,1.4)  ;
				    
				    \draw[cyan, dashed,line width=1.5pt,red] (-3,1.2) to[out=-80,in=120]
				    (-2.2,-0.7) ;
				
				    \draw[cyan,line width=1.5pt,red] (-2.2,-0.7) to[out=-60,in=-180]
				    (0,-3) to[out=0,in=-180]
				    (4,3.4) to[out=0,in=-180]
				    (7.3,-3) ;
				
					\draw[line width=1.5pt,fill=gray!25] (4,1.5) circle (0.5);
					
					\draw (-8.8,0) node {$p_{1}$};
				    \draw (-7.8,-3.3) node {$\wp_{0}$};
				    \draw (-6.1,-0.6) node {$\wp_{i-1}$};
				    \draw (2.4,-0.6) node {$\wp_{i}$};
				    \draw (5,-0.6) node {$\wp_{i+1}$};
				    \draw (7.9,-3.3) node {$\wp_{t}$};
				    \draw (8.8,0) node {$p_{2}$};
				    \draw (0,0.5) node {$\gamma$};
				    
				    \draw (-7.6,-1.5) node[red] {$\varsigma_{1}$};
				    \draw (0,-2) node[red] {$\ell^*$};
                    \draw (-4.5,3.8) node[red] {$\varsigma_{i}$};
				    \draw (4,3.8) node[red] {$\varsigma_{i+1}$};
				    \draw (6.5,-1.8) node[red] {$\varsigma_{t}$};

					\draw[line width=1.6pt,fill=white] (-3.4,1.2) to[bend left=60] (-2.6,1.2) (-3.6,1.3) to[bend left=-40] (-2.4,1.3);
     \draw [line width=1pt] (-8,0) to (8,0);
\end{tikzpicture}
\begin{tikzpicture}
[>=stealth,scale=.5]       			


       			\draw[cyan, line width=1.5pt,red] (-7.3,-3) to[out=75, in=-180] 
       			(-4.5,3.4) to[out=0,in=100]
       			(-3,1.4)  ;
       			
       			\draw[cyan, dashed,line width=1.5pt,red] (-3,1.2) to[out=-80,in=120]
       			(-2.2,-0.7) ;
       			
       			\draw[cyan,line width=1.5pt,red] (-2.2,-0.7) to[out=-60,in=-180]
       			(0,-3) to[out=0,in=-180]
       			(4,3.4) to[out=0,in=-180]
       			(7.3,-3) ;
\draw [line width=7.5pt,white] (-9,0) to (9,0);
\draw [line width=2pt] (-8,0.3) to (8,0.3);
\draw [line width=2pt] (-8,-0.3) to (8,-0.3);

       			\draw[thick ,fill=white] (-8,0.3) circle (0.13);
       			\draw[thick ,fill=white] (8,0.3) circle (0.13);
       			\draw[thick ,fill=white] (-8,-0.3) circle (0.13);
       			\draw[thick ,fill=white] (8,-0.3) circle (0.13);
       			\draw[thick, red ,fill=red] (-7.3,-3) circle (0.12);
       			\draw[thick, red ,fill=red] (7.3,-3) circle (0.12);
       			\draw[thick, red ,fill=red] (0,0.3) circle (0.12);
       			\draw[thick, red ,fill=red] (0,-0.3) circle (0.12);
       			\draw[line width=1.5pt,fill=gray!25] (4,1.5) circle (0.5);
       			
       			\draw (-8.3,0.9) node {$p'_{1}$};
       			\draw (-8.3,-0.9) node {$p''_{1}$};
       			\draw (8.6,0.8) node {$p'_{2}$};
       			\draw (8.6,-0.8) node {$p''_{2}$};
       			\draw (-1.5,1) node {$\zg'$};
       			\draw (-1.5,-1) node {$\zg''$};
          
\draw (-5.8,1) node {$\wp'_{i-1}$};
\draw (2.7,1) node {$\wp'_{i}$};
\draw (6.5,1) node {$\wp'_{i+1}$};
\draw (-5.8,-1) node {$\wp''_{i-1}$};
\draw (2.7,-1) node {$\wp''_{i}$};
\draw (4.7,-1) node {$\wp''_{i+1}$};

\draw[line width=1.5pt,red] (-7,.6) to (-6.6,.1);
\draw[line width=1.5pt,red] (-7,.1) to (-6.6,.6); 
\draw[line width=1.5pt,red] (-7-.07,.6-.65) to (-6.6-.07,.1-.65);
\draw[line width=1.5pt,red] (-7-.07,.1-.65) to (-6.6-.07,.6-.65);

\draw[line width=1.5pt,red] (-7+8.8,.6) to (-6.6+8.8,.1);
\draw[line width=1.5pt,red] (-7+8.8,.1) to (-6.6+8.8,.6); 
\draw[line width=1.5pt,red] (-7-.07+8.79,.6-.65) to (-6.6-.07+8.79,.1-.65);
\draw[line width=1.5pt,red] (-7-.07+8.79,.1-.65) to (-6.6-.07+8.79,.6-.65);

\draw[line width=1.5pt,red] (7-1.15,.6) to (6.6-1.15,.1);
\draw[line width=1.5pt,red] (7-1.15,.1) to (6.6-1.15,.6); 
\draw[line width=1.5pt,red] (7-1.1,.6-.65) to (6.6-1.1,.1-.65);
\draw[line width=1.5pt,red] (7-1.1,.1-.65) to (6.6-1.1,.6-.65);

       			\draw (0,1) node[red] {$q'$};
       			\draw (0,-1) node[red] {$q''$};
       			\draw (-7.7,-1.8) node[red] {$\overline{\varsigma}_{1}$};
       			\draw (-4.7,3.8) node[red] {$\overline{\varsigma}_{i}$};
       			\draw (4,3.8) node[red] {$\overline{\varsigma}_{i+1}$};
       			\draw (6.6,-1.8) node[red] {$\overline{\varsigma}_{t}$};
       			
       			\draw[line width=1.6pt,fill=white] (-3.4,1.2) to[bend left=60] (-2.6,1.2) (-3.6,1.3) to[bend left=-40] (-2.4,1.3);

\end{tikzpicture}
\begin{tikzpicture}[>=stealth,scale=.5]
\draw [line width=2pt] (-8,0.3) to (8,0.3);
\draw [line width=2pt] (-8,-0.3) to (8,-0.3);
					\draw [line width=1.5pt,red] (7.3,-3) to (0,-0.3);
					
					\draw[thick ,fill=white] (-8,0.3) circle (0.13);
					\draw[thick ,fill=white] (8,0.3) circle (0.13);
					\draw[thick ,fill=white] (-8,-0.3) circle (0.13);
					\draw[thick ,fill=white] (8,-0.3) circle (0.13);
					\draw[thick, red ,fill=red] (-7.3,-3) circle (0.12);
					\draw[thick, red ,fill=red] (7.3,-3) circle (0.12);
					\draw[thick, red ,fill=red] (0,0.3) circle (0.12);
					\draw[thick, red ,fill=red] (0,-0.3) circle (0.12);
					
					\draw[cyan, line width=1.5pt,red]
				    (0,0.3) to[out=160, in=-90]
				    (-5.5,2) to[out=90,in=180]
					(-4.5,3) to[out=0,in=100]
					(-3,1.4)  ;
					
					\draw[cyan, dashed,line width=1.5pt,red] (-3,1.2) to[out=-80,in=180]
					(-1.5,-2.7) ;
					
					\draw[cyan,line width=1.5pt,red] 
					(-1.5,-2.7) to[out=30,in=-90]
					(0,-0.3);
					
					\draw[cyan,line width=1.5pt,red] 
					(-7.3,-3) to[out=10,in=-130]
					(0,-0.3);
					
					\draw[cyan,line width=1.5pt,red] 
					(0,0.3) to[out=20,in=140]
					(4.6,2)to[out=-40,in=5]
					(0,0.3);
					
					\draw[line width=1.5pt,fill=gray!25] (3.6,1.5) circle (0.5);
					\draw [line width=6.5pt,white] (-4,0) to (-2,0);
					
	       		\draw (-8.3,0.9) node {$p'_{1}$};
       			\draw (-8.3,-0.9) node {$p''_{1}$};
       			\draw (8.6,0.8) node {$p'_{2}$};
       			\draw (8.6,-0.8) node {$p''_{2}$};
       			\draw (-6,1) node {$\zg'$};
       			\draw (-6,-1) node {$\zg''$};
					
					\draw (0,0.9) node[red] {$q'$};
					\draw (0.5,-1.3) node[red] {$q''$};
				     \draw (-4,-3) node[red] {$l^{*}_{1}$};
				     \draw (4,-2.5) node[red] {$l^{*}_{t}$};
				     \draw (-4.8,1.8) node[red] {$l^{*}_{i}$};
					\draw (5.6,1.8) node[red] {$l^{*}_{i+1}$};

					\draw[line width=1.6pt,fill=white] (-3.4,1.2) to[bend left=60] (-2.6,1.2) (-3.6,1.3) to[bend left=-40] (-2.4,1.3);
\end{tikzpicture}
	\caption{The top picture depicts a simple $\bpoint$-arc $\zg$ on $(\cals,\calm)$ cuts an $\rpoint$-arc $\ell^*$ into several segments $\varsigma_i$, with endpoints $\wp_{i-1}$ and $\wp_i$, which naturally induce segments $\overline{\varsigma}_i$ on the surface $\cals_\zg$, whose endpoints are $\wp'_{i-1}/\wp''_{i-1}$ and $\wp'_{i}/\wp''_{i}$, that is the red crossed points in the middle picture. Then after smoothing these segments with the line segments $\wp'_{i-1}q'/\wp''_{i-1}q''$ and $\wp'_iq'/\wp''_iq''$, we get the $\rpoint$-arcs $\ell^*_{i}$ on $(\cals_\zg,\calm_\zg)$. The final simple coordinate $\zD^*_\zg$ is obtained from $\zD^*$ after replacing each $\ell^*$ by arcs $\ell^*_{i}$.} 
	\label{fig:coord.} 
\end{figure}

\begin{example}\label{ex:cutting}
See Figure \ref{fig:ex-cutting} for a concrete example of cutting a marked surface with a simple coordinate.
	\begin{figure}[ht]
		\begin{center}
\begin{tikzpicture}[>=stealth,scale=0.6]
				\draw[line width=1.8pt,fill=white] (0,0) circle (4cm);
				\draw[thick,fill=gray!50] (0,0) circle (0.8cm);
				\path (0:4) coordinate (b1)
				(70:4) coordinate (b2)
				(110:4) coordinate (b3)
				(135:4) coordinate (b4)
				(200:4) coordinate (b5)
				(90:4) coordinate (b6)
				(-90:4) coordinate (b7)
				(-70:4) coordinate (b8)
				(-50:4) coordinate (b9)
				(45:4) coordinate (b10);
				
				\draw[red,line width=1.3pt]plot [smooth,tension=0.6] coordinates {(0,-4) (-1.7,0) (-0.8,1.2) (0,0.8)};
				\draw[red,line width=1.3pt]plot [smooth,tension=0.6] coordinates {(0,-4) (1.7,0) (.8,1.2) (0,0.8)};
				\draw[black,line width=1.3pt]plot [smooth,tension=0.6] coordinates {(b8) (0.5,-2.7) (-1.5,0.2) (-0.9,1.6)(0.5,1.5)(1.3,0.5)(1,-.6)(0,-0.8)};			
				\draw [red, line width=1.5pt] (b6) to (0,0.8);
				\draw[cyan,line width=1.5pt,red] (b6) to[out=-120,in=-10](b4);
				\draw[cyan,line width=1.5pt,red] (b6) to[out=-60,in=-170](b10);
				\draw[cyan,line width=1.5pt,red] (b7) to[out=50,in=170](b9);
				\draw[red] (-1.5,2.5) node {$1$};
				\draw[red] (.2,2.5) node {$2$};
				\draw[red] (1.5,2.5) node {$3$};
				\draw[red] (1.5,-1.8) node {$4$};
				\draw[red] (-1.5,-1.8) node {$5$};
				\draw[red] (1.5,-2.7) node {$6$};
				\draw[black] (-2,0.9) node {$\gamma$};
				\draw[black] (0,-1.3) node {$p_{2}$};
				\draw[black] (1.8,-4.1) node {$p_{1}$};
	\draw[thick,fill=white] (b1) circle (0.15cm)
				(b2) circle (.15cm)
				(b3) circle (.15cm)
				(b5) circle (.15cm)
				(b8) circle (.15cm);

				\draw[thick,red, fill=red] (b4) circle (0.15cm)
				(b6) circle (.15cm)
				(b7) circle (.15cm)
				(b9) circle (.15cm)
				(b10) circle (.15cm);
				
				\draw[thick, red ,fill=red] (0,0.8) circle (0.15cm);
				\draw[thick, black ,fill=white] (0,-0.8) circle (0.15cm);

				\draw[line width=2.5pt,blue, dashed, ->] (4.5, 0) -- (9, 0) ;
				\draw[line width=2.5pt,blue,  ->] (1.5, -4.5) -- (4, -7) ;			
\end{tikzpicture}
\begin{tikzpicture}[>=stealth,scale=0.6]
				\draw[line width=1.8pt,fill=white] (0,0) circle (4cm);
				\draw[thick,fill=gray!50] (0,0) circle (0.8cm);

				\path (0:4) coordinate (b1)
				(70:4) coordinate (b2)
				(110:4) coordinate (b3)
				(135:4) coordinate (b4)
				(200:4) coordinate (b5)
				(90:4) coordinate (b6)
				(-90:4) coordinate (b7)
				(-70:4) coordinate (b8)
				(-50:4) coordinate (b9)
				(45:4) coordinate (b10);
				
				\draw[red,line width=1.3pt]plot [smooth,tension=0.6] coordinates {(0,-4) (-1.7,0) (-0.8,1.5) (0,0.8)};
				\draw[red,line width=1.3pt]plot [smooth,tension=0.6] coordinates {(0,-4) (2,-.5) (1.4,1.3) (0,0.8)};

				\draw [red, line width=1.5pt] (b6) to (0,0.8);
				\draw[cyan,line width=1.5pt,red] (b6) to[out=-120,in=-10](b4);
				\draw[cyan,line width=1.5pt,red] (b6) to[out=-60,in=-170](b10);
				\draw[cyan,line width=1.5pt,red] (b7) to[out=50,in=170](b9);

			\draw[black] (-2.6,0.5) node {$\gamma$};
			\draw[black] (-1,-0.6) node {$p_{2}'$};
				\draw[black] (0.7,-1) node {$p_{2}''$};
				\draw[black] (0.6,-4.5) node {$p_{1}'$};
				\draw[black] (1.6,-4.3) node {$p_{1}''$};
				\draw[black,line width=1.3pt, fill=gray!50]plot [smooth,tension=0.6] coordinates {(0.75,-3.9) (0,-2.5) (-2.3,0) (-0.9,2.35)(0.5,1.8)(1.6,0.5)(1.4,-1.3)(0.5,-1.8)(-0.5,-0.6)(0,-0.8)(0.5,-1.5)(1.3,-1)(1.3,0.5)(0.5,1.5)(-0.9,2)(-2,0)(0.5,-2.5) (b8)};
				
				\draw [gray!50, line width=4.5pt] (-0.5,-0.5) to (0,-0.8);
				\draw [gray!50, line width=4.5pt] (0.75,-3.9) to (1.4,-3.8);
	\draw[thick,fill=white] (b1) circle (0.15cm)
				(b2) circle (.15cm)
				(b3) circle (.15cm)
				(b5) circle (.15cm)
				(b8) circle (.15cm);

				\draw[thick,red, fill=red] (b4) circle (0.15cm)
				(b6) circle (.15cm)
				(b7) circle (.15cm)
				(b9) circle (.15cm)
				(b10) circle (.15cm);
				\draw[thick, red ,fill=red] (0,0.8) circle (0.15cm);
				\draw[thick, black ,fill=white] (0,-0.8) circle (0.15cm);
				\draw[thick, black ,fill=white] (-0.5,-0.6) circle (0.15cm);
				\draw[thick, black ,fill=white] (0.75,-3.9) circle (0.15cm);
				\draw[line width=2.5pt,blue, dashed, ->] (-1.5, -4.5) -- (-4, -7) ;
		\end{tikzpicture}
		\begin{tikzpicture}[>=stealth,scale=0.6]
				\draw[line width=1.8pt,fill=white] (0,0) circle (4cm);
			
				\path (0:4) coordinate (b1)
				(45:4) coordinate (b2)
				(90:4) coordinate (b3)
				(135:4) coordinate (b4)
				(180:4) coordinate (b5)
				(225:4) coordinate (b6)
				(270:4) coordinate (b7)
				(315:4) coordinate (b8)
				
				(22.5:4) coordinate (r1)
				(67.5:4) coordinate (r2)
				(112.5:4) coordinate (r3)
				(157.5:4) coordinate (r4)
				(202.5:4) coordinate (r5)
				(247.5:4) coordinate (r6)
				(292.5:4) coordinate (r7)
				(337.5:4) coordinate (r8);
				\draw[cyan,line width=1.5pt,red] (r3) to[out=-90,in=0](r4);
				\draw[cyan,line width=1.5pt,red] (r3) to[out=-45,in=-135](r2);
				\draw[cyan,line width=1.5pt,red] (r3) to[out=-60,in=60](r6);
				\draw[cyan,line width=1.5pt,red] (r6) to[out=90,in=-10](r5);
				\draw[cyan,line width=1.5pt,red] (r6) to[out=45,in=180](r8);
				\draw[cyan,line width=1.5pt,red] (r8) to[out=-160,in=90](r7);
				\draw[cyan,line width=1.5pt,red] (r8) to[out=135,in=-135](r1);
				\draw[cyan,line width=1.5pt,black] (b6) to[out=0,in=120](b7);
				\draw[cyan,line width=1.5pt,black] (b8) to[out=120,in=-150](b1);
				\draw[thick,fill=white] (b1) circle (0.15cm)
				(b2) circle (.15cm)
				(b3) circle (.15cm)
				(b4) circle (.15cm)
				(b5) circle (.15cm)
				(b6) circle (.15cm)
				(b7) circle (.15cm)
				(b8) circle (.15cm);

				\draw[thick,red, fill=red] (r1) circle (0.15cm)
				(r2) circle (0.15cm)
				(r3) circle (0.15cm)
				(r4) circle (0.15cm)
				(r5) circle (0.15cm)
				(r6) circle (0.15cm) 
				(r7) circle (0.15cm)
				(r8) circle (0.15cm);	
				
				\draw[red] (-2.3,1.7) node {$1$};
				\draw[red] (-0.5,2.1) node {$2$};
				\draw[red] (1,2.8) node {$3$};
				\draw[red] (1.5,-1.5) node {$4$};
				\draw[red] (-2.1,-1.8) node {$5$};
				\draw[red] (2.8,0.8) node {$6$};
				\draw[red] (1.5,-2.5) node {$7$};
				\draw[red] (-1.8,-4.2) node {$q'$};
				\draw[red] (4.2,-1.9) node {$q''$};
				
				\draw[black] (-1.5,-2.4) node {$\gamma_1$};
				\draw[black] (-1.2,-4.4) node {$\gamma'$};
				\draw[black] (2.5,-1) node {$\gamma_2$};
				\draw[black] (4.5,-1) node {$\gamma''$};
				\draw[black] (0,-4.6) node {$p_{2}'$};
				\draw[black] (3.4,-3.4) node {$p_{2}''$};
				\draw[black] (-3,-3.5) node {$p_{1}'$};
				\draw[black] (4.7,0) node {$p_{1}''$};
	\end{tikzpicture}
		\end{center}	
  \caption{
 An example of the cutting surface.} 
	\label{fig:ex-cutting}
	\end{figure}
\end{example}

\begin{proposition-definition}\label{prop-def-cut-surf}
Let $(\cals,\calm,\zD^*)$ be a marked surface with a simple coordinate, and let $\zg$ be a zigzag simple $\bpoint$-arc on the surface.
Then  $\zD^*_\zg$ is a simple coordinate on the marked surface $(\cals_\zg,\calm_\zg)$.
We call $(\cals_\zg,{\calm_\zg},{\zD^*_\zg})$ the \emph{cutting surface} of $(\cals,\calm,\zD^*)$ along $\zg$.
\end{proposition-definition}
\begin{proof}
Denote by $\zG=\{\varepsilon_1,\cdots,\varepsilon_r\}$ the set of segments obtained from truncating $\zg$ by using the arcs in $\zD^*$. 
Recall that $\mathcal{L}$ is the set of arcs $\ell^*$ in $\zD^*$ that intersect with $\zg$, and $\zg$ divides each $\ell^*$ into segments $\varsigma_1,\cdots,\varsigma_t$. 
Since the arcs in $\zD^*$ cut  $\cals$ into polygons and $\zg$ is simple, the segments in  $\zD^*\setminus \mathcal{L}\cup \{\varsigma_1,\cdots,\varsigma_t, \ell^*\in \mathcal{L}\}\cup \zG$ also cut $\cals$ into polygons.
Furthermore, there are two kinds of such polygons, one kind has no $\bpoint$-point, and another contains $\bpoint$-points. 
Note that $\zD^*_\zg$ is obtained by smoothing the segments in $\zD^*\setminus \mathcal{L}\cup \{\varsigma_1,\cdots,\varsigma_t, \ell^*\in \mathcal{L}\}\cup \zG$.
Therefore, to prove that $\zD^*_\zg$ is a simple coordinate on $(\cals_\zg,\calm_\zg)$, it is enough to show that the first kind of polygon is contractible (to an $\rpoint$-arc) when smoothing the segments, and the second kind of polygon will become a polygon with exactly one $\bpoint$-point from $\calm_\zg$.

Assume that we have a polygon $\bbp_\zg$ of the first kind. Since there is no $\bpoint$-point, it is surrounded by red segments $\varsigma_i$ and black segments $\varepsilon_j$ (rather than boundary segments with $\bpoint$-points), as depicted in the left picture of Figure \ref{figure:contra-gon}, where each $\varepsilon_j$ turns at a corner of $\zD^*$, since $\zg$ is zigzag with respect to $\zD^*$. Furthermore, if there are more than two red segments in $\bbp_\zg$, then it implies that there is a polygon of $\zD^*$ without a $\bpoint$-point, which contradicts the fact that $\zD^*$ is a simple coordinate. Therefore the polygon has at most two red segments, which must be a quadrangle as depicted in the middle picture of Figure \ref{figure:contra-gon} or a triangle formed by two segments $\varsigma_i$ and one segment $\varepsilon_j$ as depicted in the right picture of Figure \ref{figure:contra-gon}. For both cases, the polygon $\bbp_\zg$ contracts to an $\rpoint$-arc on $(\cals_\zg,\calm_\zg)$ after smoothing the segments.

Now assume that we have a polygon $\bbp_\zg$ of the second kind. Note that $\bbp_\zg$ must be part of some polygon of $\zD^*$, say $\bbp$. Let $p$ be the unique $\bpoint$-point in $\bbp$.
At first, we assume that $p$ is not an endpoint of $\zg$. 
Note that $\zg$ may pass through the polygon $\bbp$ several times, which cuts some $\rpoint$-arcs $\ell^*$ in $\bbp$ into segments $\varsigma_i\prime s$, see the left picture in Figure \ref{figure:keeppolygon}, which shows the case that $p$ belongs to the boundary of the surface. The argument is similar for the case when $p$ is a puncture.
At the same time, the $\rpoint$-arcs in $\bbp$ cut $\zg$ into segments, where we label the ones closest to the interior of $\bbp$ by $\varepsilon_j\prime s$.
Then $\bbp_\zg$ is formed by the arcs in $\bbp$ which do not intersect $\zg$, the pre-existing boundary segment that $p$ belongs to, and the segments $\varepsilon_j\prime s$, $\varsigma_i\prime s$, see the left picture of Figure \ref{figure:keeppolygon}.
Finally, after cutting the surface along $\zg$ and then smoothing $\varsigma_i\prime s$ along $\zg'/\zg''$, $\bbp_\zg$ induces a polygon $\widehat{\bbp}_\zg$ containing a unique $\bpoint$-point $\widehat{p}$, which is induced from $p$.
Recall that we add new $\rpoint$-point $q'/q''$ on $\zg'/\zg''$, and up to homotopy (concerning the arcs $\zg', \zg'' $ ), we can always assume that these points are at $\varepsilon_j\prime s$. See the polygon $\widehat{\bbp}_\zg$ in the right picture of Figure \ref{figure:keeppolygon}. In particular, note that the number of $\rpoint$-arcs in $\bbp$ coincides with the number of $\rpoint$-arcs in the induced polygon $\widehat{\bbp}_\zg$. 
 
     	\begin{figure}
		\begin{center}
		\begin{tikzpicture}[>=stealth,scale=.7]

   \draw[red!20,thick,fill=red!20] (0,0) circle (0.1);
    \draw[red!20,thick,fill=red!20] (4,4) circle (0.1);

\draw[red!20,dashed,line width=1.5pt] (4,4)to (-1,6)to(-3,3)to(0,0)to(4,4); 
\draw[red,line width=1.5pt]  (-1,6)to(-3,3); 
\draw[red,line width=1.5pt](1.07,1.07)to(3.2,3.2); 

\draw[red,line width=1.5pt](-.78,.8)to (-3,3); 
\draw[red,line width=1.5pt](2.3,4.7)to(-1,6); 

\draw[red,thick,fill=red] (-3,3) circle (0.1);
\draw[red,thick,fill=red] (-1,6) circle (0.1);
  
\draw[line width=1pt,bend left](5,2)to(2,6);   
\draw[line width=1pt,bend left](-2,.2)to(3,.5);   
       
       \node at (1,3) {$\bbp_\zg$};
       \node at (2.3,3.6) {$\varepsilon_j$};
        \node[red] at (-1.5,2) {$\varsigma_i$};
        \node[red] at (2,2.5) {$\varsigma_i$};
        \node[red] at (.5,5) {$\varsigma_i$};
        \node at (.5,1.4) {$\varepsilon_j$};
			\end{tikzpicture}
   \begin{tikzpicture}[>=stealth,scale=.7]
    \draw[red!20,thick,fill=red!20] (4,4) circle (0.1);

\draw[red!20,dashed,line width=1.5pt] (4,4)to (-1,6); 
\draw[red!20,dashed,line width=1.5pt] (0,0)to(4,4); 
\draw[red,line width=1.5pt](1.6,1.6)to(3.2,3.2); 

\draw[red,line width=1.5pt](2.3,4.7)to(0,5.6);

\draw[line width=1pt,bend left](5,2)to(2,6);   
\draw[line width=1pt,bend right](0,6.5)to(3,.5);   
       
       \node at (1.5,3.5) {$\bbp_\zg$};
       \node at (2.3,3.6) {$\varepsilon_j$};
        \node[red] at (2.5,2) {$\varsigma_i$};
        \node[red] at (1.3,5.4) {$\varsigma_i$};
        \node at (0,3.6) {$\varepsilon_j$};
\end{tikzpicture}
\begin{tikzpicture}[>=stealth,scale=.7]
\draw[red,thick,fill=red] (4,4) circle (0.1);

\draw[red,line width=1.5pt] (4,4)to (-1,6); 
\draw[red,line width=1.5pt] (0,0)to(4,4); 
  
\draw[line width=1pt,bend right](0,6.5)to(3,.5);   
       
       \node at (1.5,3.5) {$\bbp_\zg$};
        \node[red] at (2.5,2) {$\varsigma_i$};
        \node[red] at (1.3,5.4) {$\varsigma_i$};
        \node at (0,3.6) {$\varepsilon_j$};
			\end{tikzpicture}
\end{center}
\begin{center}
			\caption{The leftmost is a polygon $\bbp_\zg$ formed by $\rpoint$-arcs or segments in the set $\zD^*\setminus \mathcal{L}\cup \{\varsigma_1,\cdots,\varsigma_t, \ell^*\in \mathcal{L}\}\cup \{\varepsilon_1,\cdots,\varepsilon_r\}$, which contains no $\bpoint$-point. Then it must be of the form as depicted in the right two pictures, where in both cases the polygon can be contracted to an $\rpoint$-arc in the cutting surface.}\label{figure:contra-gon}
		\end{center}
	\end{figure}
 
\begin{figure}
\begin{center}
\begin{tikzpicture}[>=stealth,scale=.7]
\draw[red,line width=1.5pt] (1,0)--(3,1)--(3,3)--(1,4);
\draw[red,line width=1.5pt,dashed] (-1,4)--(1,4);
\draw[red,line width=1.5pt] (-1,0)--(-3,1)--(-3,3)--(-1,4);
\draw [ gray!60, line width=4pt] (-2,-.07)--(2,-.07);		
\draw[line width=2pt] (-2,0)--(2,0);		
     	\draw[line width=1pt,bend right](3.5,2)to(2.5,.2);    
     	\draw[line width=1pt,bend right](-4,2)to(-2,4);      
        \draw[line width=1pt,bend right](-3.8,1.7)to(-1.7,4.2);    
	\draw[red] (2.7,2.3) node {$\varsigma_i$};
 	\draw[red] (1.7,.7) node {$\varsigma_i$};
  	\draw[red] (-1.5,3.3) node {$\varsigma_i$};
 	\draw[red] (-2.7,1.5) node {$\varsigma_i$};
    \draw[] (2.3,1.4) node {$\varepsilon_j$};   
    \draw[] (-1.8,2.4) node {$\varepsilon_j$};   
     \draw[thick,fill=white] (0,0) circle (0.1);
	\draw (0,-.5) node {$p$};
	\node at (0,2) {$\bbp_\zg$};
	\draw[red,thick,fill=red] (1,0) circle (0.1);
				\draw[red,thick,fill=red] (-1,0) circle (0.1);
				\draw[red,thick,fill=red] (3,1) circle (0.1);
				\draw[red,thick,fill=red] (3,3) circle (0.1);
				\draw[red,thick,fill=red] (1,4) circle (0.1);
				\draw[red,thick,fill=red] (1,0) circle (0.1);
				\draw[red,thick,fill=red] (-3,1) circle (0.1);
				\draw[red,thick,fill=red] (-3,3) circle (0.1);
				\draw[red,thick,fill=red] (-1,4) circle (0.1);
    	\node at (6,0) {};
			\end{tikzpicture}
   \begin{tikzpicture}[>=stealth,scale=.7]
				\draw[red!20,line width=1.5pt] (1,0)--(3,1)--(3,3);
    			\draw[red,line width=1.5pt] (3,3)--(1,4);
				\draw[red,line width=1.5pt,dashed] (-1,4)--(1,4);
				\draw[red,line width=1.5pt] (-1,0)--(-3,1);
    			\draw[red!20,line width=1.5pt] (-3,1)--(-3,3)--(-1,4);
		\draw [ gray!60, line width=4pt] (-2,-.07)--(2,-.07);		
		\draw[line width=2pt] (-2,0)--(2,0);		
     	\draw[black!20,line width=1pt,bend right](3.5,2)to(2.5,.2);    
     	\draw[black!20,line width=1pt,bend right](-4,2)to(-2,4);      
        \draw[black!20,line width=1pt,bend right](-3.8,1.7)to(-1.7,4.2);
     \draw[thick,fill=white] (0,0) circle (0.1);
	\draw (0,-.5) node {$p=\widehat{p}$};
	\node at (0,2) {$\widehat{\bbp}_\zg$};
	\node at (-1.7,2.3) {$q'/q''$};
	\node at (2,1.5) {$q'/q''$};
			\draw[red,thick,fill=red] (1,0) circle (0.1);
				\draw[red,thick,fill=red] (-1,0) circle (0.1);
				\draw[red!20,fill=red!20] (3,1) circle (0.1);
				\draw[red,thick,fill=red] (3,3) circle (0.1);
				\draw[red,thick,fill=red] (1,4) circle (0.1);
				\draw[red,thick,fill=red] (1,0) circle (0.1);
				\draw[red,thick,fill=red] (-3,1) circle (0.1);
				\draw[red!20,fill=red!20] (-3,3) circle (0.1);
				\draw[red,thick,fill=red] (-1,4) circle (0.1);
				\draw[red,thick,fill=red] (-2.33,2.7) circle (0.1);
				\draw[red,thick,fill=red] (2.7,1.2) circle (0.1);
    		\draw[red,line width=1.5pt] (1,0)--(2.7,1.2)--(3,3);
          	\draw[red,line width=1.5pt] (-3,1)--(-2.33,2.7) --(-1,4);
			\end{tikzpicture}
		\end{center}
		\begin{center}
			\caption{The zigzag simple $\bpoint$-arc $\zg$ goes through a boundary polygon $\bbp$ of $\zD^*$ several times, which cuts some $\rpoint$-arcs in $\zD^*$ as segments, denoted by ${\varsigma_i}'$. At the same time, the $\rpoint$-arcs in  $\zD^*$ cut $\zg$ as segments ${\varepsilon_j}'$. By replacing the $\rpoint$-arcs in $\bbp$ that intersects $\zg$, we get a polygon $\bbp_\zg$, as depicted in the left picture. Then after smoothing, we get a polygon $\widehat{\bbp}_\zg$ on $\cals_\zg$ in the right picture.}\label{figure:keeppolygon}
		\end{center}
	\end{figure}
When $p$ is an endpoint of $\zg$, it splits as some new $\bpoint$-points: $p'_1,p'_2,p''_1,p''_2$, where some of them may coincide. A similar argument as used above shows that each such $\bpoint$-point induced by $p$ belongs to a polygon $\widehat{\bbp}_\zg$ of $\zD^*_\zg$, and it is the unique $\bpoint$-point that belongs to the polygon.
The proof is a case-by-case checking, depending on the position of $p$, on the boundary or in the interior. We explain the construction when $\zg$ is a loop, see the pictures in Figure \ref{fig:twocases}. 
\begin{figure} 
	\centering 
 \begin{tikzpicture}[>=stealth,scale=.5]
\draw[red,line width=1.5pt] (1,0)--(3,1)--(3,3)--(1,4)--(-1,4)--(-3,3)--(-3,1)--(-1,0)--(1,0);

\draw (1+.3,4+.3) node {\tiny$q_1$};
\draw (3+.5,3) node {\tiny$q_2$};
\draw (3+.5,1) node {\tiny$q_3$};
\draw (1+.3,0-.3) node {\tiny$q_4$};
\draw (-1-.3,0-.3) node {\tiny$q_5$};
\draw (-3-.5,1) node {\tiny$q_6$};
\draw (-3-.5,3) node {\tiny$q_7$};
\draw (-1-.3,4+.3) node {\tiny$q_8$};

\draw[line width=1pt](0,2)to(0,5);    
\draw[line width=1pt](0,2)to(3,0);

\draw[line width=1pt,bend right](-4,2)to(-2,4);      
\draw[line width=1pt,bend right](-3.8,1.7)to(-1.7,4.2);
\draw[thick,fill=white] (0,2) circle (0.1);
\draw (0,2-.5) node {\tiny$p$};
\draw (-1,2) node {$\bbp_\zg$};
\draw (1.2,2.5) node {$\bbp_\zg$};
\draw (1.5,2-.5) node {\tiny$\zg$};
\draw (.3,3) node {\tiny$\zg$};
\draw (-1.8,3) node {\tiny$\zg$};
\draw[red,thick,fill=red] (1,0) circle (0.1);
	
\draw[red,thick,fill=red] (-1,0) circle (0.1);

\draw[red,thick,fill=red] (3,1) circle (0.1);

\draw[red,thick,fill=red] (3,3) circle (0.1);

\draw[red,thick,fill=red] (1,4) circle (0.1);

\draw[red,thick,fill=red] (1,0) circle (0.1);

\draw[red,thick,fill=red] (-3,1) circle (0.1);

\draw[red,thick,fill=red] (-3,3) circle (0.1);

\draw[red,thick,fill=red] (-1,4) circle (0.1);

\end{tikzpicture}
\begin{tikzpicture}[>=stealth,scale=.5]
\draw[red,line width=1.5pt] (1,0)--(3,1)--(3,3)--(0,4);
\draw[red,line width=1.5pt] (-1,0)--(-3,1)--(-3,3)--(0,4);
\draw [ gray!60, line width=4pt] (-2,-.07)--(2,-.07);		
\draw[line width=2pt] (-2,0)--(2,0);		
\draw[thick,fill=white] (0,0) circle (0.1);
\draw[red,thick,fill=red] (1,0) circle (0.1);
	
\draw[red,thick,fill=red] (-1,0) circle (0.1);

\draw[red,thick,fill=red] (3,1) circle (0.1);

\draw[red,thick,fill=red] (3,3) circle (0.1);

\draw[red,thick,fill=red] (1,0) circle (0.1);

\draw[red,thick,fill=red] (-3,1) circle (0.1);

\draw[red,thick,fill=red] (-3,3) circle (0.1);

\draw[red,thick,fill=red] (0,4) circle (0.1);

\draw (0,-.5) node {\tiny$p'$};
\draw (3+.5,3) node {\tiny$q'$};
\draw (3+.5,1) node {\tiny$q_8$};
\draw (1+.3,0-.5) node {\tiny$q'$};
\draw (2.3,0) node {\tiny$\zg'$};
\draw (-1-.3,0-.5) node {\tiny$q'$};
\draw (-3-.5,1) node {\tiny$q_4$};
\draw (-3-.5,3) node {\tiny$q_5$};
\draw (0,4+.5) node {\tiny$q_6$};
\draw (0,2) node {$\widehat{\bbp}_\zg$};
\end{tikzpicture}
\begin{tikzpicture}[>=stealth,scale=.5]
\draw[red,line width=1.5pt] (1,0)--(2.6,2)--(0,4);
\draw[red,line width=1.5pt] (-1,0)--(-2.6,2)--(0,4);
\draw [ gray!60, line width=4pt] (-2,-.07)--(2,-.07);		
\draw[line width=2pt] (-2,0)--(2,0);		
\draw[thick,fill=white] (0,0) circle (0.1);
\draw[red,thick,fill=red] (1,0) circle (0.1);
	
\draw[red,thick,fill=red] (-1,0) circle (0.1);

\draw[red,thick,fill=red] (2.6,2) circle (0.1);

\draw[red,thick,fill=red] (1,0) circle (0.1);

\draw[red,thick,fill=red] (-2.6,2) circle (0.1);

\draw[red,thick,fill=red] (0,4) circle (0.1);
\draw (0,-.5) node {\tiny$p''$};
\draw (3.2,2) node {\tiny$q_3$};
\draw (1+.3,0-.5) node {\tiny$q''$};
\draw (2.5,0) node {\tiny$\zg''$};
\draw (-1-.3,0-.5) node {\tiny$q''$};
\draw (-3.2,2) node {\tiny$q_1$};
\draw (0,4+.5) node {\tiny$q_2$};
\draw (0,2) node {$\widehat{\bbp}_\zg$};
\end{tikzpicture} 
\begin{tikzpicture}[>=stealth,scale=.5]
\draw[red,line width=1.5pt] (1,0)--(3,1)--(3,3)--(1,4)--(-1,4)--(-3,3)--(-3,1)--(-1,0);
\draw [ gray!60, line width=4pt] (-2,-.07)--(2,-.07);	
\draw[line width=2pt] (-2,0)--(2,0);		

\draw (1+.3,4+.3) node {\tiny$q_1$};
\draw (3+.5,3) node {\tiny$q_2$};
\draw (3+.5,1) node {\tiny$q_3$};
\draw (1,0-.5) node {\tiny$q_4$};
\draw (-1,0-.5) node {\tiny$q_5$};
\draw (-3-.5,1) node {\tiny$q_6$};
\draw (-3-.5,3) node {\tiny$q_7$};
\draw (-1-.3,4+.3) node {\tiny$q_8$};
\draw (-1.8,3) node {\tiny$\zg$};
\draw (-1,2) node {$\bbp_\zg$};
\draw (1.5,2) node {$\bbp_\zg$};
\draw (.5,2.9) node {$\bbp_\zg$};
\draw[line width=1pt](0,0)to(0,5);    
\draw[bend left,line width=1pt](0,0)to(3.5,4.5);

\draw[line width=1pt,bend right](-4,2)to(-2,4);      
\draw[line width=1pt,bend right](-3.8,1.7)to(-1.7,4.2);
\draw[thick,fill=white] (0,0) circle (0.1);
\draw (0,0-.5) node {\tiny$p$};
\draw (.5,1) node {\tiny$\zg$};
\draw (-.3,1) node {\tiny$\zg$};
\draw[red,thick,fill=red] (1,0) circle (0.1);
	
\draw[red,thick,fill=red] (-1,0) circle (0.1);

\draw[red,thick,fill=red] (3,1) circle (0.1);

\draw[red,thick,fill=red] (3,3) circle (0.1);

\draw[red,thick,fill=red] (1,4) circle (0.1);

\draw[red,thick,fill=red] (1,0) circle (0.1);

\draw[red,thick,fill=red] (-3,1) circle (0.1);

\draw[red,thick,fill=red] (-3,3) circle (0.1);

\draw[red,thick,fill=red] (-1,4) circle (0.1);

\end{tikzpicture}
\begin{tikzpicture}[>=stealth,scale=.5]
\draw[red,line width=1.5pt] (1,0)--(2.6,2)--(0,4);
\draw[red,line width=1.5pt] (-1,0)--(-2.6,2)--(0,4);
\draw [ gray!60, line width=4pt] (-2,-.07)--(2,-.07);		
\draw[line width=2pt] (-2,0)--(2,0);		
\draw[thick,fill=white] (0,0) circle (0.1);
\draw[red,thick,fill=red] (1,0) circle (0.1);
	
\draw[red,thick,fill=red] (-1,0) circle (0.1);

\draw[red,thick,fill=red] (2.6,2) circle (0.1);

\draw[red,thick,fill=red] (1,0) circle (0.1);

\draw[red,thick,fill=red] (-2.6,2) circle (0.1);

\draw[red,thick,fill=red] (0,4) circle (0.1);
\draw (0,-.5) node {\tiny$p_1'$};
\draw (0,2) node {$\widehat{\bbp}_\zg$};
\draw (2,2) node {\tiny$q_8$};
\draw (2.5,0) node {\tiny$\zg'$};
\draw (1+.3,0-.5) node {\tiny$q'$};
\draw (-1-.3,0-.5) node {\tiny$q_5$};
\draw (-2,2) node {\tiny$q_6$};
\draw (0,4+.5) node {\tiny$q'$};
\end{tikzpicture} 
\begin{tikzpicture}[>=stealth,scale=.5]
\draw[red,line width=1.5pt] (1,0)--(2.3,3.7)--(-2.3,3.7)--(-1,0);
\draw [gray!60, line width=4pt] (-2,-.07)--(2,-.07);		
\draw[line width=2pt] (-2,0)--(2,0);		
\draw[thick,fill=white] (0,0) circle (0.1);
\draw[red,thick,fill=red] (1,0) circle (0.1);
	
\draw[red,thick,fill=red] (-1,0) circle (0.1);

\draw[red,thick,fill=red] (2.3,3.7) circle (0.1);

\draw[red,thick,fill=red] (1,0) circle (0.1);

\draw[red,thick,fill=red] (-2.3,3.7) circle (0.1);

\draw (-2.3,4.1) node {\tiny$q_2$};
\draw (2.3,4.1) node {\tiny$q_3$};
\draw (2.3,0) node {\tiny$\zg'$};
\draw (0,2) node {$\widehat{\bbp}_\zg$};
\draw (0,-.5) node {\tiny$p_2'$};
\draw (1+.3,0-.5) node {\tiny$q_4$};
\draw (-1-.3,0-.5) node {\tiny$q'$};
\end{tikzpicture} 
\begin{tikzpicture}[>=stealth,scale=.5]
\draw[red,line width=1.5pt] (1,0)--(0,4);
\draw[red,line width=1.5pt] (-1,0)--(0,4);
\draw [ gray!60, line width=4pt] (-2,-.07)--(2,-.07);		
\draw[line width=2pt] (-2,0)--(2,0);		
\draw[thick,fill=white] (0,0) circle (0.1);
\draw[red,thick,fill=red] (1,0) circle (0.1);
	
\draw[red,thick,fill=red] (-1,0) circle (0.1);
\draw[red,thick,fill=red] (1,0) circle (0.1);
\draw (0,1) node {$\widehat{\bbp}_\zg$};
\draw[red,thick,fill=red] (0,4) circle (0.1);
\draw (0,-.5) node {\tiny$p''$};
\draw (1+.3,0-.5) node {\tiny$q''$};
\draw (2.3,0) node {\tiny$\zg''$};
\draw (-1-.3,0-.5) node {\tiny$q''$};
\draw (0,4+.5) node {\tiny$q_1$};
\end{tikzpicture} 
\caption{When cutting the surface along a simple zigzag $\bpoint$-arc $\zg$ with two endpoints $p$ coincide, the polygon $\bbp$ of $\zD^*$ which contains $p$ becomes several polygons. The left-above picture is when $p$ is a puncture, where we get two polygons of $\zD^*_\zg$ each which contains the new $\bpoint$-point $p'$ or $p''$, see the right-above pictures. The left-below picture is when $p$ is a boundary marked point, where we get three polygons of $\zD^*_\zg$ each which contains the new $\bpoint$-point $p_1', p_2'$ or $ p''$, see the below-right pictures. } 
	\label{fig:twocases} 
\end{figure}

To sum up,  $\zD^*_\zg$ is a simple coordinate on the marked surface $(\cals_\zg,\calm_\zg)$.

\end{proof}

We have the following corollary derived from the proof of Proposition \ref{prop-def-cut-surf}.

\begin{corollary}\label{lem:twotypeofpolygons}
Each polygon $\widehat{\bbp}_\zg$ of cutting surface $(\cals_\zg,{\calm_\zg},{\zD^*_\zg})$ is induced from a polygon $\bbp$ of $(\cals,\calm,\zD^*)$.
Moreover, each oriented intersection $\widehat{\aaa}$ in $\widehat{\bbp}_\zg$ is induced from an oriented intersection $\aaa$ in $\bbp$, and this gives a bijection between the set of minimal oriented intersections of $\zD^*$ and the set of minimal oriented intersections of $\zD^*_\zg$.
\end{corollary}
\begin{proof}
It follows from the proof of Proposition \ref{prop-def-cut-surf} that each polygon of $\zD^*_\zg$ is induced from a polygon of $\zD^*$.
Now we prove the second part. Let $\aaa: \ell_1^*\rightarrow \ell^*_2$ be a minimal oriented intersection of $\zD^*$, which corresponds to a corner in a polygon $\bbp$ of $\zD^*$. 
If $\zg$ does not pass through the corner of $\aaa$, then $\aaa$ is still an oriented intersection of the induced polygon $\widehat{\bbp}_\zg$ on the cutting surface. Otherwise, denote by $\varepsilon_1,\cdots,\varepsilon_t$ the segments of $\zg$ between $\rpoint$-arcs $\ell_1^*$ and $\ell_2^*$, see Figure \ref{fig:inher-zigzag}. By the construction of the cutting surface, we may assume that the newly added $\rpoint$-point $q'$ or $q''$ is located on any segment of $\zg$, say $\varepsilon_t$. Then $\aaa$ induces an oriented intersection $\widehat{\aaa}$ based at $q'$ or $q''$, which is depicted as the dashed lines in Figure \ref{fig:inher-zigzag}.
It is straightforward to see that this gives a bijection as expected.
\begin{figure} 
	\centering 
\begin{tikzpicture}
[>=stealth,scale=.65]				

\draw[red!20,fill=red!20] (0,0) -- (-8,2) -- (-6,-2);
\draw[red,line width=1.5pt] (-16,4) -- (0,0) -- (-12,-4);
\draw [line width=1pt] (-2.5,2) to (-2,-2);
\draw [line width=1pt] (-5,3) to (-4,-3);
\draw [line width=1pt] (-9,4) to (-5.5,-3);

\draw[thick, red ,fill=red] (0,0) circle (0.12);
\draw[thick, red ,fill=red] (-7,0) circle (0.12);

\draw[red,dashed,line width=1.5pt,bend right](-9,-3)to(-7,0);

\draw[dashed,line width=1.5pt,red] (-14,3.5) to[out=-20,in=130](-7,0) ;

\draw[very thick,bend right,->](-7.5,.44)to(-7.2,-.89);
\node [] at (-8.2,0) {$\widehat{\aaa}$};	

\draw[very thick,bend right,->](-1.1,.24)to(-1,-.3);
\node [] at (-1.5,0) {$\aaa$};	
\draw (-2,-2.5) node {$\varepsilon_1$};
\draw (-3.6,-2.5) node {$\varepsilon_2$};
\draw (-5.3,-2.5) node {$\varepsilon_t$};
\draw (-6,0) node {${q'/q''}$};
\draw (-10,2.9) node[red] {$\ell^*_1$};
\draw (-10.2,-3.9) node[red] {$\ell^*_2$};

	\end{tikzpicture}
	\caption{
 Consider an oriented intersection $\aaa$ from an $\rpoint$-arc $\ell^*_1$ to an $\rpoint$-arc $\ell^*_2$ in a polygon of $\zD^*$. The red-dashed arcs are the arcs in $\zD^*_\zg$ which are induced from $\ell^*_1$ and $\ell^*_2$. Then $\aaa$ induces an oriented intersection $\widehat{\aaa}$ in a polygon of $\zD^*_\zg$ based at the newly added $\rpoint$-point $q'/q''$.} 
	\label{fig:inher-zigzag}
\end{figure}
\end{proof}

\begin{remark}\label{rem:redmap}
We will use $\widehat{\rpoint}$ to denote the map from the set of minimal oriented intersections on $\surf$ to the set of minimal oriented intersections on $(\cals_\zg,\calm_\zg,\zD^*_\zg)$.
\end{remark}

Denote by $\zg_1$ and $\zg_2$ the $\bpoint$-arcs on $(\cals_\zg,\calm_\zg)$ which form bigons with the new boundary segments $\zg'$ and $\zg''$ respectively, see Figure \ref{fig:boundarc} for an explanation. We have the following lemma, whose proof is direct, noticing that $q'$ and $q''$ are the unique $\rpoint$-points in the bigons formed by $\zg'$ and $ \zg_1$, $\zg''$ and $\zg_2$ respectively.

\begin{lemma}\label{lemma:corresp1.}
Both $\zg_1$ and $\zg_2$ are simple zigzag $\bpoint$-arcs on $(\cals_\zg,{\calm_\zg},{\zD^*_\zg})$. 
\end{lemma}

Let $\za$ be an $\bpoint$-arc
on $(\cals,\calm)$. If $\za$ and $\zg$ intersect in the interior of the surface, then it disappears when cutting the surface along $\zg$. If $\za=\zg$, then it induces two $\bpoint$-arcs $\zg_1$ and $\zg_2$ as described above. Now assume that $\za\neq \zg$ and has no interior intersections with $\zg$. Then $\za$ and $\zg$ share common endpoints or completely disjoint. For both cases, $\za$ induces an (unique) $\bpoint$-arc on $(\cals_\zg,\calm_\zg)$, which is denoted by $\widehat\za$. The associated map will be denoted by $\widehat{\bullet}$.
We introduce the following four sets of arcs:
$$\mathfrak{A}:=\{\text{$\bpoint$-arc $\za$ on $(\cals,\calm,\zD^*)$ which has no interior intersections with $\zg$}\}\setminus \{\zg\};$$ 
$$\mathfrak{Z}:=\{\text{zigzag $\bpoint$-arc $\za$ on $(\cals,\calm,\zD^*)$ which has no interior intersections with $\zg$}\}\setminus \{\zg\};$$
$$\mathfrak{A}_\zg:=\{\text{$\bpoint$-arc $\widehat{\za}$ on $(\cals_\zg,\calm_\zg,\zD_\zg^*)$}\}\setminus \{\zg_1,\zg_2\};$$ 
$$\mathfrak{Z}_\zg:=\{\text{zigzag $\bpoint$-arc $\widehat{\za}$ on $(\cals_\zg,\calm_\zg,\zD_\zg^*)$}\}\setminus \{\zg_1,\zg_2\}.$$ 

\begin{figure} 
	\centering 
\begin{tikzpicture}
[>=stealth,scale=.65]				

\draw[fill=red!25] (-5.5,0) -- (-2.5,0) -- (-4,-2);
\draw[fill=red!25] (0,0) -- (-2.5,0) -- (-1,2);
\draw[fill=red!25] (0,0) -- (2.5,0) -- (1,-2);
\draw[fill=red!25] (5.5,0) -- (2.5,0) -- (4,2);
\draw [line width=2pt] (-8,0) to (8,0);
\draw[line width=1.5pt,red ] (-7.8,1.2) to (-7,2) to (-4,-2)to (-4.75,-3)to (-7.8,-3);
				\draw[line width=1.5pt,red ] (-4,1.5) to (-4,-2) to (-1,2)to (1,-2)to (4,2)to (7,-2)to (7.8,-1.2);
				\draw[line width=1.5pt,red ] (-2,-2)to (-1,2)to (0,-2) ;
				\draw[line width=1.5pt,red ] (0,2)to (1,-2)to (2,2) ;
				\draw[line width=1.5pt,red] (4,-2)to (4,2)to (6,3)to (6,3)to (7,2) ;
	\draw [line width=2pt] (-8,0) to (8,0);
				\draw[line width=1pt,fill=white] (-8,0) circle (0.13);
				\draw[line width=1pt,fill=white] (8,0) circle (0.13);
				\draw[thick, red ,fill=red] (-7,2) circle (0.12);
				\draw[thick, red ,fill=red] (-4,-2) circle (0.12);
				\draw[thick, red ,fill=red] (-4.75,-3) circle (0.12);
				\draw[thick, red ,fill=red] (-1,2) circle (0.12);
			    \draw[thick, red ,fill=red] (1,-2) circle (0.12);
			    \draw[thick, red ,fill=red] (4,2) circle (0.12);
				\draw[thick, red ,fill=red] (7,-2) circle (0.12);
				\draw[thick, red ,fill=red] (6,3) circle (0.12);
				\draw (6.5,0.6) node {$\gamma$};
	\end{tikzpicture}
\begin{tikzpicture}
[>=stealth,scale=.65]				

\draw [line width=2pt] (-8,0.6) to (8,0.6);
\draw [line width=2pt] (-8,-0.6) to (8,-0.6);
\draw [line width=1.5pt,red] (2,3) to (0.5,0.6)to (3,3); 
\draw [line width=1.5pt,red] (-7.5,2) to (-6.5,3)to(0.5,0.6) to (4.5,3)to (6.15,4) to (7.8,3.2); 
\draw [line width=1.5pt,red] (-2,3) to (0.5,0.6)to (-3.5,3); 
\draw [line width=1.5pt,red] (-2.7,-3) to (-0.5,-0.6)to (-0.5,-3); 
\draw [line width=1.5pt,red] (1.2,-3) to (-0.5,-0.6)to (4.5,-3.5);
\draw [line width=1.5pt,red] (-7.5,-3.5) to (-5,-3.5)to(-4,-2.4) to (-0.5,-0.6)to (6,-2) to (7,-1.5);
\draw [ gray!60, line width=3pt] (-7.87,0.47) to (7.87,0.47);
\draw [ gray!60, line width=2.5pt] (-7.87,-0.474) to (7.87,-0.474);
				
\draw[thick, red ,fill=red] (0.5,0.6) circle (0.15);
\draw[thick, red ,fill=red] (-0.5,-0.6) circle (0.15);
\draw[thick, red ,fill=red] (4.5,3) circle (0.13);
\draw[thick, red ,fill=red] (6.15,4) circle (0.1);
\draw[thick, red ,fill=red] (-2,3) circle (0.13);
\draw[thick, red ,fill=red] (-6.5,3) circle (0.13);\draw[thick, red ,fill=red](1.2,-3) circle (0.13);
\draw[thick, red ,fill=red](-5,-3.5) circle (0.13);
\draw[thick, red ,fill=red](-4,-2.4) circle (0.13);
\draw[thick, red ,fill=red](6,-2)  circle (0.13);

\draw[cyan,line width=1pt,black] (-8,0.6) to[out=15,in=165](8,0.6);
\draw[cyan,line width=1pt,black] (-8,-0.6) to[out=-15,in=-165](8,-0.6);

\draw (4,1) node[black] {$\zg'$};
\draw (-4,-1) node[black] {$\zg''$};
\draw (4,1.8) node[black] {$\zg_1$};
\draw (-4,-1.8) node[black] {$\zg_2$};
\draw (-.45,0) node[black] {$q''$};
\draw (.45,0) node[black] {$q'$};

\draw[line width=1pt,fill=white] (-8,0.6) circle (0.13);
\draw[line width=1pt,fill=white] (8,0.6) circle (0.13);
\draw[line width=1pt,fill=white] (-8,-0.6) circle (0.13);
\draw[line width=1pt,fill=white] (8,-0.6) circle (0.13);			
\end{tikzpicture}
	\caption{An example to show the zigzag arcs  $\zg_1$ and $\zg_2$ that form a bigon with new boundary segments $\zg'$ and $\zg''$ respectively. This is a special case when the simple zigzag $\bpoint$-arc $\zg$ factors through the $\rpoint$-arcs in $\zD^*$ at most once.} 
	\label{fig:boundarc} 
\end{figure}

\begin{lemma}\label{lemma:corresp.}
The map $\widehat{\bullet}$ establishes an one-to-one correspondence from $\mathfrak{A}$ to $\mathfrak{A}_\zg$, as well as an one-to-one correspondence from $\mathfrak{Z}$ to $\mathfrak{Z}_\zg$. Furthermore, $\widehat{\za}$ is simple if and only if $\za$ is simple.
\end{lemma}
\begin{proof}
It can be directly seen from the definition of the cutting surface that $\widehat{\bullet}$ establishes a one-to-one correspondence from $\mathfrak{A}$ to $\mathfrak{A}_\zg$.
For the second bijection, it only needs to show the following two claims: (1) if $\za\in\mathfrak{Z}$, then $\widehat{\za}\in\mathfrak{Z}_\zg$; (2) if $\widehat{\za}\in \mathfrak{Z}_\zg$ for some $\za$ in $\mathfrak{A}$, then $\za\in \mathfrak{Z}$.

For the first claim, we follow a zigzag arc $\za$ in $\mathfrak{Z}$ and assume that $\za$ intersects $\ell_1^*$ and $\ell_2^*$ consecutively. Denote by $\aaa$ the oriented intersection from $\ell_1^*$ to $\ell_2^*$, and denote by $\varepsilon_1,\cdots,\varepsilon_t$ the segments of $\zg$ between $\ell_1^*$ and $\ell_2^*$, see Figure \ref{fig:inher-zigzag}. 
Since there is no interior intersection between $\za$ and $\zg$, $\za$ does not intersect each segment $\varepsilon_i$, and it goes through the regions cut out by $\varepsilon_i's$. 
If $\za$ goes through the polygons formed $\ell^*_1$, $\ell^*_2$ and $\varepsilon_i's$, that is, the red part in Figure \ref{fig:inher-zigzag}, then on the cutting surface, $\widehat{\za}$ pass through an $\rpoint$-arcs in $\zD^*_\zg$, which is obtained by contracting these polygons, as explained in the proof of Proposition-Definition \ref{prop-def-cut-surf}.
Otherwise, if $\za$ goes through the rest region, that is, the white part in Figure \ref{fig:inher-zigzag}, then $\widehat{\za}$ passes through the induced polygon along the oriented intersection $\widehat{\aaa}$ based at the newly added $\rpoint$-point $q'/q''$ on $(\cals_\zg,\calm_\zg)$.
Furthermore, since by corollary \ref{lem:twotypeofpolygons}, the map from $\aaa$ to $\widehat{\aaa}$ gives a bijection between the set of minimal oriented intersections of $\zD^*$ and the set of minimal oriented intersections of $\zD^*_\zg$, therefore with respect to $\zD^*_\zg$, $\widehat{\za}$ is still a zigzag $\bpoint$-arc.

Now let $\za$ be an $\bpoint$-arc in $\mathfrak{A}$ such that $\widehat{\za}$ is zigzag with respect to $\zD^*_\zg$. Suppose that $\za$ is not zigzag with respect to $\zD^*$. Then there is a polygon $\bbp$ of $\zD^*$ such that $\za$ passes through at least two corners simultaneously. Without loss of generality, we assume that $\za$ passes through two minimal intersections $\aaa: \ell^*_1\mapsto \ell^*_2$ and $\bbb:\ell^*_2\mapsto \ell^*_3$ in $\bbp$. In particular, it enters $\bbp$ through $\ell^*_1$ and leaves $\bbp$ through $\ell^*_3$. Since there is no interior intersection between $\za$ and $\zg$, $\aaa$ and $\bbb$ belong to the same polygon $\bbp_\zg$, and the induced minimal oriented intersections $\widehat{\aaa}$ and $\widehat{\bbb}$ belong to the same induced polygon $\widehat{\bbp}_\zg$ (check the pictures in Figure \ref{figure:keeppolygon} and \ref{fig:twocases}). 
Then $\widehat{\za}$ passes $\widehat{\bbp}_\zg$ through $\widehat{\aaa}$ and $\widehat{\bbb}$ simultaneously, which contradicts the assumption that $\widehat{\za}$ is zigzag with respect to $\zD^*_\zg$. Thus $\za$ is zigzag with respect to $\zD^*$.

The last statement is clear.
\end{proof}
%
%
%
%

Note that the map $\widehat{\bullet}$ can be extended to a map on oriented intersections between $\bpoint$-arcs. Let $\za$ and $\zb$ be two $\bpoint$-arcs in $\mathfrak{A}$ intersecting at a point $p$. We have three cases:

(1) If $p$ is in the interior of the surface, then it gives rise to two oriented intersections $\mathfrak{p}$ and $\mathfrak{p}'$ between $\za$ and $\zb$, which induce two oriented intersections between $\widehat{\za}$ and $\widehat{\zb}$, where we denote them by $\widehat{\mathfrak{p}}$ and $\widehat{\mathfrak{p}}'$ respectively.

(2) If $p$ is a $\bpoint$-point which differs from the endpoints of $\zg$, then it induces a unique $\bpoint$-point on the cutting surface. In this case the oriented intersection $\mathfrak{p}$ associated with $p$ induces a unique oriented intersection $\widehat{\mathfrak{p}}$ on the cutting surface.

(3) If $p$ is a $\bpoint$-point which coincides an endpoint of $\zg$, then it induces several $\bpoint$-point on the cutting surface.
Denote by $\mathfrak{p}$ the oriented intersection between $\za$ and $\zb$ associated to $p$, and by $\widehat{p}_\za$ and $\widehat{p}_\zb$ respectively the endpoints of $\widehat{\za}$ and $\widehat{\zb}$ which induced from $p$.
If $\widehat{p}_\za$ and $\widehat{p}_\zb$ are different on the cutting surface, then $\mathfrak{p}$ disappears. Otherwise, $\mathfrak{p}$ induces an oriented intersection between $\widehat{\za}$ and $\widehat{\zb}$, which we denote by $\widehat{\mathfrak{p}}$.

We have the following lemma, whose proof is straightforward.

\begin{lemma}\label{lem:indint}
Let $\widehat{\za}$ and $\widehat{\zb}$ be two $\bpoint$-arcs in $\mathfrak{A}_\zg$. Then an oriented intersection $\widehat{\mathfrak{p}}$ from $\widehat{\za}$ to $\widehat{\zb}$ on the cutting surface $(\cals_\zg,{\calm_\zg},{\zD^*_\zg})$ is induced from a unique oriented intersection $\mathfrak{p}$ from $\za$ to $\zb$ on $(\cals,\calm,\zD^*)$.
\end{lemma}

The lemma \ref{lemma:corresp.} ensures that we can inductively define a cutting surface along a set $\zG$ of zigzag simple $\bpoint$-arcs without interior intersections. More precisely, let $\zG=\{\zg^i, 0\leq i \leq t\}$ be a set of simple zigzag $\bpoint$-arcs on $(\cals,\calm,\zD^*)$, and assume that there is no interior intersection between any two arcs. For convenience, we denote $\zg^0$ by $\zg$.
By Lemma \ref{lemma:corresp.}, $\{\zg^i, 1\leq i \leq t\}$ induces a (unique) set of zigzag simple $\bpoint$-arcs on the cutting surface $(\cals_\zg,{\calm_\zg},{\zD^*_\zg})$. Then we can further define the cutting surface of $(\cals_\zg,{\calm_\zg},{\zD^*_\zg})$ along $\zg^i$, $1\leq i \leq t$, step by step. Note that the definition is independent of the choice of the order in which we cut the surface. We denote by $(\cals_\zG,{\calm_\zG},{\zD^*_\zG})$ the finally cutting surface.

\begin{lemma}\label{lem:key}
Under the notations above, the following statements hold:

(1) The arc $\zg$ is partial-tilting on $(\cals,\calm,\zD^*)$ if and only if the collection $\{\zg_1,\zg_2\}$ is partial-tilting on $(\cals_\zg,{\calm_\zg},{\zD^*_\zg})$.

(2) The collection of arcs $\zG$ is partial-tilting on $(\cals,\calm,\zD^*)$ if and only if the collection $\{\zg_1^i,\zg_2^i, 0\leq i \leq t\}$ is partial-tilting on $(\cals_\zG,{\calm_\zG},{\zD^*_\zG})$.
\end{lemma}
\begin{proof}
We only prove the first part of the statement, and then the second part can be proved inductively.
Assume that $\zg$ is partial-tilting on $(\cals,\calm,\zD^*)$ and let $\top$ be a tilting dissection containing $\zg$.
We show that $\top_\zg:=\{\widehat{\za}, \za\in \top\setminus \zg\}\cup \{\zg_1,\zg_2\}$ is a tilting dissection on $(\cals_\zg,{\calm_\zg},{\zD^*_\zg})$.
At first, by using Lemma \ref{lemma:corresp1.} and Lemma \ref{lemma:corresp.}, we see that $\top_\zg$ is a set of simple zigzag $\bpoint$-arcs on $(\cals_\zg,\calm_\zg)$, noticing that each arc in $\top\setminus \{\zg\}$ has no interior intersection with $\zg$, and thus satisfies the condition in Lemma \ref{lemma:corresp.}. 
Furthermore, $\top_\zg$ is an admissible dissection on  $(\cals_\zg,\calm_\zg)$, which can be derived directly by the construction of the cutting surface and the fact that $\top$ is an admissible dissection on $(\cals,\calm)$, noticing that in particular, the newly added $\rpoint$-points $q'/q''$ are contained in the bigons formed by $\zg', \zg_1$ and $\zg'',\zg_2$ respectively. 

Now we show that $\top_\zg$ is a tilting dissection on $(\cals_\zg,{\calm_\zg})$.
Recalling the definition of a tilting dissection given in Definition \ref{def:tilting-dissection}, we only need to show that the weight of any oriented intersection $\widehat{\mathfrak{p}}$ between any two arcs in $\top_\zg$ equals zero.
There are two cases. 

Case I: both the arcs differ from $\zg_1$ and $\zg_2$. Then they are of the form $\widehat{\za}$ and $\widehat{\zb}$, which are induced from $\bpoint$-arcs $\za$ and $\zb$ on $(\cals,{\calm})$ respectively.
Assume that $\widehat{\mathfrak{p}}$ is from $\widehat{\za}$ to $\widehat{\zb}$ which arises from a common endpoint $\widehat{p}$ of $\widehat{\za}$ and $\widehat{\zb}$.
By Lemma \ref{lem:indint}, $\widehat{p}$ is induced from a $\bpoint$-point $p$ on $(\cals,{\calm})$, and $\widehat{\mathfrak{p}}$ is induced from an oriented intersection $\mathfrak{p}$ from $\za$ to $\zb$ based at $p$.
Since $\omega(\mathfrak{p})=0$, starting from $p$, $\za$ and $\zb$ intersect the same  $\rpoint$-arc in $\zD^*$, recalling the definition of the weight given in Definition \ref{definition:weighted intersections1}.
Therefore starting from $\widehat{p}$, $\widehat{\za}$ and $\widehat{\zb}$ intersect the same $\rpoint$-arc in $\zD_\zg^*$, see the polygons $\widehat{\bbp}_\zg$ which contains $\widehat{p}$ in Figure \ref{figure:keeppolygon} and Figure \ref{fig:twocases}.
Thus $\omega(\widehat{\mathfrak{p}})=0$.


Case II: the starting arc of $\widehat{\mathfrak{p}}$ or the ending arc of $\widehat{\mathfrak{p}}$ coincides with $\zg_1$ or $\zg_2$. 
Without loss of generality, we assume that the starting arc and the ending arc of $\widehat{\mathfrak{p}}$ are $\zg_1$ and $\widehat{\za}$ respectively. Then similarly, we can lift $\widehat{\mathfrak{p}}$ to an oriented intersection $\mathfrak{p}$ from $\zg$ to $\za$. Furthermore, we have $\omega(\widehat{\mathfrak{p}})=0$ since $\omega(\widehat{\mathfrak{p}})=0$, comparing the pictures in Figure \ref{figure:keeppolygon} and Figure \ref{fig:twocases}.



For the other direction, a converse argument shows that any tilting dissection $\top_\zg$ on $(\cals_\zg,{\calm}_\zg,{\zD^*}_\zg)$ which contains $\zg_1$ and $\zg_2$ can be lifted as a tilting dissection $\top$ on $(\cals,{\calm},{\zD^*})$ by replacing each arc $\widehat{\za}$ by $\za$, and by replacing $\zg_1$ and $\zg_2$ by $\zg$.  
\end{proof}

Combining Lemma \ref{lemma:corresp.} and the proof of the above lemma, we have the following corollary.

\begin{corollary}\label{cor:key}
Let $\zG$ be a pre-tilting collection on $(\cals,\calm,\zD^*)$.
Then the map $\widehat{\bullet}$ induces a map from $\top$ to 
$\top_\zG:=\{\widehat{\za}, \za\in \top\}\cup \{\zg_1,\zg_2, \zg\in \zG\}$, which gives a one-to-one correspondence between the set of partial-tilting dissections on $(\cals,\calm,\zD^*)$ containing $\zG$ and the set of partial-tilting dissections on $(\cals_\zG,\calm_\zG,\zD_\zG^*)$ containing $\{\zg_1,\zg_2, \zg\in \zG\}$.
\end{corollary}

Now let's describe the algebra associated with the cutting surface. 
For convenience, denote by $A$ and $A_\zg$ the gentle algebras associated to $(\cals,\calm,\zD^*)$ and $(\cals_\zg,{\calm_\zg},{\zD^*_\zg})$ respectively, and denote by $Q=(Q_0,Q_1)$ and $\widehat{Q}=(\widehat{Q}_0,\widehat{Q}_1)$ the associated quivers.

\begin{example}\label{ex:alg-cutting}
   See Figure \ref{fig:alg-cutting} for the algebras $A$ and $A_\zg$ associated with the surface and the cutting surface given in Example \ref{ex:cutting}, c.f. Figure \ref{fig:ex-cutting}. 
\begin{figure}
   	\begin{center}
		{\begin{tikzpicture}[ar/.style={->,very thick,>=stealth}]
				\draw(-5,3)node(6){$6$}
				(-4,1.5)node(4){$4$}
				(-2,1.5)node(5){$5$}
				(-3,0)node(2){$2$}
				(-4,-1.5)node(3){$3$}
				(-2,-1.5)node(1){$1$}
				
				(2,3)node(7){$6$}
				(3,1.5)node(8){$4$}
				(5,1.5)node(9){$5$}
				(2,0)node(10){$7$}
				(4,0)node(11){$2$}
				(3,-1.5)node(12){$3$}
				(5,-1.5)node(13){$1$};

				\draw[ar](6) to  (4) ;
				\draw[ar](4) to (5);
				\draw[ar](4) to  (2);
				\draw[ar](2)to  (5);
				\draw[ar](2) to  (3);
				\draw[ar](1) to  (2);
				\draw[ar](2)to (3);
				\draw[ar](-1,0) to  (1,0);
				
				\draw[ar](7) to  (8) ;
				\draw[ar](8) to (10);
				\draw[ar](8) to  (11);
				\draw[ar](11)to  (9);
				\draw[ar](11) to  (12);
				\draw[ar](13) to  (11);
		
				\draw[bend right,dotted,thick]($(6)!.6!(4)-(0,0.1)$) to ($(4)!.25!(2)-(0,0.1)$);
				\draw[bend right,dotted,thick]($(4)!.55!(2)-(0,0.1)$) to ($(2)!.3!(3)-(0,0.05)$);
				\draw[bend left,dotted,thick]($(5)!.55!(2)-(0,0.1)$) to ($(2)!.3!(1)-(0,0.05)$);
				\draw[bend left,dotted,thick]($(7)!.6!(8)+(0,0.1)$) to ($(8)!.3!(11)+(0,0.1)$);
				\draw[bend right,dotted,thick]($(13)!.6!(11)$) to ($(11)!.4!(9)$);
				\draw[bend right,dotted,thick]($(8)!.55!(11)$) to ($(11)!.3!(12)$);
				
		\end{tikzpicture}}
	\end{center}
  \caption{
 The algebras $A$ and $A_\zg$ associated with the surface and the cutting surface appearing in Example \ref{ex:cutting}, see the Figure \ref{fig:ex-cutting}.} 
	\label{fig:alg-cutting} 
	\end{figure}
\end{example}


The following proposition shows some numerical relations of the associated quivers, where the equality $|\widehat{Q}_1|=|Q_1|$ also follows by Corollary \ref{lem:twotypeofpolygons}, noticing that the minimal oriented intersections in the associated simple coordinate give the arrows in the quiver.

\begin{proposition}\label{prop:rank}
	The rank of $(\cals_\zg,\calm_\zg)$ equals the rank of $(\cals,\calm)$ plus one. In particularly, we have $|\widehat{Q}_0|=|Q_0|+1$ and $|\widehat{Q}_1|=|Q_1|$.
\end{proposition}
\begin{proof}
Denote by $\chi$ the Euler character of $\cals$ and by $|\calm_\bpoint|$ the number of $\bpoint$-point on the boundary of the marked surface. Then the Euler character of the 
topological quotient of $\cals$ obtained by identifying all the points in the image of $\zg$ equals $\chi-1$, see for example the proof of \cite[Proposition 1.11\footnote{In the published version, the proof the result, that is \cite[Proposition 3.12]{APS23} is given differently.}]{APS19}. 
Note that when we define the cutting surface, instead of identifying the points in the image of $\zg$, we add two more $\rpoint$-points $q'/q''$ on the image $\zg'$ and $\zg''$ of $\zg$. Thus the Euler character of $\cals$ is the same as the Euler character of the topological quotient of $\cals$, that is, equals $\chi-1$, see the definition of the Euler character of a marked surface in subsection \ref{subsection:coord-diss}. On the other hand, the number of marked points always equals $|\calm_\bpoint|+2$, which can be proved case-by-case, seeing the pictures in Figure \ref{table:list}. 
Therefore we get the equalities in the statement, recalling from the subsection \ref{subsection:coord-diss} that the rank of the marked surface with a simple coordinate equals the number of the vertices in the associated quiver, which is $|\calm_{\bpoint}|-\chi$, and the number of the arrows in the quiver is $|\calm_{\bpoint}|-2\chi$. 
\end{proof}

In the appendix, we describe the algebra associated with the cutting surface for a special case when the $\bpoint$-arc $\zg$ intersects each $\rpoint$-arc in $\zD^*$ at most once. It seems interesting to give a general combinatorial characterization of the algebra obtained by cutting the surface.

\subsection{Almost-tilting modules are partial-tilting}\label{section:almost-tilting}
In this subsection, we consider when a pre-tilting module is partial-tilting, that is, when it can be completed as a tilting module.
Let $(\cals,\calm,\zD^*)$ be a marked surface with a simple coordinate, where the associate algebra is denoted by $A$.
Let $\zG$ be a set of simple zigzag $\bpoint$-arcs without interior intersection, and let $(\cals_\zG,{\calm_\zG},{\zD^*_\zG})$ be the cutting surface of  $(\cals,\calm,\zD^*)$ along the arcs in $\zG$, where we denote the associated algebra by $A_\zG$.

The following lemma can be directly derived from Proposition \ref{prop:tilting dissection} and Lemma \ref{lem:key}.

\begin{lemma}\label{lem:key2}
The module $\bigoplus_{\zg\in \zG}M_\zg$ is partial-tilting in $\ma$ if and only if the module $\bigoplus_{\zg\in\zG}M_{\zg_1}\oplus M_{\zg_2}$ is partial-tilting in $\ma_\zG$.
\end{lemma}

\begin{proposition}\label{prop:partialtilting}
If the algebra $A_\zG$ is representation-finite, then the module $\bigoplus_{\zg\in \zG}M_\zg$ is partial-tilting in $\ma$. In particular, this is the case when the arcs in $\zG$ cut the surface into disks.    
\end{proposition}
\begin{proof}
Note that any pre-tilting module over a representation-finite algebra is partial-tilting, see \cite{RS89}, then the first statement can be derived from Lemma \ref{lem:key2}. Since each gentle algebra arising from a disk is a tilted algebra of type $A$, thus it is representation-finite. Therefore the second statement holds.
\end{proof}
\begin{corollary}\label{cor:partialtilting}
If $\zG$ is a set of non-separating arcs with at least $1-\chi$ arcs, then the module $\bigoplus_{\zg\in \zG}M_\zg$ is partial-tilting in $\ma$.   
\end{corollary}
\begin{proof}
This follows from Proposition \ref{prop:partialtilting} and the fact that if the conditions in the statement hold, then the arcs in $\zG$ cut the surface into disks.  
\end{proof}
The following lemma is a fundamental step of the main theorem in this subsection.
\begin{lemma}\label{lem:keykey}
    Let $\surf$ be a connected marked surface with rank $n$. Assume that there are $n-1$ zigzag $\bpoint$-arcs $\{\zg_1,\cdots,\zg_{n-1}\}$ and $n-1$ boundary segments $\{\zg'_1,\cdots,\zg'_{n-1}\}$ such that for each $1\leq i \leq n-1$, 

    (1) $\zg_i$ and $\zg'_i$ form a bigon which contains one $\rpoint$-point;

    (2) $\bigoplus_{i=1}^{n-1}M_{\zg_i}$ is an almost-tilting module in $\ma$.

    Then $\surf$ is one of the cases depicted in Figure \ref{fig:twoposs2}. Furthermore, denote by $r$ the number of the choices of the $\bpoint$-arc $\zg_n$ such that $\bigoplus_{i=1}^{n}M_{\zg_i}$ is tilting in $\ma$, then $1\leq r\leq n+1$.
    
\begin{figure}
\begin{center}
\begin{tikzpicture}[>=stealth,scale=0.5]
\draw[line width=1.8pt,fill=white] (0,0) circle (4cm);
\path (0:4) coordinate (b1)
(45:4) coordinate (b2)			(90:4) coordinate (b3)
(135:4) coordinate (b4)
(180:4) coordinate (b5)
(225:4) coordinate (b6)
(270:4) coordinate (b7)
(315:4) coordinate (b8)
				
(22.5:4) coordinate (r1)
(67.5:4) coordinate (r2)
(112.5:4) coordinate (r3)
(157.5:4) coordinate (r4)
(202.5:4) coordinate (r5)(247.5:4) coordinate (r6)
(292.5:4) coordinate (r7)
(337.5:4) coordinate (r8);


\draw[very thick,bend left](r7)to(r8)to(r1)to(r2);
\draw[very thick,bend left](r3)to(r4)to(r5)to(r6);
\draw[very thick,bend left,red!60](b8)to(b1)to(b2);
\draw[very thick,bend left,red!60](b4)to(b5)to(b6);

\draw[thick,red,fill=red] (b1) circle (0.15cm)
	(b2) circle (.15cm)
	(b3) circle (.15cm)
	(b4) circle (.15cm)
	(b5) circle (.15cm)
	(b6) circle (.15cm)
	(b7) circle (.15cm)
	(b8) circle (.15cm);	
	
 \draw[thick,fill=white] (r1) circle (0.15cm)
	(r2) circle (0.15cm)
	(r3) circle (0.15cm)
	(r4) circle (0.15cm)
	(r5) circle (0.15cm)
	(r6) circle (0.15cm) 
	(r7) circle (0.15cm)
	(r8) circle (0.15cm);	
\node [] at (-2,-2){$\zg_i$};
\node [] at (-3.4,-3.4){$\zg'_i$};
	\end{tikzpicture}
\begin{tikzpicture}[>=stealth,scale=0.5]
\draw[line width=1.8pt,fill=white] (0,0) circle (4cm);
\path (0:4) coordinate (b1)
(45:4) coordinate (b2)			(90:4) coordinate (b3)
(135:4) coordinate (b4)
(180:4) coordinate (b5)
(225:4) coordinate (b6)
(270:4) coordinate (b7)
(315:4) coordinate (b8)
				
(22.5:4) coordinate (r1)
(67.5:4) coordinate (r2)
(112.5:4) coordinate (r3)
(157.5:4) coordinate (r4)
(202.5:4) coordinate (r5)(247.5:4) coordinate (r6)
(292.5:4) coordinate (r7)
(337.5:4) coordinate (r8);


\draw[very thick,bend left](r7)to(r8)to(r1)to(r2)to(r3)to(r4)to(r5)to(r6);
\draw[very thick,bend left,red!60](b8)to(b1)to(b2)to(b3)to(b4)to(b5)to(b6);

\draw[thick,red,fill=red] (b1) circle (0.15cm)
	(b2) circle (.15cm)
	(b3) circle (.15cm)
	(b4) circle (.15cm)
	(b5) circle (.15cm)
	(b6) circle (.15cm)
	(b7) circle (.15cm)
	(b8) circle (.15cm);	
	
\draw[thick,fill=white] (r1) circle (0.15cm)
    (0,0) circle (0.15cm)
	(r2) circle (0.15cm)
	(r3) circle (0.15cm)
	(r4) circle (0.15cm)
	(r5) circle (0.15cm)
	(r6) circle (0.15cm) 
	(r7) circle (0.15cm)
	(r8) circle (0.15cm);	
\node [] at (-2,-2){$\zg_i$};
\node [] at (-3.4,-3.4){$\zg'_i$};
\draw[] (4.8,0) {};
	\end{tikzpicture}
\begin{tikzpicture}[>=stealth,scale=0.5]
\draw[line width=1.8pt,fill=white] (0,0) circle (4cm);
\draw[line width=1.8pt,fill=gray!20] (0,0) circle (1.3cm);
\path (0:4) coordinate (b1)
(45:4) coordinate (b2)			(90:4) coordinate (b3)
(135:4) coordinate (b4)
(180:4) coordinate (b5)
(225:4) coordinate (b6)
(270:4) coordinate (b7)
(315:4) coordinate (b8)
				
(22.5:4) coordinate (r1)
(67.5:4) coordinate (r2)
(112.5:4) coordinate (r3)
(157.5:4) coordinate (r4)
(202.5:4) coordinate (r5)(247.5:4) coordinate (r6)
(292.5:4) coordinate (r7)
(337.5:4) coordinate (r8);

\path (0:1.3) coordinate (bb1)
(90:1.3) coordinate (bb2)		(180:1.3) coordinate (bb3)
(270:1.3) coordinate (bb4)
(45:1.3) coordinate (rr1)(135:1.3) coordinate (rr2)(225:1.3) coordinate (rr3)(315:1.3) coordinate (rr4);


\draw[very thick,bend left](r7)to(r8)to(r1)to(r2)to(r3)to(r4)to(r5)to(r6);

\draw[very thick,bend left,red!60](b8)to(b1)to(b2)to(b3)to(b4)to(b5)to(b6);

\draw[very thick]plot [smooth,tension=1] coordinates {(rr1)(0,1.7) (rr2)};
\draw[very thick]plot [smooth,tension=1] coordinates {(rr2)(-1.7,0) (rr3)};
\draw[very thick]plot [smooth,tension=1] coordinates { (rr3)(0,-1.7) (rr4)};
\draw[very thick]plot [smooth,tension=1] coordinates {(rr4)(1.7,0)(rr1)};
\draw[very thick,red!60]plot [smooth,tension=1] coordinates {(bb2)(1.4,1.4)(bb1)};
\draw[very thick,red!60]plot [smooth,tension=1] coordinates {(bb1)(1.4,-1.4)(bb4)};
\draw[very thick,red!60]plot [smooth,tension=1] coordinates {(bb4)(-1.4,-1.4)(bb3)};
\draw[very thick,red!60]plot [smooth,tension=1] coordinates {(bb3)(-1.4,1.4)(bb2)};
\draw[thick,red,fill=red] (b1) circle (0.15cm)
	(b2) circle (.15cm)
	(b3) circle (.15cm)
	(b4) circle (.15cm)
	(b5) circle (.15cm)
	(b6) circle (.15cm)
	(b7) circle (.15cm)
	(b8) circle (.15cm)
 (bb1) circle (.15cm)
 (bb2) circle (.15cm)
 (bb3) circle (.15cm)
 (bb4) circle (.15cm);	
	
 \draw[thick,fill=white] (r1) circle (0.15cm)
	(r2) circle (0.15cm)
	(r3) circle (0.15cm)
	(r4) circle (0.15cm)
	(r5) circle (0.15cm)
	(r6) circle (0.15cm) 
	(r7) circle (0.15cm)
	(r8) circle (0.15cm)
 (rr1) circle (.15cm)
 (rr2) circle (.15cm)
 (rr3) circle (.15cm)
 (rr4) circle (.15cm);	
\node [] at (-2,-2){$\zg_i$};
\node [] at (-3.4,-3.4){$\zg'_i$};
  \draw[] (4.8,0) {};
\end{tikzpicture}
		\end{center}
		\begin{center}
			\caption{
   The three kinds of surfaces $\surf$ satisfies conditions in Lemma \ref{lem:keykey}, where the arcs in $\zD^*$ can be partially determined by using the condition that $\bigoplus_{i=1}^{n-1}M_{\zg_i}$ is pre-tilting in $\ma$, see the $\rpoint$-arcs in the pictures.}\label{fig:twoposs2}
		\end{center}
	\end{figure}
    
\begin{figure}
\begin{center}
\begin{tikzpicture}[>=stealth,scale=.5]
\draw[bend left,line width=1.5pt] (1,0)to(3.5,1.3)to(3,4)to(1.5,6);
\draw[bend left,line width=1.5pt] (-1.5,6)to(-3,4)to(-3.5,1.3)to(-1,0);
\draw [ gray!60, line width=4pt] (-2,-.07)--(2,-.07);
\draw[line width=2pt] (-2,0)--(2,0);	
\draw [ gray!60, line width=4pt] (-2,6.07)--(2,6.07);
\draw[line width=2pt] (-2,6)--(2,6);	
    \draw[red,thick,fill=red] (0,6) circle (0.1);
    \draw[red,thick,fill=red] (0,0) circle (0.1);
\draw[thick,fill=white] (1,0) circle (0.1);
\draw[thick,fill=white] (3.5,1.3) circle (0.1);
\draw[thick,fill=white] (3,4) circle (0.1);
\draw[thick,fill=white] (1.5,6) circle (0.1);
\draw[thick,fill=white] (-1,0) circle (0.1);
\draw[thick,fill=white] (-3.5,1.3) circle (0.1);
\draw[thick,fill=white] (-3,4) circle (0.1);
\draw[thick,fill=white] (-1.5,6) circle (0.1);
\node [] at (-2.3,2.5){$\zg_i$};
	\node at (5.5,0) {};
\end{tikzpicture}
\begin{tikzpicture}[>=stealth,scale=.5]
\draw[bend left,line width=1.5pt] (1,0)to(3.5,1.3)to(3,4)to(1.5,6)to(-1.5,6)to(-3,4)to(-3.5,1.3)to(-1,0);
\draw [ gray!60, line width=4pt] (-2,-.07)--(2,-.07);
\draw[line width=2pt] (-2,0)--(2,0);		
    \draw[red,thick,fill=red] (0,0) circle (0.1);
\draw[thick,fill=white] (0,3) circle (0.1);\draw[thick,fill=white] (1,0) circle (0.1);
\draw[thick,fill=white] (3.5,1.3) circle (0.1);
\draw[thick,fill=white] (3,4) circle (0.1);
\draw[thick,fill=white] (1.5,6) circle (0.1);
\draw[thick,fill=white] (-1,0) circle (0.1);
\draw[thick,fill=white] (-3.5,1.3) circle (0.1);
\draw[thick,fill=white] (-3,4) circle (0.1);
\draw[thick,fill=white] (-1.5,6) circle (0.1);
\node [] at (-2.3,2.5){$\zg_i$};
	\node at (5.5,0) {};
\end{tikzpicture}
\begin{tikzpicture}[>=stealth,scale=.5]
\draw[bend left,line width=1.5pt] (1,0)to(3.5,1.3)to(3,4)to(1.5,6)to(-1.5,6)to(-3,4)to(-3.5,1.3)to(-1,0);
\draw[bend left,line width=1.5pt,fill=gray!20] (0,2)to(1.7,3)to(0,4)to(-1.7,3)to(0,2);
\draw [ gray!60, line width=4pt] (-2,-.07)--(2,-.07);
\draw[line width=2pt] (-2,0)--(2,0);		
    \draw[red,thick,fill=red] (0,0) circle (0.1);
\draw[thick,fill=white] (-1.7,3) circle (0.1);
\draw[thick,fill=white] (0,4) circle (0.1);
\draw[thick,fill=white] (1.7,3) circle (0.1);
\draw[thick,fill=white] (0,2) circle (0.1);\draw[thick,fill=white] (1,0) circle (0.1);
\draw[thick,fill=white] (3.5,1.3) circle (0.1);
\draw[thick,fill=white] (3,4) circle (0.1);
\draw[thick,fill=white] (1.5,6) circle (0.1);
\draw[thick,fill=white] (-1,0) circle (0.1);
\draw[thick,fill=white] (-3.5,1.3) circle (0.1);
\draw[thick,fill=white] (-3,4) circle (0.1);
\draw[thick,fill=white] (-1.5,6) circle (0.1);
\node [] at (-2.3,2.5){$\zg_i$};
\end{tikzpicture}
		\end{center}
		\begin{center}
			\caption{
  Three possibilities of the sub-surface formed by $\bpoint$-arcs $\zg_i$ and boundary segments, whose rank is one, that is, it only needs one more $\bpoint$-arc to cut it into polygons each which contains exactly one $\rpoint$-point.}\label{fig:twoposs}
		\end{center}
	\end{figure}

\begin{figure}
\begin{center}
\begin{tikzpicture}[>=stealth,scale=.7]
\draw[line width=1.8pt,fill=white] (0,0) circle (4cm);
\path (0:4) coordinate (b1)
(45:4) coordinate (b2)			(90:4) coordinate (b3)
(135:4) coordinate (b4)
(180:4) coordinate (b5)
(225:4) coordinate (b6)
(270:4) coordinate (b7)
(315:4) coordinate (b8)
				
(22.5:4) coordinate (r1)
(67.5:4) coordinate (r2)
(112.5:4) coordinate (r3)
(157.5:4) coordinate (r4)
(202.5:4) coordinate (r5)(247.5:4) coordinate (r6)
(292.5:4) coordinate (r7)
(337.5:4) coordinate (r8);


\draw[very thick,bend left](r7)to(r8)to(r1)to(r2);
\node [] at (2,-2.1){$\zg_1$};
\node [] at (2.9,0){$\zg_2$};
\node [] at (-1.8,-2){$\zg_{n-1}$};

\draw[very thick,bend left](r3)to(r4)to(r5)to(r6);
\draw[very thick,bend left,red](b8)to(b1)to(b2);
\draw[very thick,bend left,red](b4)to(b5)to(b6);
\node [red] at (2.5,-1.1){$\ell^*_{1,2}$};
\node [red] at (2.6,1.1){$\ell^*_{2,3}$};

\draw[thick,red,fill=red] (b1) circle (0.07cm)
	(b2) circle (.07cm)
	(b3) circle (.07cm)node[above]{$q_m$}
	(b4) circle (.07cm)
	(b5) circle (.07cm)
	(b6) circle (.07cm)node[below left]{$q_{n}$}
	(b7) circle (.07cm)node[below]{$q_{n+1}$}
	(b8) circle (.07cm)node[below right]{$q_1$};	
	
 \draw[thick,fill=white] (r1) circle (0.07cm)
	(r2) circle (0.07cm)node[above right]{$p_m$} 
	(r3) circle (0.07cm)
	(r4) circle (0.07cm)
	(r5) circle (0.07cm)
	(r6) circle (0.07cm)node[below left]{$p_{n+1}$} 
	(r7) circle (0.07cm)node[below right]{$p_1$}
	(r8) circle (0.07cm)node[below right]{$p_2$};	
	\end{tikzpicture}
		\end{center}
		\begin{center}
			\caption{
We label the $\rpoint$-points $q_i$, the $\bpoint$-points $p_i$, the $\bpoint$-arc $\zg_i$ and the $\rpoint$-arcs $\ell^*_{ij}$ for the case of a disk appearing in Lemma \ref{lem:keykey}. }\label{fig:twoposs3}
		\end{center}
	\end{figure}

 \begin{figure}
\begin{center}
\begin{tikzpicture}[>=stealth,scale=.7]
\draw[line width=1.8pt,fill=white] (0,0) circle (4cm);
\path (0:4) coordinate (b1)
(45:4) coordinate (b2)			(90:4) coordinate (b3)
(135:4) coordinate (b4)
(180:4) coordinate (b5)
(225:4) coordinate (b6)
(270:4) coordinate (b7)
(315:4) coordinate (b8)
				
(22.5:4) coordinate (r1)
(67.5:4) coordinate (r2)
(112.5:4) coordinate (r3)
(157.5:4) coordinate (r4)
(202.5:4) coordinate (r5)(247.5:4) coordinate (r6)
(292.5:4) coordinate (r7)
(337.5:4) coordinate (r8);


\draw[very thick,bend left](r7)to(r8)to(r1)to(r2);

\draw[dashed,line width=1.2pt,red](b3)to(b7);
\draw[dashed,line width=1.2pt,red,bend right](b1)to(b7);
\draw[dashed,line width=1.2pt,red,bend left](b5)to(b7);
\draw[dashed,line width=1.2pt,bend left](r5)to(r8);

\draw[very thick,bend left](r3)to(r4)to(r5)to(r6);
\draw[very thick,bend left,red](b8)to(b1)to(b2);
\draw[very thick,bend left,red](b4)to(b5)to(b6);

\draw[thick,red,fill=red] (b1) circle (0.07cm)
	(b2) circle (.07cm)
	(b3) circle (.07cm)node[above]{$q_m$}
	(b4) circle (.07cm)
	(b5) circle (.07cm)
	(b6) circle (.07cm)node[below left]{$q_{n}$}
	(b7) circle (.07cm)node[below]{$q_{n+1}$}
	(b8) circle (.07cm)node[below right]{$q_1$};	
	
 \draw[thick,fill=white] (r1) circle (0.07cm)
	(r2) circle (0.07cm)node[above right]{$p_m$} 
	(r3) circle (0.07cm)
	(r4) circle (0.07cm)
	(r5) circle (0.07cm)
	(r6) circle (0.07cm)node[below left]{$p_{n+1}$} 
	(r7) circle (0.07cm)node[below right]{$p_1$}
	(r8) circle (0.07cm)node[below right]{$p_2$};	
\node [] at (7,0){};
	\end{tikzpicture}
\begin{tikzpicture}[>=stealth,scale=.7]
\draw[line width=1.8pt,fill=white] (0,0) circle (4cm);
\path (0:4) coordinate (b1)
(45:4) coordinate (b2)			(90:4) coordinate (b3)
(135:4) coordinate (b4)
(180:4) coordinate (b5)
(225:4) coordinate (b6)
(270:4) coordinate (b7)
(315:4) coordinate (b8)
				
(22.5:4) coordinate (r1)
(67.5:4) coordinate (r2)
(112.5:4) coordinate (r3)
(157.5:4) coordinate (r4)
(202.5:4) coordinate (r5)(247.5:4) coordinate (r6)
(292.5:4) coordinate (r7)
(337.5:4) coordinate (r8);


\draw[very thick,bend left](r7)to(r8)to(r1)to(r2);

\draw[dashed,line width=1.2pt,red](b3)to(b7);
\draw[dashed,line width=1.2pt,red,bend left](b2)to(b3);
\draw[dashed,line width=1.2pt,red,bend left](b6)to(b7);
\draw[dashed,line width=1.2pt](r2)to(r6);

\draw[very thick,bend left](r3)to(r4)to(r5)to(r6);
\draw[very thick,bend left,red](b8)to(b1)to(b2);
\draw[very thick,bend left,red](b4)to(b5)to(b6);

\draw[thick,red,fill=red] (b1) circle (0.07cm)
	(b2) circle (.07cm)
	(b3) circle (.07cm)node[above]{$q_m$}
	(b4) circle (.07cm)
	(b5) circle (.07cm)
	(b6) circle (.07cm)node[below left]{$q_{n}$}
	(b7) circle (.07cm)node[below]{$q_{n+1}$}
	(b8) circle (.07cm)node[below right]{$q_1$};	
	
 \draw[thick,fill=white] (r1) circle (0.07cm)
	(r2) circle (0.07cm)node[above right]{$p_m$} 
	(r3) circle (0.07cm)
	(r4) circle (0.07cm)
	(r5) circle (0.07cm)
	(r6) circle (0.07cm)node[below left]{$p_{n+1}$} 
	(r7) circle (0.07cm)node[below right]{$p_1$}
	(r8) circle (0.07cm)node[below right]{$p_2$};	
\end{tikzpicture}
\begin{tikzpicture}[>=stealth,scale=.7]
\draw[line width=1.8pt,fill=white] (0,0) circle (4cm);
\path (0:4) coordinate (b1)
(45:4) coordinate (b2)			(90:4) coordinate (b3)
(135:4) coordinate (b4)
(180:4) coordinate (b5)
(225:4) coordinate (b6)
(270:4) coordinate (b7)
(315:4) coordinate (b8)
				
(22.5:4) coordinate (r1)
(67.5:4) coordinate (r2)
(112.5:4) coordinate (r3)
(157.5:4) coordinate (r4)
(202.5:4) coordinate (r5)(247.5:4) coordinate (r6)
(292.5:4) coordinate (r7)
(337.5:4) coordinate (r8);


\draw[very thick,bend left](r7)to(r8)to(r1)to(r2);

\draw[dashed,line width=1.2pt,red,bend left](b2)to(b3);
\draw[dashed,line width=1.2pt,red,bend left](b7)to(b8);
\draw[dashed,line width=1.2pt,red,bend left](b6)to(b7);
\draw[dashed,bend right,line width=1.2pt](r7)to(r6);

\draw[very thick,bend left](r3)to(r4)to(r5)to(r6);
\draw[very thick,bend left,red](b8)to(b1)to(b2);
\draw[very thick,bend left,red](b4)to(b5)to(b6);

\draw[thick,red,fill=red] (b1) circle (0.07cm)
	(b2) circle (.07cm)
	(b3) circle (.07cm)node[above]{$q_m$}
	(b4) circle (.07cm)
	(b5) circle (.07cm)
	(b6) circle (.07cm)node[below left]{$q_{n}$}
	(b7) circle (.07cm)node[below]{$q_{n+1}$}
	(b8) circle (.07cm)node[below right]{$q_1$};	
	
 \draw[thick,fill=white] (r1) circle (0.07cm)
	(r2) circle (0.07cm)node[above right]{$p_m$} 
	(r3) circle (0.07cm)
	(r4) circle (0.07cm)
	(r5) circle (0.07cm)
	(r6) circle (0.07cm)node[below left]{$p_{n+1}$} 
	(r7) circle (0.07cm)node[below right]{$p_1$}
	(r8) circle (0.07cm)node[below right]{$p_2$};	
\node [] at (6,0){};
	\end{tikzpicture}
\begin{tikzpicture}
[>=stealth,scale=.7]
\draw[line width=1.8pt,fill=white] (0,0) circle (4cm);
\path (0:4) coordinate (b1)
(45:4) coordinate (b2)			(90:4) coordinate (b3)
(135:4) coordinate (b4)
(180:4) coordinate (b5)
(225:4) coordinate (b6)
(270:4) coordinate (b7)
(315:4) coordinate (b8)
				
(22.5:4) coordinate (r1)
(67.5:4) coordinate (r2)
(112.5:4) coordinate (r3)
(157.5:4) coordinate (r4)
(202.5:4) coordinate (r5)(247.5:4) coordinate (r6)
(292.5:4) coordinate (r7)
(337.5:4) coordinate (r8);


\draw[very thick,bend left](r7)to(r8)to(r1)to(r2)to(r3)to(r4)to(r5);
\draw[very thick,bend left,dashed](r3)to(r6);
\draw[very thick,bend left,dashed](r4)to(r6);
\draw[very thick,bend left,dashed](r5)to(r6);
\draw[very thick,dashed,bend left](r2)to(r6);
\draw[dashed,very thick,bend left,red](b5)to(b6);
\draw[very thick,bend left,red](b8)to(b1)to(b2)to(b3)to(b4)to(b5);
\draw[line width=1.2pt,dashed,red](b3)to(b7);
\node [red] at (.9,-2.7){$\ell^*_{j,n+1}$};
\draw[thick,red,fill=red] (b1) circle (0.07cm)
	(b2) circle (.07cm)
	(b3) circle (.07cm)node[above]{$q_j$}
	(b4) circle (.07cm)
	(b5) circle (.07cm)
	(b6) circle (.07cm)node[below left]{$q_{n}=q_m$}
	(b7) circle (.07cm)node[below]{$q_{n+1}$}
	(b8) circle (.07cm)node[below right]{$q_1$};	
	
 \draw[thick,fill=white] (r1) circle (0.07cm)
	(r2) circle (0.07cm)
	(r3) circle (0.07cm)
	(r4) circle (0.07cm)
	(r5) circle (0.07cm)
	(r6) circle (0.07cm)node[below left]{$p_{n+1}$} 
	(r7) circle (0.07cm)node[below right]{$p_1$}
	(r8) circle (0.07cm)node[below right]{$p_2$};	
	\end{tikzpicture}
		\end{center}
		\begin{center}
			\caption{
By adding the dashed $\rpoint$-arcs, we complete the $\rpoint$-arcs in the disk of Figure \ref{fig:twoposs3} to get the simple coordinate $\zD^*$.
If $m\neq n$ and $n\geq 5$, then we need three more $\rpoint$-arcs. Depending on these arcs' positions, there are several cases, where we list three representatives, see the first three pictures. For these cases, there is only one choice of $\bpoint$-arc $\zg_n$ such that $\bigoplus_{i=1}^{n}M_{\zg_i}$ is a tilting module in $\ma$, where $\zg_n$ is the dashed $\bpoint$-arc.
When $m=n$, there are two undetermined $\rpoint$-arcs, where at least one of them has $q_{n+1}$ as an endpoint, see $\ell^*_{j,n+1}\in \zD^*$ in the last picture. If we choose $\ell^*_{n-1,n}$ as the other $\rpoint$-arc, then there are $r=n-j+1$ choices of $\bpoint$-arc $\zg_n$ such that $\bigoplus_{i=1}^{n}M_{\zg_i}$ is a tilting module in $\ma$, see the dashed $\bpoint$-arcs.
}\label{fig:twoposs4}
		\end{center}
	\end{figure}

\end{lemma}
\begin{proof}
The case for $n=1$ is trivial, so we assume $n\geq 2$ in the following.
Since $\bigoplus_{i=1}^{n-1}M_{\zg_i}$ is a pre-tilting module, by Proposition \ref{prop:tilting dissection}, $\zG:=\{\zg_1,\cdots,\zg_{n-1}\}$ is an admissible collection, which can be completed as an admissible dissection by adding exactly one more $\bpoint$-arc.
Namely, the sub-surfaces of $(\cals,\calm)$ formed by arcs in $\zG$ and boundary segments are polygons each containing exactly one $\rpoint$-point except for one sub-surface. 
Furthermore, since it is needed one more $\bpoint$-arc to cut the exceptional sub-surface into polygons each of which has exactly one $\rpoint$-point, the rank of the sub-surface is one.
Thus by Lemma \ref{lamma:rankonesur}, there are three possibilities for this special sub-surface: a disk, a once-punctured disk, and an annulus, see Figure \ref{fig:twoposs}. Note that the annulus of Figure \ref{fig:twoposs} is homotopic to a once-punctured disk, if we view the boundary formed by $\bpoint$-arcs as a $\bpoint$-puncture. 

Since each $\bpoint$-arc $\zg_i$ and the boundary segment $\zg'_i$ form a bigon on $(\cals,\calm,\zD^*)$ which has exactly one $\rpoint$-point, the shape of the surface $(\cals,\calm)$ must be of the form depicted in Figure \ref{fig:twoposs2}.
At last, since the module $\bigoplus_{i=1}^{n-1}M_{\zg_i}$ is pre-tilting, the weights of the intersections between two arcs in $\{\zg_1,\cdots,\zg_{n-1}\}$ are always zero, equivalently, when starting from a common endpoint, two $\bpoint$-arcs $\zg_i$ intersect the same $\rpoint$-arc immediately. Therefore part of $\rpoint$-arcs in $\zD^*$ are determined, see the $\rpoint$-arcs in Figure \ref{fig:twoposs2}.

We consider the case of the disk, where we label the $\rpoint$-points by $q_i$ anti-clockwisely, and label the $\bpoint$-point in between $q_{i-1}$ and $q_i$ by $p_i$, see Figure \ref{fig:twoposs3}, where the indices are considered up to $n+1$.
Note that there are exactly two $\rpoint$-points that are not in the bigon formed by $\zg_i$ and $\zg'_i$.
Without loss of generality, 
we assume they are $q_m$ and $q_{n+1}$, where $2\leq m \leq n$. The $\rpoint$-arc with endpoints $q_i$ and $q_j$ will be denote by $\ell^*_{i,j}$. Under the notation, $\ell^*_{i,i+1}$ are the $\rpoint$-arcs in $\zD^*$ that are already determined, where $i\in [1,m-1]\cup[m+1,n-1]$. Since $q_{n+1}$ is isolated, there is at least one $\rpoint$-arc in $\zD^*$ connecting it, which is denoted by $\ell^*_{j,n+1}$.
We have the following claim: 

If $m\neq n$, then there exists exactly one choice of $\bpoint$-arc $\zg_n$ such that $\bigoplus_{i=1}^{n}M_{\zg_i}$ is a tilting module in $\ma$; if $m=n$ then there are $r$ choices of $\bpoint$-arc $\zg_n$ such that $\bigoplus_{i=1}^{n}M_{\zg_i}$ is a tilting module in $\ma$, where $1\leq r \leq n$. 

The proof of the claim is a straightforward check. At first, we complete the simple coordinate $\zD^*$ in Figure \ref{fig:twoposs3} by adding more $\rpoint$-arcs, and then find $\zg_n$ by using the following two properties: $\zg_n$ is zigzag with respect to $\zD^*$; the weight of any oriented intersection of $\zg_n$ with $\zg_i, 1\leq i \leq n-1$, is zero. 
We list four representatives of them in Figure \ref{fig:twoposs4}, where the first three pictures explain the case $m\neq n$ and the last picture explains the case $m=n$. 

A similar argument shows that $1\leq r \leq n$ for the other two kinds of surfaces: a once-punctured disk and an annulus.
We list four representatives in Figure \ref{fig:twoposs5}.
The only exception is the annulus that has only one $\bpoint$-point on one of the two boundary components, and this $\bpoint$-point does not belong to a bigon formed by $\zg'_i$ and $\zg_i$, see in Figure \ref{fig:upbound}. In this case $r=n+1$.

To sum up, we have shown that for the almost-tilting module $\bigoplus_{i=1}^{n}M_{\zg_i}$ given in the statement, there is at least one complement, and the number of the complements will not exceed the rank of the surface plus one.

\begin{figure}
\begin{center}
\begin{tikzpicture}[>=stealth,scale=0.7]
\draw[line width=1.8pt,fill=white] (0,0) circle (4cm);
\path (0:4) coordinate (b1)
(45:4) coordinate (b2)			(90:4) coordinate (b3)
(135:4) coordinate (b4)
(180:4) coordinate (b5)
(225:4) coordinate (b6)
(270:4) coordinate (b7)
(315:4) coordinate (b8)
				
(22.5:4) coordinate (r1)
(67.5:4) coordinate (r2)
(112.5:4) coordinate (r3)
(157.5:4) coordinate (r4)
(202.5:4) coordinate (r5)(247.5:4) coordinate (r6)
(292.5:4) coordinate (r7)
(337.5:4) coordinate (r8);


\draw[very thick,bend left](r7)to(r8)to(r1)to(r2)to(r3)to(r4)to(r5)to(r6);
\draw[very thick,bend left,red!60](b8)to(b1)to(b2)to(b3)to(b4)to(b5)to(b6);
\draw[very thick,bend left,dashed,red!60](b6)to(b7)to(b8);
\draw[very thick,bend left,dashed](r6)to(r7);
\draw[thick,red,fill=red] (b1) circle (0.1cm)
	(b2) circle (.1cm)
	(b3) circle (.1cm)
	(b4) circle (.1cm)
	(b5) circle (.1cm)
	(b6) circle (.1cm)
	(b7) circle (.1cm)
	(b8) circle (.1cm);	
	
\draw[thick,fill=white] (r1) circle (0.1cm)
    (0,0) circle (0.1cm)
	(r2) circle (0.1cm)
	(r3) circle (0.1cm)
	(r4) circle (0.1cm)
	(r5) circle (0.1cm)
	(r6) circle (0.1cm) 
	(r7) circle (0.1cm)
	(r8) circle (0.1cm);	
 \draw[] (5.4,0) {};
	\end{tikzpicture}
 \begin{tikzpicture}[>=stealth,scale=0.7]
\draw[line width=1.8pt,fill=white] (0,0) circle (4cm);
\path (0:4) coordinate (b1)
(45:4) coordinate (b2)			(90:4) coordinate (b3)
(135:4) coordinate (b4)
(180:4) coordinate (b5)
(225:4) coordinate (b6)
(270:4) coordinate (b7)
(315:4) coordinate (b8)
				
(22.5:4) coordinate (r1)
(67.5:4) coordinate (r2)
(112.5:4) coordinate (r3)
(157.5:4) coordinate (r4)
(202.5:4) coordinate (r5)(247.5:4) coordinate (r6)
(292.5:4) coordinate (r7)
(337.5:4) coordinate (r8);


\draw[very thick,bend left](r7)to(r8)to(r1)to(r2)to(r3)to(r4)to(r5)to(r6);
\draw[very thick,bend left,red!60](b8)to(b1)to(b2)to(b3)to(b4)to(b5)to(b6);
\draw[very thick,bend left,dashed,red!60](b3)to(b7);

\draw[very thick,red!60,dashed]plot [smooth,tension=1] coordinates {(0,4)(-.8,-.5)(.8,-.5)(0,4)};

\draw[very thick,dashed]plot [smooth,tension=1] coordinates {(r2)(-.3,0)(.2,-.3)(r2)};
\draw[very thick,dashed]plot [smooth,tension=1] coordinates {(r3)(.3,0)(-.2,-.3)(r3)};
\draw[thick,red,fill=red] (b1) circle (0.1cm)
	(b2) circle (.1cm)
	(b3) circle (.1cm)
	(b4) circle (.1cm)
	(b5) circle (.1cm)
	(b6) circle (.1cm)
	(b7) circle (.1cm)
	(b8) circle (.1cm);	
	
\draw[thick,fill=white] (r1) circle (0.1cm)
    (0,0) circle (0.1cm)
	(r2) circle (0.1cm)
	(r3) circle (0.1cm)
	(r4) circle (0.1cm)
	(r5) circle (0.1cm)
	(r6) circle (0.1cm) 
	(r7) circle (0.1cm)
	(r8) circle (0.1cm);	
	\end{tikzpicture}
\begin{tikzpicture}[>=stealth,scale=0.7]
\draw[line width=1.8pt,fill=white] (0,0) circle (4cm);
\draw[line width=1.8pt,fill=gray!20] (0,0) circle (1.3cm);
\path (0:4) coordinate (b1)
(45:4) coordinate (b2)			(90:4) coordinate (b3)
(135:4) coordinate (b4)
(180:4) coordinate (b5)
(225:4) coordinate (b6)
(270:4) coordinate (b7)
(315:4) coordinate (b8)
				
(22.5:4) coordinate (r1)
(67.5:4) coordinate (r2)
(112.5:4) coordinate (r3)
(157.5:4) coordinate (r4)
(202.5:4) coordinate (r5)(247.5:4) coordinate (r6)
(292.5:4) coordinate (r7)
(337.5:4) coordinate (r8);

\path (0:1.3) coordinate (bb1)
(90:1.3) coordinate (bb2)		(180:1.3) coordinate (bb3)
(270:1.3) coordinate (bb4)
(45:1.3) coordinate (rr1)(135:1.3) coordinate (rr2)(225:1.3) coordinate (rr3)(315:1.3) coordinate (rr4);


\draw[very thick,bend left](r7)to(r8)to(r1)to(r2)to(r3)to(r4)to(r5)to(r6);

\draw[very thick,bend left,red!60](b8)to(b1)to(b2)to(b3)to(b4)to(b5)to(b6);
\draw[very thick,dashed,red!60](b3)to(bb2);
\draw[very thick,dashed,red!60](b7)to(bb4);

\draw[very thick]plot [smooth,tension=1] coordinates {(rr1)(0,1.7) (rr2)};
\draw[very thick]plot [smooth,tension=1] coordinates {(rr2)(-1.7,0) (rr3)};
\draw[very thick]plot [smooth,tension=1] coordinates { (rr3)(0,-1.7) (rr4)};
\draw[very thick]plot [smooth,tension=1] coordinates {(rr4)(1.7,0)(rr1)};
\draw[very thick,red!60]plot [smooth,tension=1] coordinates {(bb2)(1.4,1.4)(bb1)};
\draw[very thick,red!60]plot [smooth,tension=1] coordinates {(bb1)(1.4,-1.4)(bb4)};
\draw[very thick,red!60]plot [smooth,tension=1] coordinates {(bb4)(-1.4,-1.4)(bb3)};
\draw[very thick,red!60]plot [smooth,tension=1] coordinates {(bb3)(-1.4,1.4)(bb2)};

\draw[very thick,dashed]plot [smooth,tension=1] coordinates {(r2)(1,3)(-.5,2)(rr2)};

\draw[very thick,dashed]plot [smooth,tension=1] coordinates {(r3)(-1,3)(.5,2)(rr1)};

\draw[thick,red,fill=red] (b1) circle (0.1cm)
	(b2) circle (.1cm)
	(b3) circle (.1cm)
	(b4) circle (.1cm)
	(b5) circle (.1cm)
	(b6) circle (.1cm)
	(b7) circle (.1cm)
	(b8) circle (.1cm)
 (bb1) circle (.1cm)
 (bb2) circle (.1cm)
 (bb3) circle (.1cm)
 (bb4) circle (.1cm);	
	
 \draw[thick,fill=white] (r1) circle (0.1cm)
	(r2) circle (0.1cm)
	(r3) circle (0.1cm)
	(r4) circle (0.1cm)
	(r5) circle (0.1cm)
	(r6) circle (0.1cm) 
	(r7) circle (0.1cm)
	(r8) circle (0.1cm)
 (rr1) circle (.1cm)
 (rr2) circle (.1cm)
 (rr3) circle (.1cm)
 (rr4) circle (.1cm);	
  \draw[] (5.4,0) {};
\end{tikzpicture}
\begin{tikzpicture}[>=stealth,scale=0.7]
\draw[line width=1.8pt,fill=white] (0,0) circle (4cm);
\draw[line width=1.8pt,fill=gray!20] (0,0) circle (1.3cm);
\path (0:4) coordinate (b1)
(45:4) coordinate (b2)			(90:4) coordinate (b3)
(135:4) coordinate (b4)
(180:4) coordinate (b5)
(225:4) coordinate (b6)
(270:4) coordinate (b7)
(315:4) coordinate (b8)
				
(22.5:4) coordinate (r1)
(67.5:4) coordinate (r2)
(112.5:4) coordinate (r3)
(157.5:4) coordinate (r4)
(202.5:4) coordinate (r5)(247.5:4) coordinate (r6)
(292.5:4) coordinate (r7)
(337.5:4) coordinate (r8);

\path (0:1.3) coordinate (bb1)
(90:1.3) coordinate (bb2)		(180:1.3) coordinate (bb3)
(270:1.3) coordinate (bb4)
(45:1.3) coordinate (rr1)(135:1.3) coordinate (rr2)(225:1.3) coordinate (rr3)(315:1.3) coordinate (rr4);


\draw[very thick,bend left](r7)to(r8)to(r1)to(r2)to(r3)to(r4)to(r5)to(r6);

\draw[very thick,bend left,red!60](b8)to(b1)to(b2)to(b3)to(b4)to(b5)to(b6);
\draw[very thick,dashed,red!60]plot [smooth,tension=1] coordinates {(b1)(2.8,-1)(b7)};
\draw[very thick,dashed,red!60](b7)to(bb4);
\draw[very thick,dashed]plot [smooth,tension=1] coordinates {(r8)(0,-2)(rr3)};

\draw[very thick]plot [smooth,tension=1] coordinates {(rr1)(0,1.7) (rr2)};
\draw[very thick]plot [smooth,tension=1] coordinates {(rr2)(-1.7,0) (rr3)};
\draw[very thick]plot [smooth,tension=1] coordinates { (rr3)(0,-1.7) (rr4)};
\draw[very thick]plot [smooth,tension=1] coordinates {(rr4)(1.7,0)(rr1)};
\draw[very thick,red!60]plot [smooth,tension=1] coordinates {(bb2)(1.4,1.4)(bb1)};
\draw[very thick,red!60]plot [smooth,tension=1] coordinates {(bb1)(1.4,-1.4)(bb4)};
\draw[very thick,red!60]plot [smooth,tension=1] coordinates {(bb4)(-1.4,-1.4)(bb3)};
\draw[very thick,red!60]plot [smooth,tension=1] coordinates {(bb3)(-1.4,1.4)(bb2)};


\draw[thick,red,fill=red] (b1) circle (0.1cm)
	(b2) circle (.1cm)
	(b3) circle (.1cm)
	(b4) circle (.1cm)
	(b5) circle (.1cm)
	(b6) circle (.1cm)
	(b7) circle (.1cm)
	(b8) circle (.1cm)
 (bb1) circle (.1cm)
 (bb2) circle (.1cm)
 (bb3) circle (.1cm)
 (bb4) circle (.1cm);	
	
 \draw[thick,fill=white] (r1) circle (0.1cm)
	(r2) circle (0.1cm)
	(r3) circle (0.1cm)
	(r4) circle (0.1cm)
	(r5) circle (0.1cm)
	(r6) circle (0.1cm) 
	(r7) circle (0.1cm)
	(r8) circle (0.1cm)
 (rr1) circle (.1cm)
 (rr2) circle (.1cm)
 (rr3) circle (.1cm)
 (rr4) circle (.1cm);	
  \draw[] (0,4.8) {};
\end{tikzpicture}
		\end{center}
		\begin{center}
			\caption{
Complete the $\rpoint$-arcs in the once-punctured disk and the annulus of Figure \ref{fig:twoposs3} to get the simple coordinate $\zD^*$.
We always need two more $\rpoint$-arcs. Depending on these arcs' positions, there are several cases, where we list four representatives, where the new $\rpoint$-arcs are dashed. For these cases, there are $1\leq r \leq 2$ choices of $\bpoint$-arc $\zg_n$ such that $\bigoplus_{i=1}^{n}M_{\zg_i}$ is a tilting module in $\ma$, where $\zg_n$ is the dashed $\bpoint$-arc.
The only exception is when the surface is an annulus and there is only one $\bpoint$-point on one of the boundary component, see Figure \ref{fig:upbound}.}\label{fig:twoposs5}
		\end{center}
	\end{figure}

\begin{figure}
\begin{center}
\begin{tikzpicture}[>=stealth,scale=0.7]
\draw[line width=1.8pt,fill=white] (0,0) circle (4cm);
\draw[line width=1.8pt,fill=gray!20] (0,0) circle (.6cm);
\path (0:4) coordinate (b1)
(45:4) coordinate (b2)			(90:4) coordinate (b3)
(135:4) coordinate (b4)
(180:4) coordinate (b5)
(225:4) coordinate (b6)
(270:4) coordinate (b7)
(315:4) coordinate (b8)
				
(22.5:4) coordinate (r1)
(67.5:4) coordinate (r2)
(112.5:4) coordinate (r3)
(157.5:4) coordinate (r4)
(202.5:4) coordinate (r5)(247.5:4) coordinate (r6)
(292.5:4) coordinate (r7)
(337.5:4) coordinate (r8);

\path (0:.6) coordinate (bb1)
(90:.6) coordinate (bb2)		(180:.6) coordinate (bb3)
(270:.6) coordinate (bb4)
(45:.6) coordinate (rr1)(135:.6) coordinate (rr2)(225:.6) coordinate (rr3)(315:.6) coordinate (rr4);


\draw[very thick,bend left](r7)to(r8)to(r1)to(r2)to(r3)to(r4)to(r5)to(r6)to(r7);

\draw[very thick,bend left,red!60](b8)to(b1)to(b2)to(b3)to(b4)to(b5)to(b6)to(b7)to(b8);
\draw[very thick,red!60](b3)to(bb2);

\draw[very thick,dashed]plot [smooth,tension=1] coordinates {(r2)(0,2)(-.8,0.15)(0,-.6)};

\draw[very thick,dashed]plot [smooth,tension=1] coordinates {(r3)(0,2)(.8,0.15)(0,-.6)};

\draw[very thick,dashed]plot [smooth,tension=1] coordinates {(r2)(1.5,0.5)(0,-.6)};

\draw[very thick,dashed]plot [smooth,tension=1] coordinates {(r3)(-1.5,0.5)(0,-.6)};

\draw[very thick,dashed]plot [smooth,tension=1] coordinates {(r1)(2.2,0)(0,-.6)};

\draw[very thick,dashed]plot [smooth,tension=1] coordinates {(r4)(-2.2,0)(0,-.6)};

\draw[very thick,dashed]plot [smooth,tension=1] coordinates {(r5)(0,-.6)};
\draw[very thick,dashed]plot [smooth,tension=1] coordinates {(r6)(0,-.6)};
\draw[very thick,dashed]plot [smooth,tension=1] coordinates {(r7)(0,-.6)};
\draw[very thick,dashed]plot [smooth,tension=1] coordinates {(r8)(0,-.6)};

\draw[thick,red,fill=red] (b1) circle (0.1cm)
	(b2) circle (.1cm)
	(b3) circle (.1cm)
	(b4) circle (.1cm)
	(b5) circle (.1cm)
	(b6) circle (.1cm)
	(b7) circle (.1cm)
	(b8) circle (.1cm)
 (bb2) circle (.1cm);	
	
 \draw[thick,fill=white] (r1) circle (0.1cm)
	(r2) circle (0.1cm)
	(r3) circle (0.1cm)
	(r4) circle (0.1cm)
	(r5) circle (0.1cm)
	(r6) circle (0.1cm) 
	(r7) circle (0.1cm)
	(r8) circle (0.1cm)
 (bb4) circle (.1cm);	
\end{tikzpicture}
		\end{center}
		\begin{center}
			\caption{
A gentle algebra from an annulus with rank $n$, and an almost-tilting module which has $n+1$ complements, where $n=8$ in the picture. The almost-tilting module is given by the solid $\bpoint$-arcs, and its complements are given by the dashed arcs.}\label{fig:upbound}
		\end{center}
	\end{figure}
\end{proof}

Happel conjectured that
any almost-tilting module over a finite-dimensional algebra admits a finite number of complements, and the number of complements is bounded by $2n-1$. As explained in the introduction, this supposed bound should be changed to $2n$, see the conjecture {\bf{(B)}} in the introduction. The following theorem says that the modified Happel's conjecture holds for gentle algebras.
In particular, the theorem shows that the answer to the problem ${\bf (Cn-1)}$ is positive for any gentle algebra with rank $n.$

\begin{theorem}\label{thm:almost-tilting is partial-tilting}
Any almost-tilting module over a gentle algebra is partial-tilting, and there are at most $2n$ complements, where $n$ is the rank of the algebra.
\end{theorem}
\begin{proof}
A gentle algebra with rank one is a single vertex or a vertex with a loop such that the composition of the loop is zero. It is clear that the statements hold for both cases.
Now let $A$ be a gentle algebra associated with a marked surface $(\cals,\calm,\zD^*)$ with a simple coordinate, and the rank $n$ of $A$ is at least two. 
Let $M=\bigoplus_{i=1}^{n-1}M_{\zg^i}$ be an almost-tilting module over $A$ with a pre-tilting collection $\zG=\{\zg^i, 1\leq i \leq n-1\}$. By Proposition \ref{prop:rank}, the rank of the algebra $A_\zG$ associated with the cutting surface $(\cals_\zG,{\calm_\zG},{\zD^*_\zG})$ is $2n-1$.
Denote by $\widehat{\zG}=\{\zg^i_1,\zg^i_2, 1\leq i \leq n-1\}$ the set of $\bpoint$-arcs on the cutting surface induced by arcs in $\zG$.
Now consider the sub-surfaces formed by arcs in $\widehat{\zG}$ and the boundary segments. Since there are $2n-2$ $\bpoint$-arcs in $\widehat{\zG}$, excepting for one sub-surface, each sub-surface is a polygon with one $\rpoint$-point. 
Note that this exceptional sub-surface satisfies the conditions in the statement of Lemma \ref{lem:keykey}, and this sub-surface corresponds a connected component of $(\cals_\zG,{\calm_\zG},{\zD^*_\zG})$, which we denote by $(\widetilde{\cals},\widetilde{\calm})$.
The module $M$ in $\ma$ induces a module $\widehat{M}=\bigoplus_{i=1}^{n-1}(M_{\zg^i_1}\oplus M_{\zg^i_2})$ in $\maG$, which is almost-tilting.
Let $\tilde{n}$ be the rank of $(\widetilde{\cals},\widetilde{\calm})$, and let $m$ be the number of the complements to $\widehat{M}$.
Then $\tilde{n}\leq 2n-1$, and by Lemma \ref{lem:keykey}, $1\leq m\leq \tilde{n}+1\leq 2n$. Furthermore, by Corollary \ref{cor:key} and Proposition \ref{prop:tilting dissection}, there are $m$ complements to $M$, where $1\leq m\leq 2n$.
\end{proof}

\begin{remark}
Note that there exists a gentle algebra and an almost-tilting on it with $2n-1$ complements. For example, we consider the gentle algebra $A$ associated with an annulus with two marked $\bpoint$-points given in Figure \ref{fig:upbound1}. Then the rank of $A$ is two, and the module $M_\zg$ associated with the $\bpoint$-arc $\zg$ is almost-tilting. There are three complements to $M_\zg$, given by the dashed $\bpoint$-arcs. 
It seems interesting to ask whether the up-bound $2n$ can be reached. A possible way to construct such an example is starting from the annulus with a simple coordinate given in Figure \ref{fig:upbound}, and then `gluing' it to a (connected) marked surface with a simple coordinate by $n-1$ steps, such that the lifted $\bpoint$-arcs form a pre-tilting collection.
Then the pre-tilting collection gives rise to an almost-tilting module over the `gluing algebra', whose rank is $n$ and it has $2n$ complements.
\begin{figure}
\begin{center}
\begin{tikzpicture}[>=stealth,scale=0.7]
\draw[line width=1.8pt,fill=white] (0,0) circle (4cm);
\draw[line width=1.8pt,fill=gray!20] (0,0) circle (.7cm);

\path (90:4) coordinate (r1)
      (90:.7) coordinate (r2)
      (270:4) coordinate (b1)
      (270:.7) coordinate (b2);

\draw[very thick,red]plot [smooth,tension=1] coordinates {(r1)(-1.8,-.9)(1.8,-.9)(r1)};
\draw[very thick,red]plot [smooth,tension=1] coordinates {(r1)(r2)};
\draw[very thick]plot [smooth,tension=1] coordinates {(b1)(-2.4,-2)(-2,2)(2,2)(2.4,-2)(b1)};

\draw[very thick,dashed]plot [smooth,tension=1] coordinates {(b1)(b2)};
\draw[very thick,dashed]plot [smooth,tension=1] coordinates {(b1)(1.5,-.5)(.2,1.5)(-1.5,0)(b2)};
\draw[very thick,dashed]plot [smooth,tension=1] coordinates {(b1)(-1.2,-1)(-.2,1.2)(1,-.2)(b2)};

\node [] at (2.5,0){$\zg$};

\draw[thick,red,fill=red] (r1) circle (0.1cm);	
\draw[thick,red,fill=red](r2) circle (0.1cm);	
\draw[thick,fill=white] (b1) circle (0.1cm);	
\draw[thick,fill=white](b2) circle (0.1cm);	
\end{tikzpicture}
		\end{center}
		\begin{center}
			\caption{
An example of a gentle algebra with rank $n$, and an almost-tilting on it which has $2n-1$ complements, where $n=2$. The almost-tilting module is given by the solid $\bpoint$-arc $\zg$, and its complements are given by the dashed arcs.}\label{fig:upbound1}
		\end{center}
	\end{figure}
\end{remark}

\subsection{Pre-tilting modules that are not partial-tilting}\label{section:pre tilting}
\begin{theorem}\label{thm:pre-tilting not partial-tilting}
Any pre-tilting module over a gentle algebra with rank $n=2$ is always partial-tilting. For any $n\geq 3$ and $1\leq m \leq n-2$, there always exists a (connected) gentle algebra with rank $n$ and a pre-tilting module over it with rank $m$ which is not partial-tilting.
\end{theorem}
\begin{proof}
The first claim follows from Theorem \ref{thm:almost-tilting is partial-tilting}.
We explicitly construct counterexamples for the second statement. 
For any $n\geq 3$, let $(\cals, \calm)$ be a torus with one boundary component which has $n-1$ marked $\bpoint$-points. We consider a simple coordinate $\zD^*$ given in Figure \ref{fig:not partial2}, with the associated gentle algebra $A(n)$. Let $\zg^1, \zg^2, \cdots, \zg^m$ be $\bpoint$-arcs depicted in the figure. Then for each $1\leq m \leq n-2$,  the module $M=\oplus_{i=1}^{m}M_{\zg^i}$ is a pre-tilting module over $A(n)$, which we claim not to be partial-tilting. Note that it is only needed to prove the case when $m=1$. 

Suppose that $\top$ is a tilting dissection which contains $\zg^1$. 
We claim that each $\bpoint$-arc $\zg$ in $\top$ does not intersect the $\rpoint$-arcs $1$ and $2$.
Then $\zg$ belongs to the sub-surface of the torus formed by $\rpoint$-arcs $1$, $2$, and the boundary segments, just as depicted in the figure. Note that this sub-surface is topologically a disk with rank $n-2$, after deforming the $\rpoint$-arcs $1$, $2$ as $\rpoint$-points. Furthermore, $\top$ is a tilting dissection on the sub-surface. However, this is impossible since there are $n$ arcs in $\top$. 

Now we prove the claim. At first, if one endpoint of $\zg$ is $p_1$, then it can not intersect the $\rpoint$-arcs $1$, $2$, otherwise, the weight of the oriented intersection between $\zg$ and $\zg^1$ at $p_1$ does not equal zero, which contradicts the assumption that both $\zg$ and $\zg^1$ belong to the tilting dissection $\top$.
Secondly, if the endpoints of $\zg$ are $p_i$, $2\leq i \leq p-1$, then it is enough to show that $\zg$ does not go through $\ccc$ and $\ddd$. 
Suppose that $\zg$ goes through $\ccc$, then one of its endpoints must be $p_2$, otherwise, there is an interior intersection between $\zg$ and $\zg^1$. Furthermore, another endpoint must be $p_1$, which derives a contradiction by the first step. A similar argument works for the arrow $\ddd$, that is, if $\zg$ goes through $\ddd$, then one of the endpoints of $\zg$ must be $p_1$. This confirms the claim, and we are done. 

\begin{figure}
\begin{center}
\begin{tikzpicture}
[>=stealth,scale=0.6]
		\draw[line width=1.5pt,fill=gray!30]  (-4,3.5) ellipse[x radius=1.8, y radius=1];
		\draw[line width=1.5pt,fill=gray!30]  (-6,-3.5) ellipse[x radius=1.8, y radius=1];
		\draw[line width=1.5pt,fill=gray!30]  (6,3.5) ellipse[x radius=1.8, y radius=1];
		\draw[line width=1.5pt,fill=gray!30]  (4,-3.5) ellipse[x radius=1.8, y radius=1];
  
		\draw[thick,red, fill=red] (-4,4.5) circle (0.11);
		\draw[thick,red,fill=red] (-4.9,4.35) circle (0.11);	
		\draw[thick,red,fill=red] (-5.8,3.5) circle (0.11);
		\draw[thick,red,fill=red] (-2.2,3.5) circle (0.11);
				
		\draw[thick,red, fill=red] (6,4.5) circle (0.11);
		\draw[thick,red,fill=red] (4.9,4.3) circle (0.11);	
		\draw[thick,red,fill=red] (7.8,3.5) circle (0.11);
		\draw[thick,red,fill=red] (4.2,3.5) circle (0.11);

		\draw[thick,red, fill=red] (4,-2.5) circle (0.11);
      \draw[thick,red,fill=red] (2.9,-2.7) circle (0.11);
		\draw[thick,red,fill=red] (5.8,-3.5) circle (0.11);
		\draw[thick,red,fill=red] (2.2,-3.5) circle (0.11);

	    \draw[thick,red, fill=red] (-6,-2.5) circle (0.11);
			\draw[thick,red,fill=red] (-7.1,-2.7) circle (0.11);
			\draw[thick,red,fill=red] (-7.8,-3.5) circle (0.11);
			\draw[thick,red,fill=red] (-4.2,-3.5) circle (0.11);

			\draw [red, line width=1pt] (-5.8,3.5) to (-4.2,-3.5);
			\draw [red, line width=1pt] (-2.2,3.5) to (4.2,3.5);
			\draw [red, line width=1pt] (4.2,3.5) to (5.8,-3.5);
			\draw [red, line width=1pt] (-4.2,-3.5) to (2.2,-3.5);
				
				\draw[cyan,line width=1pt,red] (5.8,-3.5) 
				to[out=110,in=0](0.5,2)
				to[out=-180,in=120](-1,-1.5)
				to[out=-60,in=170](2.2,-3.5);
				
				\draw[cyan,line width=1pt,red] (5.8,-3.5) 
				to[out=115,in=0](2,0.5)
				to[out=180,in=90](1,-0.5)
				to[out=-90,in=150](2.9,-2.7);
				
				\draw[cyan,line width=1pt,red] (5.8,-3.5) 
				to[out=120,in=0](3,-0.8)
				to[out=180,in=150](4,-2.5);
				
				\draw[red] (0,3.8) node {$1$};
				\draw[red] (0,-3.8) node {$1$};
				\draw[red] (-5.3,0) node {$2$};
				\draw[red] (5.3,0) node {$2$};
				\draw[red] (-1.5,1) node {$3$};
				\draw[red] (0.8,0) node {$4$};
				\draw[red] (2.5,-0.8) node {$n$};
				
				\draw[black] (1,1.3) node {\tiny$\gamma^{1}$};
				\draw[black] (1.9,-1.5) node {\tiny$\gamma^{2}$};
				\draw[black] (2.7,-2.1) node {\tiny$\gamma^{n-2}$};
				\draw[red,thick,bend right,->,>=stealth] (5,0) to node[above ]{$\ccc$} (4.35,-0.3);
				\draw[red,thick,bend right,->,>=stealth] (3.8,3.45) to node[below left]{$\aaa$} (4.3,3);
				\draw[red,thick,bend right,->,>=stealth] (-3.8,-3.45) to node[above right]{$\bbb$} (-4.3,-3);
				\draw[red,thick,bend right,->,>=stealth] (0.2,-2.8) to node[ left]{$\ddd$} (0,-3.48);
			\draw[cyan,line width=1pt,black] (-4,2.5) to[out=0,in=100](4.7,-2.5);
            \draw[] (-4,2.1) node {\tiny$p_1$};
            \draw[] (4.7,-3) node {\tiny$p_2$};
            \draw[] (3,-3.4) node {\tiny$p_{n-1}$};
			\draw[cyan,line width=1pt,black] (4.7,-2.5) 
			to[out=140,in=0](3.5,-1.8)to[out=180,in=150](3.45,-2.55);
			\draw[cyan,line width=1pt,black] (4.7,-2.5)to[out=120,in=0](3,-1)to[out=180,in=130](2.4,-3.05);

\draw[thick,black,fill=white] (4.7,-2.57) circle (0.11);	
\draw[thick,black,fill=white] (6,2.5) circle (0.11);
\draw[thick,black,fill=white] (-4,2.5) circle (0.11);
\draw[thick,black,fill=white] (4,-4.5) circle (0.11);
\draw[thick,black,fill=white] (3.45,-2.55) circle (0.11);	
\draw[thick,black,fill=white] (2.4,-3.05) circle (0.11);	\draw[thick,black,fill=white] (-6,-4.5) circle (0.11);
\end{tikzpicture}
\begin{tikzpicture}[>=stealth]
				\draw (-6,1.2) node {$1$};
				\draw (-6,-1.2) node {$2$};
				\draw (-4.5,0) node {$3$};
				\draw (-2.8,0) node {$4$};
				\draw (-.5,0) node {$\dots$};
				\draw (1.8,0) node {$n$};

				\draw (-6.35,0) node {$\aaa$};
				\draw (-5.65,0) node {$\bbb$};
				\draw (-5,-0.8) node {$\ccc$};
				\draw (-5,0.8) node {$\ddd$};
				
				\draw [very thick,->] (-5.9,0.9) -- (-5.9,-0.9);
				\draw [very thick,->] (-6.1,0.9) -- (-6.1,-0.9);
\draw [very thick,->] (-4.6,0.15) -- (-5.8,1.1);
\draw [very thick,->] (-5.8,-1.1) --(-4.6,-0.15) ;
\draw [very thick,->] (-4.3,0) -- (-3,0);
\draw [very thick,->] (-2.5,0) -- (-1.1,0);
\draw [very thick,->] (0,0) -- (1.5,0);
\draw[dotted,thick,bend left](-5.4,0.8) to (-5.85,0.4);		
\draw[dotted,thick,bend left](-4.9,-0.4) to (-4.9,0.4);	
\draw[dotted,thick,black](-6.2,-0.5) to[out=180,in=180](-6,-1.6)to[out=0,in=-45](-5.5,-1);	
\end{tikzpicture}
\end{center}
\caption{The surface model and the associated gentle algebra such that the pre-tilting module $\bigoplus_{i=1}^{m}M_{\zg^i}$ is not partial-tilting for any $1 \leq m \leq n-2$, as a generalization of the example given in \cite{RS89} when $n=3$.} 
	\label{fig:not partial2}
\end{figure}
\end{proof}

\begin{figure}
		\begin{center}
\begin{tikzpicture}
[>=stealth,scale=0.6]
\draw [line width=2pt] (-8,4) to (8,4);
\draw [line width=2pt] (-8,-4) to (8,-4);
	            \draw [line width=1.5pt,red] (-6,4) to (-6,-4);
	            \draw [line width=1.3pt,red] (6,4) to (6,-4);
				\draw [line width=1.3pt,red] (6,4) to (-6,-4);
				
				\draw[red,thick,bend right,->,>=stealth] (-6,3.2) to node[below right]{$\ccc$} (-5.5,3.65);
				\draw[red,thick,bend right,->,>=stealth] (5.5,3.65) to node[below left]{$\bbb$} (6,3.2);

\draw[red,thick,bend right,->,>=stealth] (-5.5,-3.65) to (-6,-3.2);
\draw[red] (-5.2,-2.9) node {$\aaa$};	

\draw[red,thick,bend right,->,>=stealth] (-1,-1.25) to (-1.3,-0.88);\draw[red] (-.7,-.8) node {$\ddd$};

				\draw[line width=1.5pt,red,fill=red] (-6,4) circle (0.13);
				\draw[line width=1.5pt,red,fill=red] (0,4) circle (0.13);
				\draw[line width=1.5pt,red,fill=red] (6,4) circle (0.13);
				\draw[line width=1.5pt,red,fill=red] (-6,-4) circle (0.13);
				\draw[line width=1.5pt,red,fill=red] (-3,-4) circle (0.13);
				\draw[line width=1.5pt,red,fill=red] (0,-4) circle (0.13);
				\draw[line width=1.5pt,red,fill=red] (3,-4) circle (0.13);
				\draw[line width=1.5pt,red,fill=red] (6,-4) circle (0.13);
				
				\draw[line width=1.5pt, fill=black] (-3,4) circle (0.13);
				\draw[line width=1.5pt, fill=black] (3,4) circle (0.13);
				\draw[line width=1.5pt, fill=black] (-4.5,-4) circle (0.13);
				\draw[line width=1.5pt, fill=black] (-1.5,-4) circle (0.13);
				\draw[line width=1.5pt, fill=black] (1.5,-4) circle (0.13);
				\draw[line width=1.5pt, fill=black] (4.5,-4) circle (0.13);

				\draw[cyan,line width=1.3pt,red] (-6,4) to[out=-35,in=-145](0,4);
				\draw[cyan,line width=1.3pt,red] (-6,-4)
				to[out=30,in=180](1,-1)
				to[out=0,in=90](3,-4);
				\draw[cyan,line width=1.3pt,red] (-3,-4)
				to[out=45,in=180](1,-1.8)
				to[out=0,in=95](3,-4);
				\draw[cyan,line width=1.3pt,red] (0,-4)
				to[out=65,in=180](1.5,-2.6)
				to[out=0,in=95](3,-4);

				\draw[cyan,line width=1.3pt,black] (-3,4) to[out=-35,in=-145](3,4);
				\draw[cyan,line width=1.3pt,black] (-4.5,-4) 
				to[out=45,in=180](-1.5,-2.5)
				to[out=0,in=135](1.5,-4);
				\draw[cyan,line width=1.3pt,black] (-1.5,-4) 
				to[out=35,in=180](0,-3.3)
				to[out=0,in=145](1.5,-4);
				\draw[cyan,line width=1.3pt,black] (-1.5,-4) 
				to[out=35,in=180](0,-3.3)
				to[out=0,in=145](1.5,-4);
				\draw[cyan,line width=1.3pt,black] (4.5,-4) 
				to[out=110,in=0](3,-2.3)
				to[out=180,in=70](1.5,-4);
				
				\draw[red] (0,0.5) node {$1$};
				\draw[red] (-6.4,0.9) node {$2$};
				\draw[red] (6.4,0.9) node {$2$};
				\draw[red] (-3,2.6) node {$3$};
				\draw[red] (0.6,-0.7) node {$4$};
				\draw[red] (1.2,-1.5) node {$5$};
				\draw[red] (1.6,-2.4) node {$n$};
	
				\draw[black] (3,-2) node {$\gamma^{1}$};
				\draw[black] (0,2.5) node {$\gamma^{2}$};
				\draw[black] (-2,-2.2) node {$\gamma^{3}$};
				\draw[black] (-1,-3.2) node {$\gamma^{n-2}$};
	\end{tikzpicture}
			\begin{tikzpicture}[>=stealth]
				\draw (0,1) node {};
				\draw (-6,0) node {$3$};
				\draw (-4,0) node {$2$};
				\draw (-2,0) node {$1$};
				\draw (6,0) node {$n$};
				\draw (4,0) node {$\dots$};
				\draw (2,0) node {$5$};
				\draw (0,0) node {$4$};

				\draw (-3,0.4) node {$\aaa$};
				\draw (-3,-0.4) node {$\bbb$};
				\draw (-5,0.4) node {$\ccc$};
				\draw (-1,0.4) node {$\ddd$};
				
				\draw [very thick,->] (-4.2,0) -- (-5.8,0);
				\draw [very thick,->] (-2.2,.08) -- (-3.8,.08);
				\draw [very thick,->] (-2.2,-.08) -- (-3.8,-.08);
				\draw [very thick,->] (-0.2,0) -- (-1.8,0);
				\draw [very thick,->] (0.2,0) -- (1.8,0);
				\draw [very thick,->] (2.2,0) -- (3.7,0);
				\draw [very thick,->] (4.3,0) -- (5.8,0);

				\draw[dotted,thick,bend left](-4.8,0) to (-3.2,0.15);					
				\draw[dotted,thick,bend right](-2.8,-0.15) to (-1,0);	
\end{tikzpicture}
 \end{center}
	\caption{A counterexample from an annulus such that a pre-tilting module is not partial-tilting.} 
	\label{fig:not partial} 
\end{figure}
The algebra $A(3)$ is exactly the counterexample given in \cite[Section 2]{RS89}.
It is interesting to note that, with appropriate gradings, A(3) provides an example such that there exists a pre-silting object in the associated derived category that cannot be completed into a silting object, see details in \cite[Theorem 4.12]{JSW23} and \cite[Remark 1.6]{LZ23}. 
Note that the surface in Figure \ref{fig:not partial2} is a torus. 
In fact, for any $n\geq 4$ and $2\leq m \leq n$, the annulus already gives rise to counter-examples. For the surface and the algebra given in Figure \ref{fig:not partial}, the module $M=\oplus_{i=1}^{m}M_{\zg^i}$ is a pre-tilting module which is not partial-tilting.
This is because the surface with simple coordinate is obtained from the surface with coordinate in  Figure \ref{fig:not partial2} by cutting along $\zg^1$, which induces two $\bpoint$-arcs $\zg^1$ and $\zg^2$ in Figure \ref{fig:not partial}.

\section{Appendix}

Here we list all the possibilities of the cutting surfaces, depending on the endpoints of the arc $\zg$, see Figure \ref{table:list}.

\begin{figure} 
	\centering 
	\begin{tikzpicture}[>=stealth,scale=.7]
								
			\draw [line width=1pt] (2.05,2.5) to (7.95,2.5);
					
			\draw [bend left, gray!60, line width=4pt] (1.2,4.92) to (1.2,0.08);
			\draw [bend right,  gray!60, line width=4pt] (8.8,4.92) to (8.8,0.08);
					
			\draw [bend left, line width=2.5pt] (1.3,5) to (1.3,0);
			\draw [bend right, line width=2.5pt] (8.7,5) to (8.7,0);

			\draw[thick,fill=white] (2.05,2.5) circle (0.15);
			\draw[thick,fill=white] (7.95,2.5) circle (0.15);
		    \draw[thick, red ,fill=red] (1.86,3.75) circle (0.15);
		    \draw[thick, red,fill=red] (1.86,1.25) circle (0.15);
		     \draw[thick, red ,fill=red] (8.14,3.75) circle (0.15);
		    \draw[thick, red,fill=red] (8.14,1.25) circle (0.15);

		    \draw (1.5,2.5) node {$p_1$};
		    \draw (8.55,2.5) node {$p_2$};
		    \draw (5,2.8) node {$\gamma$};
\draw (12,2.8) node {};
\end{tikzpicture}
\begin{tikzpicture}[>=stealth,scale=.6]

					\draw [line width=1pt] (2.05,2.5) to (7.95,2.5);
					\draw [line width=1pt] (2.05,3.2) to (7.95,3.2);
			
					\draw [bend left, gray!60, line width=4pt] (1.88,2.5) to (1.18,0.14);
			    	\draw [bend left, gray!60, line width=4pt] (1.18,5.7) to (1.88,3.2);
			    	\draw [bend right,  gray!60, line width=4pt] (8.1,2.5) to (8.8,0.14);
					\draw [bend right,  gray!60, line width=4pt] (8.8,5.7) to (8.1,3.2);
					
					\draw [bend left, line width=2.5pt] (2.05,2.5) to (1.3,0);
					\draw [bend left, line width=2.5pt] (1.3,5.8) to (2.05,3.2);
					\draw [bend right, line width=2.5pt] (7.95,2.5) to (8.7,0);
					\draw [bend right, line width=2.5pt] (8.7,5.8) to (7.95,3.2);

					\draw[thick,fill=white] (2.05,2.5) circle (0.15);
					\draw[thick,fill=white] (7.95,2.5) circle (0.15);
					\draw[thick,fill=white] (2.05,3.2) circle (0.15);
					\draw[thick,fill=white] (7.95,3.2) circle (0.15);
					
					\draw[thick, red ,fill=red] (2.05,4.55) circle (0.15);
					\draw[thick, red,fill=red] (2.05,1.25) circle (0.15);
					\draw[thick, red ,fill=red] (7.95,4.55) circle (0.15);
					\draw[thick, red,fill=red] (7.95,1.25) circle (0.15);
					\draw[thick, red,fill=red] (5,2.5) circle (0.15);
					\draw[thick, red,fill=red] (5,3.2) circle (0.15);

					\draw (1.4,3.3) node {$p'_{1}$};
					\draw (8.65,3.3) node {$p'_2$};
					\draw (3.3,3.6) node {$\gamma'$};
					\draw (1.4,2.5) node {$p''_{1}$};
					\draw (8.65,2.5) node {$p''_2$};
					\draw (3.3,2) node {$\gamma''$};
				    \draw (5,3.8) node {$q'$};
				    \draw (5,1.9) node {$q''$};

				\end{tikzpicture}
    
    Case I. $p_{1} \neq p_{2} \in \mathcal{M}_\bpoint$

\begin{tikzpicture}    [>=stealth,scale=.6]
					
					\draw [ gray!60, line width=4pt] (0,-0.15) to (9,-0.15);
					\draw[line width=1pt]plot [smooth,tension=1] coordinates {(4.5,0) (2.5,2.5) (4.5,4.5) (6.5,2.5)(4.5,0)};

					\draw [line width=2.5pt] (0,0) to (9,0);
                    \draw[thick,fill=white] (4.5,0) circle (0.15);
                    
                    \draw[thick, red ,fill=red] (1.5,0) circle (0.15);
                    \draw[thick, red ,fill=red] (7.5,0) circle (0.15);

                 	\draw (2,2.5) node {$\gamma$};
                    \draw (4.5,.6) node {$p$};
                    \draw (11,1) node {};
\end{tikzpicture}
\begin{tikzpicture}   [>=stealth,scale=.6]
					
					\draw [ gray!60, line width=4pt] (0,-0.15) to (6,-0.15);
					\draw[ gray!60, line width=4pt]plot [smooth,tension=1] coordinates {(8.5,-1.12) (6.85,1.5) (8.5,3.15) (10.15,1.5)(8.5,-1.12)};
					\draw[line width=2.5pt]plot [smooth,tension=1] coordinates {(8.5,-1) (7,1.5) (8.5,3) (10,1.5)(8.5,-1)};

					\draw [line width=2.5pt] (0,0) to (6,0);
					\draw[thick,fill=white] (2,0) circle (0.15);
					\draw[thick,fill=white] (4,0) circle (0.15);
					\draw[thick,fill=white] (8.5,-1) circle (0.15);
					\draw[thick, red ,fill=red] (1,0) circle (0.15);
					\draw[thick, red ,fill=red] (3,0) circle (0.15);
					\draw[thick, red ,fill=red] (5,0) circle (0.15);
					\draw[thick, red ,fill=red] (8.5,3) circle (0.15);
					
					\draw (2,-0.7) node {$p'_{1}$};
					\draw (3,-0.7) node {$q'$};
					\draw (4,-0.7) node {$p'_{2}$};
					\draw (.6,0.6) node {$\gamma'$};
					\draw (7.5,1.5) node {$\gamma''$};
					\draw (8.6,-.2) node {$p''$};
					\draw (8.5,3.7) node {$q''$};
				\end{tikzpicture}
				
  Case II. $p := p_{1} = p_{2} \in \mathcal{M}_{\bpoint} $

\begin{tikzpicture}[>=stealth,scale=.7]

					\draw [line width=1pt] (2.05,2.5) to (7.95,2.5);
					
					\draw [bend left, gray!60, line width=4pt] (1.2,4.92) to (1.2,0.08);

					\draw [bend left, line width=2.5pt] (1.3,5) to (1.3,0);

					\draw[thick,fill=white] (2.05,2.5) circle (0.15);
					\draw[thick,fill=white] (7.95,2.5) circle (0.15);
					\draw[thick, red ,fill=red] (1.86,3.75) circle (0.15);
					\draw[thick, red,fill=red] (1.86,1.25) circle (0.15);

					\draw (1.5,2.5) node {$p_1$};
					\draw (8.55,2.5) node {$p_2$};
					\draw (3,2.8) node {$\gamma$};
\end{tikzpicture}
\begin{tikzpicture}[>=stealth,scale=.7]
					
					\draw [ gray!60, line width=4pt] (0,-0.15) to (8,-0.15);

					\draw [line width=2.5pt] (0,0) to (8,0);
					\draw[thick,fill=white] (2,0) circle (0.15);
					\draw[thick,fill=white] (4,0) circle (0.15);
					\draw[thick,fill=white] (6,0) circle (0.15);
					
					\draw[thick, red ,fill=red] (1,0) circle (0.15);
					\draw[thick, red ,fill=red] (3,0) circle (0.15);
					\draw[thick, red ,fill=red] (5,0) circle (0.15);
					\draw[thick, red ,fill=red] (7,0) circle (0.15);

					\draw (2,-0.6) node {$p'_{1}$};
					\draw (3,-0.6) node {$q'$};
					\draw (4,-0.6) node {$p_{2}$};
					\draw (5,-0.6) node {$q''$};
					\draw (6,-0.6) node {$p''_{1}$};
					\draw (3,0.6) node {$\gamma'$};
					\draw (5,0.6) node {$\gamma''$};
\draw (5,-2.7) node {};
\end{tikzpicture}

Case III. $p_{1}\in \mathcal{M}_{0},   p_{2} \in \mathcal{P}_{\bpoint}  $

\begin{tikzpicture}[>=stealth,scale=.7]

					\draw [line width=1pt] (2.05,2.5) to (7.95,2.5);

					\draw[thick,fill=white] (2.05,2.5) circle (0.15);
					\draw[thick,fill=white] (7.95,2.5) circle (0.15);

					\draw (1.3,2.5) node {$p_1$};
					\draw (8.75,2.5) node {$p_2$};
					\draw (5,3.1) node {$\gamma$};
\draw (10.5,1) node {};
\end{tikzpicture}
\begin{tikzpicture}[>=stealth,scale=.4]
					
				
					\draw[line width=2.5pt,fill=gray!50]  (0,0) ellipse[x radius=4, y radius=2.5];

					\draw[thick,fill=white] (-4,0) circle (0.25);
					\draw[thick,fill=white] (4,0)  circle (0.25);
					\draw[thick,red, fill=red] (0,-2.5) circle (0.25);
					\draw[thick,red,fill=red] (0,2.5) circle (0.25);

				    \draw (-4.9,0) node {$p_1$};
				    \draw (4.9,0) node {$p_2$};
				    \draw (0,3.4) node {$q'$};
				    \draw (0,-3.4) node {$q''$};
				    \draw (-2,2.9) node {$\gamma'$};
				    \draw (-2,-2.9) node {$\gamma''$};
				\end{tikzpicture}
    
Case IV. $p_{1} \neq p_{2} \in \mathcal{P}_{\bpoint} $

\begin{tikzpicture}[>=stealth,scale=.5]
					
					\draw[line width=1pt]plot [smooth,tension=1] coordinates {(4.5,0) (2.5,2.5) (4.5,4.5) (6.5,2.5)(4.5,0)};

					\draw[thick,fill=white] (4.5,0) circle (0.15);

					\draw (2,2.4) node {$\gamma$};				
				
					\draw (4.5,-0.6) node {$p$};
					\draw (11,-0.6) node {};					
				\end{tikzpicture}
\begin{tikzpicture}[>=stealth,scale=.5]
			\draw[line width=2.5pt,fill=gray!50]plot [smooth,tension=1] coordinates {(1.5,-1) (-0.15,1.5) (1.5,3.15) (3.15,1.5)(1.5,-1)};
					\draw[ gray!60, line width=6.8pt]plot [smooth,tension=1] coordinates {(8.5,-1.12) (6.85,1.5) (8.5,3.15) (10.15,1.5)(8.5,-1.12)};
					\draw[line width=2.5pt]plot [smooth,tension=1] coordinates {(8.5,-1) (7,1.5) (8.5,3) (10,1.5)(8.5,-1)};

					\draw[thick,fill=white] (1.5,-1) circle (0.15);
					\draw[thick,fill=white] (8.5,-1) circle (0.15);
					\draw[thick, red ,fill=red] (8.5,3) circle (0.15);
					\draw[thick, red ,fill=red] (1.5,3.15) circle (0.15);

					\draw (6,1.5) node {$\gamma''$};
					\draw (-0.7,1.5) node {$\gamma'$};
					\draw (8.5,-1.7) node {$p''$};
					\draw (8.5,3.8) node {$q''$};
					\draw (1.5,-1.6) node {$p'$};
					\draw (1.5,3.8) node {$q'$};
										
				\end{tikzpicture}

Case V. $p := p_{1} = p_{2} \in \mathcal{P}_{\bpoint}$
	
\caption{For a simple $\bpoint$-arc $\zg$ on $(\cals,\calm)$, the surface $\cals_\zg$ is obtained by cutting $\cals$ along $\zg$. The set $\calm_\zg$ is defined as 
$\calm_\zg=\calm\setminus \{p_1,p_2\}\cup\{p_1',p_2',p_1'',p_2''\}\cup\{q',q''\}$, where some points may coincide. There are five cases, depending on the positions of the endpoints of $\zg$. In each case, the number of marked $\bpoint$-points from $\calm_\bpoint$ (as well as $\rpoint$-points) always adds two.} 
	\label{table:list} 
\end{figure}

Let $(\cals,\calm,\zD^*)$ be a marked surface with a simple coordinate, and let $\zg$ be a zigzag $\bpoint$-arc on the surface. In the following, we give a direct construction of the gentle algebra $A_\zg$ associated with the cutting surface $(\cals_\zg,\calm_\zg,\zD^*_\zg)$ under an assumption that $\zg$ intersects each $\rpoint$-arc in $\zD^*$ at most once.
The proof is based on a direct verification using the definition of cutting surface, comparing Figure \ref{fig:boundarc}.
It seems interesting to give a combinatorial characterization of the algebra obtained by cutting the surface along a general zigzag $\bpoint$-arc.

Denote by $\omega=\omega_1\omega_2\cdots\omega_m$ the walk associated to $\zg$. In the following we consider the case that $m$ is even and $\omega_{2j}=u_j$, $\omega_{2j-1}=u_j^{-1}$ for $1\leq j \leq m/2$, where each $u_j$ is a non-trivial direct string. 
The construction for other cases is similar. 
Denote by $x_{j-1}=s(\omega_j)=t(\omega_{j-1})$ for each $1 \leq j \leq m$ and denote by $x_m=t(\omega_m)$. 
Let $A=kQ/I$ be the gentle algebra given by $\zD^*$. We construct a quiver $\widehat{Q}$ with relations $\widehat{I}$ in the following way:

(1) the vertex set $\widehat{Q}_0$ is a union of the following two sets:

$\bullet$ $\widehat{Q}'_0=Q_0\setminus\{x_j, 1\leq j \leq m-1\}$;

$\bullet$ $\widehat{Q}''_0=\{\widehat{x}_j, 1\leq j \leq m\}$.

(2) the arrow set $\widehat{Q}_1$ is a union of the following three sets:

$\bullet$ $\widehat{Q}'_1=Q_1\setminus\{\aaa ~|~ s(\aaa)=x_j \textrm{~or~} t(\aaa)=x_j, 1\leq j \leq m-1, \textrm{~or~} u_1\aaa \neq 0, \textrm{~or~} u_m \aaa\neq 0 \}$;

$\bullet$
$\widehat{Q}''_1=\{\widehat{\aaa}: x\rightarrow \widehat{x}_{i+1} ~|~ \aaa : x\rightarrow x_i \textrm{~with~} \aaa u_{i+1}\neq 0\}\cup$

$\hspace{1.32cm}\{\widehat{\aaa}: \widehat{x}_{i+1}\rightarrow x ~|~ \aaa : x_{i+1}\rightarrow x  \textrm{~with~} u_{i+1}\aaa\neq 0\}\cup$

$\hspace{1.32cm}\{\widehat{\aaa}: x\rightarrow \widehat{x}_{i} ~|~ \aaa : x\rightarrow x_i \textrm{~with~} \aaa u_{i}\neq 0\}\cup$

$\hspace{1.32cm}\{\widehat{\aaa}: \widehat{x}_{i} \rightarrow x ~|~ \aaa : x_{i-1}\rightarrow x \textrm{~with~} u_{i-1}\aaa \neq 0\}$,

$\hspace{1.32cm}$where $i=2j-1$ for $1\leq j \leq m/2$;

$\bullet$
$\widehat{Q}'''_1=\{\widehat{\bbb}: \widehat{x}_{i+1}\rightarrow x ~|~ \bbb : x_i\rightarrow x \textrm{~with~} u_{i}=\bbb u'_{i}\}\cup$

$\hspace{1.32cm}\{\widehat{\bbb}: \widehat{x}_{i}\rightarrow x ~|~ \bbb : x_i\rightarrow x \textrm{~with~} u_{i+1}=\bbb u'_{i+1}\}\cup$

$\hspace{1.32cm}\{\widehat{\bbb}: x\rightarrow \widehat{x}_{i+2} ~|~ \bbb : x\rightarrow x_{i+1} \textrm{~with~} u_{i+1}=u'_{i+1}\bbb\}\cup$

$\hspace{1.32cm}\{\widehat{\bbb}: x\rightarrow \widehat{x}_{i+1} ~|~ \bbb : x\rightarrow x_{i+1} \textrm{~with~} u_{i+2}=u'_{i+2}\bbb\}$,

$\hspace{1.32cm}$where $i=2j-1$ for $1\leq j \leq m/2$.

(3) the relation set $\widehat{I}$ is a union of the following four sets:

$\bullet$ $\widehat{I}'=\{\ccc_1\ccc_2 ~|~ \ccc_1, \ccc_2 \in \widehat{Q}'_1 \textrm{~with~} \ccc_1\ccc_2\in I\}$;

$\bullet$ $\widehat{I}''=\{\widehat{\aaa}\widehat{\bbb} \textrm{~or~}\widehat{\bbb}\widehat{\aaa} ~|~ \widehat{\aaa}\in \widehat{Q}''_1, \widehat{\bbb}\in \widehat{Q}'''_1 \textrm{~with~} \aaa\bbb\in I \textrm{~or~}\bbb\aaa\in I\}$;

$\bullet$ $\widehat{I}'''=\{\widehat{\aaa}\ccc \textrm{~or~} \ccc\widehat{\aaa} ~|~ \widehat{\aaa}\in \widehat{Q}''_1, \ccc\in \widehat{Q}'_1 \textrm{~with~} \aaa\ccc\in I \textrm{~or~} \ccc\aaa\in I \}$;

$\bullet$ $\widehat{I}''''=\{\widehat{\bbb}\ccc \textrm{~or~} \ccc\widehat{\aaa} ~|~ \widehat{\bbb}\in \widehat{Q}'''_1, \ccc\in \widehat{Q}'_1 \textrm{~with~} \aaa\ccc\in I \textrm{~or~} \ccc\bbb\in I \}$.

Then the gentle algebra $A_\zg$ associated with $\zD^*_\zg$ is isomorphic to $k\widehat{Q}/\widehat{I}$.

See Figure \ref{fig:apd1}, \ref{fig:apd2}, \ref{fig:apd3}, and \ref{fig:apd4} for an example of above constructions. Figure \ref{fig:apd1} shows a marked surface with a simple coordinate $(\cals,\calm,\zD^*)$ and a simple zigzag $\bpoint$-arc $\zg$. The associated gentle algebra is given in Figure \ref{fig:apd2}. Figure \ref{fig:apd3} depicts the cutting surface obtained by cutting $(\cals,\calm,\zD^*)$ along $\zg$, where the associated gentle algebra is given in Figure \ref{fig:apd4}.

\begin{figure}
	\begin{center}
		\begin{tikzpicture}[>=stealth,scale=0.8]
			\draw[line width=1.8pt,fill=white] (0,0) circle (7.5cm);
			\draw[line width=1.8pt,fill=gray!30] (0,2.5) circle (2cm);
			\draw[line width=1.8pt,fill=gray!30] (0,-2.5) circle (1.5cm);
	
            \path 
			    (90:7.5) coordinate (r1)
				(70:7.5) coordinate (r2)
				(40:7.5) coordinate (r3)
				(0:7.5) coordinate (r4)
				(-20:7.5) coordinate (r5)
				(-40:7.5) coordinate (r6)
				(-70:7.5) coordinate (r7)
				(-90:7.5) coordinate (r8)
				(-128:7.5) coordinate (r9)
				(-150:7.5) coordinate (r10)
				(-160:7.5) coordinate (r11)
				(180:7.5) coordinate (r12)
				(160:7.5) coordinate (r13)
				(150:7.5) coordinate (r14)
				(140:7.5) coordinate (r15)
				(120:7.5) coordinate (r16)
				(80:7.5) coordinate (b1)
				(55:7.5) coordinate (b2)
				(20:7.5) coordinate (b3)
				(-10:7.5) coordinate (b4)
				(-30:7.5) coordinate (b5)
				(-55:7.5) coordinate (b6)
				(-80:7.5) coordinate (b7)
				(-115:7.5) coordinate (b8)
				(-142:7.5) coordinate (b9)
				(-155:7.5) coordinate (b10)
				(-170:7.5) coordinate (b11)
				(170:7.5) coordinate (b12)
				(155:7.5) coordinate (b13)
				(145:7.5) coordinate (b14)
				(130:7.5) coordinate (b15)
				(105:7.5) coordinate (b16)

				(90:4.5) coordinate (bb1)
				(73:4.25) coordinate (bb2)
				(58:3.6) coordinate (bb3)
				(44:2.6) coordinate (bb4)
				(40:1.1) coordinate (bb5)
				(142:1.3) coordinate (bb6)
				(129:3.2) coordinate (bb7)
				(117:3.9) coordinate (bb8)
				(105:4.3) coordinate (bb9)
				
				(-53:2) coordinate (bbb1)
				(-70:3.6) coordinate (bbb2)
				(-125:2.5) coordinate (bbb3)
				
				(81:4.4) coordinate (rr1)
				(66:4.05) coordinate (rr2)
				(50:3.1) coordinate (rr3)
				(38:1.9) coordinate (rr4)
				(95:0.5) coordinate (rr5)
				(138:2.3) coordinate (rr6)
				(123:3.6) coordinate (rr7)
				(111:4.1) coordinate (rr8)
				(98:4.45) coordinate (rr9)
				(-65:1.2) coordinate (rrr1)
				(-58:2.8) coordinate (rrr2)
				(-83:3.95) coordinate (rrr3)
				;

				\draw[cyan,line width=1pt,red] (rr8) to[out=110,in=90](rr9);
				\draw[cyan,line width=1pt,red] (rr8) to[out=90,in=-130](r1);
				\draw[cyan,line width=1pt,red] (r2) to[out=180,in=-30](r1);
				\draw[cyan,line width=1pt,red] (r16) to[out=0,in=-150](r1);
				\draw[cyan,line width=1pt,red] (r1) to[out=-35,in=50](rr3);
				\draw[cyan,line width=1pt,red] (rr2) to[out=0,in=50](rr3);
				\draw[cyan,line width=1pt,red] (r5) to[out=120,in=-135](r4);
				\draw[cyan,line width=1pt,red] (r5) to[out=180,in=90](r6);
				\draw[cyan,line width=1pt,red] (rr5) to[out=-170,in=130](r8);
				\draw[cyan,line width=1pt,red] (rrr1) to[out=30,in=30](rrr2);
				\draw[cyan,line width=1pt,red] (rrr2) to[out=-30,in=-30](rrr3);
				\draw[cyan,line width=1pt,red] (rrr2) to[out=-20,in=90](r8);
				\draw[cyan,line width=1pt,red] (r8) to[out=20,in=160](r7);
				\draw[cyan,line width=1pt,red] (rr6) to[out=-135,in=160](r8);
				\draw[cyan,line width=1pt,red] (rr6) to[out=-145,in=90](r9);
				\draw[cyan,line width=1pt,red] (rr6) to[out=-155,in=50](r11);
				\draw[cyan,line width=1pt,red] (r10) to[out=90,in=0](r11);
				\draw[cyan,line width=1pt,red] (r12) to[out=-30,in=50](r11);
				\draw[cyan,line width=1pt,red] (r14) to[out=-40,in=0](r13);
				\draw[cyan,line width=1pt,red] (r14) to[out=0,in=-70](r15);
				\draw[cyan,line width=1pt,red] (rr6) to[out=140,in=-150](rr7);

				\draw [red, line width=1pt] (r1) to (rr1);
				\draw [red, line width=1pt] (rr3) to (r3);
				\draw [red, line width=1pt] (rr3) to (r5);
				\draw [red, line width=1pt] (rrr1) to (rr5);
				\draw [red, line width=1pt] (r14) to (rr6);

				\draw[red,line width=1pt]plot [smooth,tension=0.6] coordinates {(rr3) (3.5,-1)(3,-5) (r8)};
				\draw[red,line width=1pt]plot [smooth,tension=0.6] coordinates {(rr4) (3.2,-1)(2.2,-5) (r8)};
				\draw[red,line width=1pt]plot [smooth,tension=0.6] coordinates {(rr6) (-3.2,3.3) (-1.8,6)  (r1)};
				
				\draw[black,line width=1.5pt]plot [smooth,tension=0.6] coordinates {(b9) (-5.2,0) (-3.5,3.5) (0,5.7)  (3,4) (5.5,0)(4,-3)(0,-6)(-2.5,-4)(-2.4,-1) (bb6)};

				\draw[red] (-1,-5) node {{ $x_{0}$}};
				\draw[red] (3.4,-5) node {{$x_{1}$}};
				\draw[red] (2.5,5.2) node {{ $x_{2}$}};
				\draw[red] (-2.8,5.2) node {{ $x_{3}$}};
				\draw[red] (-4.5,0.3) node {{ $x_{4}$}};
				
				\draw[red] (-3,-5) node {{ ${z}$}};
				
				\draw[red] (0.7,-5) node {{ $y_{1}$}};
				\draw[red] (2.8,-1) node {{ $y_{2}$}};
				\draw[red] (4.9,-1) node {{ $y_{3}$}};
				\draw[red] (4.3,4.3) node {{ $y_{4}$}};
				\draw[red] (0.2,5.2) node {{ $y_{5}$}};
				\draw[red] (-1.5,5.5) node {{ $y_{6}$}};
				\draw[red] (-5,2.7) node {{ $y_{7}$}};
				\draw[] (5.9,0) node {{ $\zg$}};

				\draw[black,thick,bend right,->,>=stealth] (-1,-6) to node[above ]{\tiny$\mathbf{a_{1}}$} (-1.6,-6.5);
				\draw[black,thick,bend right,->,>=stealth] (0.1,-6.9) to node[above ]{\tiny$\mathbf{c_2}$} (-0.4,-7);
				\draw[black,thick,bend right,->,>=stealth] (0.7,-6.9) to node[above right]{\tiny$\mathbf{c_{4}}$} (0.2,-6.4);
				\draw[black,thick,bend right,->,>=stealth] (2,-6.1) to node[above right]{\tiny$\mathbf{b_{1}}$} (1.7,-5.7);
				\draw[black,thick,bend right,->,>=stealth] (2,-6.9) to node[ right]{\tiny$\mathbf{a_{2}}$} (1.7,-6.3);
				\draw[black,thick,bend right,->,>=stealth] (6.8,-1.5) to node[above left]{\tiny$\mathbf{c_{8}}$} (6.5,-2);
				\draw[black,thick,bend right,->,>=stealth] (6.5,-2) to node[ left]{\tiny$\mathbf{c_{7}}$} (6.5,-2.7);
				\draw[black,thick,bend right,->,>=stealth] (1.2,6.5) to node[ right]{\tiny$\mathbf{a_{3}}$} (1.5,7.1);
				\draw[black,thick,bend right,->,>=stealth] (0.15,6.8) to node[ below right]{\tiny$\mathbf{b_{4}}$} (0.7,7);
				\draw[black,thick,bend right,->,>=stealth] (-0.5,6.8) to node[ below ]{\tiny$\mathbf{c_{11}}$} (0.2,6.8);
				\draw[black,thick,bend right,->,>=stealth] (-1.3,6.5) to node[ below left  ]{\tiny$\mathbf{b_{5}}$} (-0.8,6.3);
				\draw[black,thick,bend right,->,>=stealth] (-1.8,6.7) to node[ below left  ]{\tiny$\mathbf{a_{6}}$} (-1.5,6.3);
				\draw[black,thick,bend right,->,>=stealth] (-6,3.5) to node[ right  ]{\tiny$\mathbf{c_{12}}$} (-5.8,4.2);
				\draw[black,thick,bend right,->,>=stealth] (-6.45,3) to node[ below right  ]{\tiny$\mathbf{c_{13}}$} (-6,3.5);
				\draw[black,thick,bend right,->,>=stealth] (-6.2,-1.6) to node[ above ]{\tiny$\mathbf{c_{15}}$} (-6.7,-1.3);
				\draw[black,thick,bend right,->,>=stealth] (-6.7,-2.7) to node[right ]{\tiny$\mathbf{c_{16}}$} (-6.7,-2.2);

				\draw[black,thick,bend right,->,>=stealth] (1,-1.1) to node[above  ]{\tiny$\mathbf{c_{6}}$} (0.4,-0.8);
				\draw[black,thick,bend right,->,>=stealth] (1.9,-2.6) to node[right  ]{\tiny$\mathbf{c_{5}}$} (1.65,-2);
				\draw[black,thick,bend right,->,>=stealth] (1.6,-3.2) to node[below  ]{\tiny$\mathbf{c_{1}}$} (2.1,-3);

				\draw[black,thick,bend right,->,>=stealth] (-0.7,0.3) to node[below  ]{\tiny$\mathbf{c_{3}}$} (0.1,0.1);
				\draw[black,thick,bend right,->,>=stealth] (2.6,1.3) to node[below right  ]{\tiny$\mathbf{b_{2}}$} (3,1.4);
				\draw[black,thick,bend right,->,>=stealth] (2.6,1.8) to node[ right  ]{\tiny$\mathbf{c_{9}}$} (2.6,2.8);
				\draw[black,thick,bend right,->,>=stealth] (2.6,2.8) to node[above right  ]{\tiny$\mathbf{b_{3}}$} (2.4,3.2);
				\draw[black,thick,bend right,->,>=stealth] (2.5,3.6) to node[above  ]{\tiny$\mathbf{a_{4}}$} (2,3.5);
				\draw[black,thick,bend right,->,>=stealth] (-0.9,4.5) to node[above right ]{\tiny$\mathbf{c_{10}}$} (-1.4,4.5);
				\draw[black,thick,bend right,->,>=stealth] (-2.25,2.5) to node[above ]{\tiny$\mathbf{a_{5}}$} (-2.8,2.5);
				\draw[black,thick,bend right,->,>=stealth] (-2.8,2.5) to node[above left ]{\tiny$\mathbf{b_{6}}$} (-3.1,2.2);
				\draw[black,thick,bend right,->,>=stealth] (-2.4,1.85) to node[ left ]{\tiny$\mathbf{c_{14}}$} (-2.4,1.2);
				\draw[black,thick,bend right,->,>=stealth] (-3,0.9) to node[below left ]{\tiny$\mathbf{a_{7}}$} (-2.78,0.55);
				\draw[black,thick,bend right,->,>=stealth] (-3.5,-0.3)to node[below ]{\tiny$\mathbf{c_{17}}$} (-2.9,-0.6);
				
\draw[thick,fill=white] 
(b1) circle (0.08cm)
(b2) circle (0.08cm)
(b3) circle (0.08cm)
(b4) circle (0.08cm)
(b5) circle (0.08cm)
(b6) circle (0.08cm)
(b7) circle (0.08cm)
(b8) circle (0.08cm)
(b9) circle (0.08cm)
(b10) circle (0.08cm)
(b11) circle (0.08cm)
(b12) circle (0.08cm)
(b13) circle (0.08cm)
(b14) circle (0.08cm)
(b15) circle (0.08cm)
(b16) circle (0.08cm)
(bb1) circle (0.08cm)
(bb2) circle (0.08cm)
(bb3) circle (0.08cm)
(bb4) circle (0.08cm)
(bb5) circle (0.08cm)
(bb6) circle (0.08cm)
(bb7) circle (0.08cm)
(bb8) circle (0.08cm)
(bb9) circle (0.08cm)
(bbb1) circle (0.08cm)
(bbb2) circle (0.08cm)
(bbb3) circle (0.08cm);	

\draw[thick,red, fill=red] 
(r1) circle (0.08cm)
(r2) circle (0.08cm)
(r3) circle (0.08cm)
(r4) circle (0.08cm)
(r5) circle (0.08cm)
(r6) circle (0.08cm)
(r7) circle (0.08cm)
(r8) circle (0.08cm)
(r9) circle (0.08cm)
(r10) circle (0.08cm)
(r11) circle (0.08cm)
(r12) circle (0.08cm)
(r13) circle (0.08cm)
(r14) circle (0.08cm)
(r15) circle (0.08cm)
(r16) circle (0.08cm)
(rr1) circle (0.08cm)
(rr2) circle (0.08cm)
(rr3) circle (0.08cm)
(rr4) circle (0.08cm)
(rr5) circle (0.08cm)
(rr6) circle (0.08cm)
(rr7) circle (0.08cm)
(rr8) circle (0.08cm)
(rr9) circle (0.08cm)
(rrr1) circle (0.08cm)
(rrr2) circle (0.08cm)
(rrr3) circle (0.08cm)
;
			
\end{tikzpicture}
\end{center}
\caption{A marked surface $(\cals,\calm,\zD^*)$ with a simple zigzag $\bpoint$-arc $\zg$.} 
	\label{fig:apd1} 
\end{figure}

\begin{figure}
\begin{center}

\begin{tikzpicture}[>=stealth, scale=1]
				\path 
				(-7,-1) coordinate (1)
				(-6,0) coordinate (2)
				(-6,2)  coordinate (3)
				(-5,-1)  coordinate (4)
				(-5,1) coordinate (5)
				(-4,0) coordinate (6)
			    (-4,2) coordinate (7)
			    (-3,1) coordinate (8)
			    (-3,3) coordinate (9)
			    (-2,2)  coordinate (10)
			    (-2,4)  coordinate (11)
			   (-1,-1) coordinate (12)
			   (-1,1) coordinate (13)
			   (-1,3) coordinate (14)
			   (0,0) coordinate (15)
			   (1,-1) coordinate (16)
			   (1,1)  coordinate (17)
			   (1,3)  coordinate (18)
			   (2,2) coordinate (19)
			   (2,4) coordinate (20)
			   (3,1) coordinate (21)
			   (3,3) coordinate (22)
			   (4,0)  coordinate (23)
			   (4,2)  coordinate (24)
			   (4,4) coordinate (25)
			   (5,1) coordinate (26)
			   (5,3) coordinate (27)
			   (6,0) coordinate (28)
			   (6,2) coordinate (29)
			   (7,-1) coordinate (30)
			   
			   ;

			   \draw 
			   (1) node {$z$}
			   (2) node {$x_0$}
			   (3) node {$\bullet$}
			   (4) node {$\bullet$}
			   (5) node {$y_1$}
			   (6) node {$\bullet$}
			   (7) node {$y_2$}
			   (8) node {$\bullet$}
			   (9) node {$x_1$}
			   (10) node {$y_3$}
			   (11) node {$\bullet$}
			   (12) node {$\bullet$}
			   (13) node {$y_4$}
			   (14) node {$\bullet$}
			   (15) node {$x_2$}
			   (16) node {$\bullet$}
			   (17) node {$y_5$}
			   (18) node {$\bullet$}
			   (19) node {$y_6$}
			   (20) node {$\bullet$}
			   (21) node {$\bullet$}
			   (22) node {$x_3$}
			   (23) node {$\bullet$}
			   (24) node {$y_7$}
			   (25) node {$\bullet$}
			   (26) node {$x_4$}
			   (27) node {$\bullet$}
			   (28) node {$\bullet$}
			   (29) node {$\bullet$}
			   (30) node {$z$};

			   \draw [thick,->] (-6.1,-0.1) to node[above left ]{$\aaa_1$}(-6.9,-0.9);
			   \draw [thick,->] (-5.2,0.8) to node[above left ]{$\ccc_2$}(-5.85,0.15);
			   \draw [thick,->] (-5.8,-0.2) to node[ right ]{$\ccc_3$}(-5.1,-0.9);
			   \draw [thick,->] (-4.1,-0.1) to node[ right ]{$\ccc_6$}(-4.9,-0.9);
			   \draw [thick,->] (-4.8,0.8) to node[ left ]{$\ccc_5$}(-4.1,0.1);
			   \draw [thick,->] (-5.9,1.9) to node[ left ]{$\ccc_1$}(-5.2,1.2);
			   \draw [thick,->] (-4.2,1.8) to node[above left ]{$\ccc_4$}(-4.8,1.2);
			   \draw [thick,->] (-3.2,2.8) to node[ above left ]{$\bbb_1$}(-3.8,2.2);
			   \draw [thick,->] (-2.2,3.8) to node[ above left ]{$\aaa_2$}(-2.8,3.2);
			   \draw [thick,->] (-2.8,2.8) to node[ above right ]{$\bbb_2$}(-2.2,2.2);
			   \draw [thick,->] (-2.2,1.8) to node[ above left ]{$\ccc_7$}(-2.8,1.2);
			   \draw [thick,->] (-1.2,2.8) to node[ below right ]{$\ccc_8$}(-1.8,2.2);
		    	\draw [thick,->] (-1.8,1.8) to node[ below left ]{$\ccc_9$}(-1.2,1.2);
				\draw [thick,->] (-0.8,0.8) to node[ below left ]{$\bbb_3$}(-0.2,0.2);
			    \draw [thick,->] (-0.2,-0.2) to node[ below right ]{$\aaa_3$}(-0.8,-0.8);
				\draw [thick,->] (0.2,-0.2) to node[ above right ]{$\aaa_4$}(0.8,-0.8);
				\draw [thick,->] (0.8,0.8) to node[ above left ]{$\bbb_4$}(0.2, 0.2);
				\draw [thick,->] (1.8,1.8) to node[ above left ]{$\ccc_{11}$}(1.2, 1.2);
			    \draw [thick,->] (2.8,2.8) to node[ above left ]{$\bbb_{5}$}(2.2, 2.2);
				\draw [thick,->] (3.8,3.8) to node[ above left ]{$\aaa_{6}$}(3.2, 3.2);
				\draw [thick,->] (1.2,2.8) to node[left ]{$\ccc_{10}$}(1.8, 2.2);
				\draw [thick,->] (2.2,3.8) to node[left ]{$\aaa_{5}$}(2.8, 3.2);
				\draw [thick,->] (3.2,2.8) to node[above right]{$\bbb_{6}$}(3.8, 2.2);
				\draw [thick,->] (3.8,1.8) to node[above left]{$\ccc_{12}$}(3.2, 1.2);
				\draw [thick,->] (4.8,2.8) to node[below right]{$\ccc_{13}$}(4.2, 2.2);
				\draw [thick,->] (4.2,1.8) to node[below left]{$\ccc_{14}$}(4.8, 1.2);
				\draw [thick,->] (5.8, 1.8) to node[below right]{$\ccc_{16}$}(5.2,1.2);
				\draw [thick,->] (4.8, 0.8) to node[below right]{$\ccc_{15}$}(4.2,0.2);
				\draw [thick,->] (5.2, 0.8) to node[right]{$\aaa_{7}$}(5.8,0.2);
				\draw [thick,->] (6.2, -0.2) to node[above right]{$\ccc_{17}$}(6.8,-0.8);
			
\end{tikzpicture}
\end{center}
\caption{The gentle algebra associated with the marked surface $(\cals,\calm,\zD^*)$ given in Figure \ref{fig:apd1}, where the composition of two arrows belongs to the ideal if and only if they are not on a straight line, that is there is a corner at the common vertex of the arrows.
The walk associated with $\zg$ is $\omega=(\ccc_2^{-1}\ccc_4^{-1}\bbb_1^{-1})(\bbb_2
\ccc_9\bbb_3)(\bbb^{-1}_4\ccc_{11}^{-1}\bbb_5^{-1})(\bbb_6\ccc_{14})$. One can construct the algebra $A_\zg$ by the general rule described above, which is depicted in Figure \ref{fig:apd4}.} 
	\label{fig:apd2} 
	\end{figure}

\begin{figure}
	\begin{center}
		\begin{tikzpicture}[>=stealth,scale=0.8]
			\draw[line width=1.8pt,fill=white] (0,0) circle (7.5cm);
			\draw[line width=1.8pt,fill=gray!30] (3.8,-4.3) circle (0.7cm);

			\path 
			(95:7.5) coordinate (r1)
			(85:7.5) coordinate (r2)
			(75:7.5) coordinate (r3)
			(65:7.5) coordinate (r4)
			(55:7.5) coordinate (r5)
			(40:7.5) coordinate (r6)
			(25:7.5) coordinate (r7)
			(5:7.5) coordinate (r8)
			(-60:7.5) coordinate (r9)
			(-90:7.5) coordinate (r10)
			(-110:7.5) coordinate (r11)
			(-125:7.5) coordinate (r12)
			(-132:7.5) coordinate (r13)
			(-140:7.5) coordinate (r14)
			(-152:7.5) coordinate (r15)
			(-160:7.5) coordinate (r16)
			(-172:7.5) coordinate (r17)
			(180:7.5) coordinate (r18)
			(172:7.5) coordinate (r19)
			(164:7.5) coordinate (r20)
			(156:7.5) coordinate (r21)
			(148:7.5) coordinate (r22)
			(140:7.5) coordinate (r23)
			(132:7.5) coordinate (r24)
			(124:7.5) coordinate (r25)
			(116:7.5) coordinate (r26)
			(108:7.5) coordinate (r27)
			(-43:5.3) coordinate (r28)
			(-43.5:6.2) coordinate (r29)
			(-55.5:5.8) coordinate (r30)

			(90:7.5) coordinate (b1)
			(80:7.5) coordinate (b2)
			(70:7.5) coordinate (b3)
			(60:7.5) coordinate (b4)
			(48:7.5) coordinate (b5)
			(32:7.5) coordinate (b6)
			(15:7.5) coordinate (b7)
			(-30:7.5) coordinate (b8)
			(-75:7.5) coordinate (b9)
			(-100:7.5) coordinate (b10)
			(-117:7.5) coordinate (b11)
			(-125:7.5) coordinate (b12)
			(-128.5:7.5) coordinate (b13)
			(-136:7.5) coordinate (b14)
			(-156:7.5) coordinate (b15)
			(-166:7.5) coordinate (b16)
			(-176:7.5) coordinate (b17)
			(176:7.5) coordinate (b18)
			(168:7.5) coordinate (b19)
			(160:7.5) coordinate (b20)
			(152:7.5) coordinate (b21)
			(144:7.5) coordinate (b22)
			(136:7.5) coordinate (b23)
			(128:7.5) coordinate (b24)
			(120:7.5) coordinate (b25)
			(112:7.5) coordinate (b26)
			(102:7.5) coordinate (b27)
			(-41.5:5.8) coordinate (b28)
			(-50:6.42) coordinate (b29)
			(-52:5.14) coordinate (b30)
			;
	
			\draw[thick,black, fill=white]
			(b1) circle (0.10cm)
			(b2) circle (.10cm)
			(b3) circle (.10cm)
			(b4) circle (.10cm)
			(b5) circle (.10cm)
			(b6) circle (0.10cm)
			(b7) circle (.10cm)
			(b8) circle (.10cm)
			(b9) circle (.10cm)
			(b10) circle (.10cm)
			(b11) circle (0.10cm)
			(b12) circle (.10cm)
			(b13) circle (.10cm)
			(b14) circle (.10cm)
			(b15) circle (.10cm)
			(b16) circle (.10cm)
			(b17) circle (0.10cm)
			(b18) circle (.10cm)
			(b19) circle (.10cm)
			(b20) circle (.10cm)
			(b21) circle (.10cm)
			(b22) circle (0.10cm)
			(b23) circle (.10cm)
			(b24) circle (.10cm)
			(b25) circle (.10cm)
			(b26) circle (.10cm)
			(b27) circle (0.10cm)
			(b28) circle (.10cm)
			(b29) circle (.10cm)
			(b30) circle (0.10cm);

			\draw[thick,red, fill=red] 
			(r1) circle (0.10cm)
			(r2) circle (.10cm)
			(r3) circle (.10cm)
			(r4) circle (.10cm)
			(r5) circle (.10cm)
			(r6) circle (0.10cm)
			(r7) circle (.10cm)
			(r8) circle (.10cm)
			(r9) circle (.10cm)
			(r10) circle (.10cm)
			(r11) circle (0.10cm)
			(r12) circle (.10cm)
			(r13) circle (.10cm)
			(r14) circle (.10cm)
			(r15) circle (.10cm)
			(r16) circle (.10cm)
			(r17) circle (0.10cm)
			(r18) circle (.10cm)
			(r19) circle (.10cm)
			(r20) circle (.10cm)
			(r21) circle (.10cm)
			(r22) circle (0.10cm)
			(r23) circle (.10cm)
			(r24) circle (.10cm)
			(r25) circle (.10cm)
			(r26) circle (.10cm)
			(r27) circle (0.10cm)
			(r28) circle (.10cm)
			(r29) circle (.10cm)
			(r30) circle (0.10cm);

			\draw[cyan,line width=1.5pt,red] (r2) to[out=-60,in=-150](r3)to[out=-60,in=-160](r4);
			\draw[cyan,line width=1.5pt,red] (r5) to[out=-90,in=180](r6);
			\draw[cyan,line width=1.5pt,red] (r11) to[out=100,in=10](r12);
			\draw[cyan,line width=1.5pt,red] (r14) to[out=80,in=-10](r15);
			\draw[cyan,line width=1.5pt,red] (r16) to[out=50,in=-20](r17);
			\draw[cyan,line width=1.5pt,red] (r19) to[out=30,in=-50](r20)to[out=20,in=-50](r21);
			\draw[cyan,line width=1.5pt,red] (r22) to[out=10,in=-50](r23)to[out=0,in=-60](r24);
			\draw[cyan,line width=1.5pt,red] (r25) to[out=-20,in=-90](r26)to[out=-50,in=-100](r27);
			\draw[cyan,line width=1.5pt,red] (r29) to[out=30,in=-20](-38.5:5.8)to[out=160,in=80](r28)to[out=160,in=80](-55:4.8)to[out=-100,in=-160](r30);

			\draw [red, line width=1.5pt] (r8) to (r9);
			\draw [red, line width=1.5pt] (r9) to (r30);
			\draw [red, line width=1.5pt] (r8) to (r28);

			\draw[cyan,line width=1.5pt,red] (r1) to[out=-90,in=3](r18);
			\draw[cyan,line width=1.5pt,red] (r26) to[out=-60,in=5](r18);
			\draw[cyan,line width=1.5pt,red] (r23) to[out=-30,in=10](r18);
			\draw[cyan,line width=1.5pt,red] (r20) to[out=0,in=10](r18);
			\draw[cyan,line width=1.5pt,red] (r3) to[out=-115,in=2](r18);
			\draw[cyan,line width=1.5pt,red] (r6) to[out=-150,in=0](r18);
			\draw[cyan,line width=1.5pt,red] (r6) to[out=-132,in=10](r17);
			\draw[cyan,line width=1.5pt,red] (r7) to[out=-145,in=5](r17);
			\draw[cyan,line width=1.5pt,red] (r8) to[out=-170,in=0](r17);
			\draw[cyan,line width=1.5pt,red] (r8) to[out=-160,in=30](r15);
			\draw[cyan,line width=1.5pt,red] (r8) to[out=-155,in=50](r13);
			\draw[cyan,line width=1.5pt,red] (r8) to[out=-145,in=70](r11);
			\draw[cyan,line width=1.5pt,red] (r8) to[out=-140,in=80](r10);

			\draw[red] (-6.3,1.5) node {{ $x_{4}$}};
			\draw[red] (-4.6,3) node {{ $y_{7}$}};
			\draw[red] (-3.5,3.5) node {{ $\widehat{x}_{3}$}};
			\draw[red] (-1.9,3.8) node {{ $y_{4}$}};
			\draw[red] (-0.3,4) node {{ $y_{3}$}};
			\draw[red] (1.5,2.8) node {{ $\widehat{x}_{1}$}};
			\draw[red] (-0.1,1) node {{ $z$}};
			\draw[red] (2,0) node {{ $\widehat{x}_{4}$}};
			\draw[red] (-2.8,-1.5) node {{ $y_{6}$}};
			\draw[red] (-2.8,-3.1) node {{ $y_{5}$}};
			\draw[red] (-1.5,-4.5) node {{ $\widehat{x}_{2}$}};
			\draw[red] (0.1,-5.8) node {{ $y_{2}$}};
			\draw[red] (4.7,-2.2) node {{ $y_{1}$}};
			\draw[red] (6,-3.2) node {{ $x_{0}$}};	
			\draw[black,thick,bend right,->,>=stealth] (-7.2,1.3) to node[below  ]{\tiny$\mathbf{c_{16}}$} (-6.8,1.1);
			\draw[black,thick,bend right,->,>=stealth] (-6.8,1.8) to node[right]{\tiny$\mathbf{c_{15}}$} (-6.8,2.4);
			\draw[black,thick,bend right,->,>=stealth] (-5.8,1.1) to node[below]{\tiny$\mathbf{c_{14}}$} (-6.8,1);
			\draw[black,thick,bend right,->,>=stealth] (-5,1)to node[above right]{\tiny$\mathbf{\widehat{b}_{6}}$} (-5.8,1.1) ;
			\draw[black,thick,bend right,->,>=stealth] (-3,1.9)to node[above right]{\tiny$\mathbf{\widehat{b}_{3}}$} (-4,2) ;
			\draw[black,thick,bend right,->,>=stealth] (-2,1.9)to node[above right]{\tiny$\mathbf{c_{9}}$} (-2.8,2.1) ;
			\draw[black,thick,bend right,->,>=stealth] (-2.3,0.75)to node[right]{\tiny$\mathbf{\widehat{b}_{2}}$} (-3,1.2) ;
			\draw[black,thick,bend right,->,>=stealth] (-6,4.1)to node[below]{\tiny$\mathbf{c_{13}}$} (-5.3,4.4) ;
			\draw[black,thick,bend right,->,>=stealth] (-5.3,4.5)to node[right]{\tiny$\mathbf{c_{12}}$} (-5.2,5) ;
			\draw[black,thick,bend right,->,>=stealth] (-3.6,6.25)to node[below]{\tiny$\mathbf{\widehat{a}_{6}}$} (-3.05,6.25) ;
			\draw[black,thick,bend right,->,>=stealth](-3,6.25) to node[below right]{\tiny$\mathbf{\widehat{a}_{3}}$} (-2.6,6.7) ;
			\draw[black,thick,bend right,->,>=stealth] (1.1,7.1)to node[left]{\tiny$\mathbf{c_{8}}$} (1.6,6.5) ;
			\draw[black,thick,bend right,->,>=stealth] (1.65,6.5)to node[below]{\tiny$\mathbf{c_{7}}$} (2.5,6.75) ;
			\draw[black,thick,bend right,->,>=stealth](5,5) to node[left]{\tiny$\mathbf{\widehat{a}_{2}}$} (4.8,4.3) ;
			\draw[black,thick,bend right,->,>=stealth](4,3.8) to node[below left]{\tiny$\mathbf{\widehat{a}_{1}}$} (4.3,3.4) ;
			\draw[black,thick,bend right,->,>=stealth] (4.7,-3.7)to node[right]{\tiny$\mathbf{c_{1}}$} (4.5,-3) ;
			\draw[black,thick,bend right,->,>=stealth] (4.35,-3)to node[above]{\tiny$\mathbf{c_{5}}$} (3.5,-3.5) ;
			\draw[black,thick,bend right,->,>=stealth] (2.8,-4.5)to node[left]{\tiny$\mathbf{c_{6}}$} (3.5,-5.5) ;
			\draw[black,thick,bend right,->,>=stealth](4.2,-5.7) to node[above]{\tiny$\mathbf{c_{3}}$}  (3.55,-5.8);
			\draw[black,thick,bend right,->,>=stealth](-6.9,-1.5) to node[below right]{\tiny$\mathbf{\widehat{a}_{5}}$}  (-6.5,-1.05);
			\draw[black,thick,bend right,->,>=stealth](-6,-3.9) to node[ right]{\tiny$\mathbf{c_{10}}$}  (-6,-3.2);
			\draw[black,thick,bend right,->,>=stealth](-2.2,-6.2) to node[above]{\tiny$\mathbf{\widehat{a}_{4}}$}  (-3.1,-6.3);
			\draw[black,thick,bend right,->,>=stealth](-1,-0.2) to node[ right]{\tiny$\mathbf{c_{17}}$}  (-1,0.4);
			\draw[black,thick,bend right,->,>=stealth](-0.5,-0.65) to node[ right]{\tiny$\mathbf{\widehat{a}_{4}}$}  (-0.5,-0.1);
			\draw[black,thick,bend right,->,>=stealth](-0.58,-0.69) to node[ left]{\tiny$\mathbf{\widehat{b}_{5}}$}  (-0.5,-1.3);
			\draw[black,thick,bend right,->,>=stealth](-0.1,-1.2) to node[left]{\tiny$\mathbf{c_{11}}$}  (-0.1,-2.1);
			\draw[black,thick,bend right,->,>=stealth](1,-1.65) to node[below  left]{\tiny$\mathbf{\widehat{b}_{4}}$}  (1.2,-2.7);
			\draw[black,thick,bend right,->,>=stealth](1.5,-2.6) to node[below  left]{\tiny$\mathbf{\widehat{b}_{1}}$}  (2,-3.6);
			\draw[black,thick,bend right,->,>=stealth](5.5,-1.7) to node[below]{\tiny$\mathbf{c_{2}}$}  (6,-2.1);
			\draw[black,thick,bend right,->,>=stealth](5.4,-0.9) to node[below left]{\tiny$\mathbf{c_{4}}$}  (5.6,-1.5);		
			
	\end{tikzpicture}	
	\end{center}
 \caption{The cutting surface $(\cals_\zg,\calm_\zg,\zD^*_\zg)$ obtained by cutting $(\cals,\calm,\zD^*)$ along $\zg$ in Figure \ref{fig:apd1}.} 
	\label{fig:apd3} 
\end{figure}

\begin{figure}
	\begin{center}
	\begin{tikzpicture}[>=stealth, scale=.9]
			
			\path 
			(-7,0) coordinate (1)
			(-8,1) coordinate (2)
			(-6,1)  coordinate (3)
			(-7,2)  coordinate (4)
			(-8,3) coordinate (5)
			(-6,3) coordinate (6)
			(-4,3) coordinate (7)
			(-5,4) coordinate (8)
			(-4,5) coordinate (9)
			(-3,6)  coordinate (10)
			(-4,7)  coordinate (11)
			(-2,7) coordinate (12)
			(-3,8) coordinate (13)
			(-1,6) coordinate (14)
			(0,5) coordinate (15)
			(1,6) coordinate (16)
			(2,7)  coordinate (17)
			(2,5)  coordinate (18)
			(3,6) coordinate (19)
			(1,4) coordinate (20)
			(3,4) coordinate (21)
			(5,4) coordinate (22)
			(4,3)  coordinate (23)
			(6,3)  coordinate (24)
			(3,2) coordinate (25)
			(5,2) coordinate (26)
			(7,2) coordinate (27)
			(4,1) coordinate (28)
			(6,1) coordinate (29)
			(5,0) coordinate (30);

			\draw 
			(1) node {$\bullet$}
			(2) node {$x_0$}
			(3) node {$\bullet$}
			(4) node {$y_1$}
			(5) node {$\bullet$}
			(6) node {$y_2$}
			(7) node {$\bullet$}
			(8) node {$\widehat{x}_2$}
			(9) node {$y_5$}
			(10) node {$y_6$}
			(11) node {$\bullet$}
			(12) node {$\widehat{x}_4$}
			(13) node {$\bullet$}
			(14) node {$\bullet$}
			(15) node {$z$}
			(16) node {$\widehat{x}_1$}
			(17) node {$\bullet$}
			(18) node {$y_3$}
			(19) node {$\bullet$}
			(20) node {$\bullet$}
			(21) node {$y_4$}
			(22) node {$\bullet$}
			(23) node {$\widehat{x}_3$}
			(24) node {$\bullet$}
			(25) node {$\bullet$}
			(26) node {$y_7$}
			(27) node {$\bullet$}
			(28) node {$\bullet$}
			(29) node {$x_4$}
			(30) node {$\bullet$};
            \draw (4.9,2.7) node {$\widehat{\bbb}_6$};

            \draw [thick,->] (-7.8,0.8) to node[below left]{$\ccc_{3}$}(-7.2,0.2);
			\draw [thick,->] (-7.2,1.8) to node[above left ]{$\ccc_2$}(-7.8,1.2);
			\draw [thick,->] (-6.8,1.8) to node[above right]{$\ccc_5$}(-6.2,1.2);
			\draw [thick,->] (-6.2,0.8) to node[below right ]{$\ccc_6$}(-6.8,0.2);
			\draw [thick,->] (-7.8,2.8) to node[above right ]{$\ccc_1$}(-7.2,2.2);
			\draw [thick,->] (-6.2,2.8) to node[below right ]{$\ccc_4$}(-6.8,2.2);
			\draw [thick,->] (-5.2,3.8) to node[ above left ]{$\widehat{\bbb}_1$}(-5.8,3.2);
			\draw [thick,->] (-4.8,3.8) to node[above right]{$\widehat{\aaa}_4$}(-4.2,3.2);
			\draw [thick,->] (-4.2,4.8) to node[ above left ]{$\widehat{\bbb}_4$}(-4.8,4.2);
			\draw [thick,->] (-3.2,5.8) to node[ above left ]{$\ccc_{11}$}(-3.8,5.2);
			\draw [thick,->] (-2.2,6.8) to node[ below right ]{$\widehat{\bbb}_5$}(-2.8,6.2);
			\draw [thick,->] (-3.8,6.8) to node[ above right ]{$\ccc_{10}$}(-3.2,6.2);
		    \draw [thick,->] (-2.8,7.8) to node[ above right ]{$\widehat{\aaa}_5$}(-2.2,7.2);
			\draw [thick,->] (-1.8,6.8) to node[ above right ]{$\widehat{\aaa}_7$}(-1.2,6.2);
			\draw [thick,->] (-0.8,5.8) to node[ below left ]{$\ccc_{17}$}(-0.2,5.2);
			\draw [thick,->] (0.8,5.8) to node[ above left ]{$\widehat{\aaa}_1$}(0.2,5.2);
			\draw [thick,->] (1.8,6.8) to node[ above left ]{$\widehat{\aaa}_2$}(1.2,6.2);
			\draw [thick,->] (1.8,4.8) to node[ above left ]{$\ccc_7$}(1.2,4.2);
			\draw [thick,->] (2.8,5.8) to node[below right ]{$\ccc_8$}(2.2,5.2);
			\draw [thick,->] (3.8,2.8) to node[ above left ]{$\widehat{\aaa}_3$}(3.2,2.2);
			\draw [thick,->] (4.8,3.8) to node[below right ]{$\widehat{\aaa}_6$}(4.2,3.2);
			\draw [thick,->] (4.8,1.8) to node[ above left ]{$\ccc_{12}$}(4.2,1.2);
			\draw [thick,->] (5.8,2.8) to node[below right ]{$\ccc_{13}$}(5.2,2.2);
			\draw [thick,->] (5.8,0.8) to node[ below right ]{$\ccc_{15}$}(5.2,0.2);
			\draw [thick,->] (6.8,1.8) to node[below right ]{$\ccc_{16}$}(6.2,1.2);
			\draw [thick,->] (1.2,5.8) to node[ above right ]{$\widehat{\bbb}_2$}(1.8,5.2);
			\draw [thick,->] (2.2,4.8) to node[above right  ]{$\ccc_9$}(2.8,4.2);
			\draw [thick,->] (3.2,3.8) to node[above right ]{$\widehat{\bbb}_3$}(3.8,3.2);
			\draw [thick,->] (4.2,2.8) to (4.8,2.2);
			\draw [thick,->] (5.2,1.8) to node[above right ]{$\ccc_{14}$}(5.8,1.2);
\end{tikzpicture}
\end{center}
\caption{The gentle algebra associated with the cutting surface $(\cals_\zg,\calm_\zg,\zD^*_\zg)$ given in Figure \ref{fig:apd3}, which can be directly constructed by the general rule given above.} 
	\label{fig:apd4} 
\end{figure}

\end{document}